\documentclass[11pt,amsfonts]{amsart}

\usepackage{amsmath,amssymb,amsthm,enumerate}
\usepackage[all]{xy}
\usepackage{graphicx}
\usepackage{xcolor}
\usepackage{comment}
\usepackage[T1]{fontenc}
\usepackage{graphicx}
\usepackage{calc}
\usepackage{tikz}

\SelectTips{eu}{12}

\newtheorem{thm}{Theorem}[section]
\newtheorem{prop}[thm]{Proposition}
\newtheorem{lem}[thm]{Lemma}
\newtheorem{cor}[thm]{Corollary}

\newtheorem{problem}[thm]{Problem}
\newtheorem{openproblem}[thm]{Open Problem}

\theoremstyle{definition}
\newtheorem{definition}[thm]{Definition}
\newtheorem{example}[thm]{Example}
\newtheorem{conjecture}[thm]{Conjecture}

\theoremstyle{remark}
\newtheorem{remark}[thm]{Remark}

\numberwithin{equation}{section}

\newcommand{\QQ}{\mathbb{Q}}
\newcommand{\RR}{\mathbb{R}}
\newcommand{\ZZ}{\mathbb{Z}}
\newcommand{\NN}{\mathbb{N}}

\DeclareMathOperator{\cl}{\mathrm{cl}}
\DeclareMathOperator{\scl}{\mathrm{scl}}

\newcommand{\bG}{N}
\newcommand{\hG}{G}

\newcommand{\genus}{l}

\newcommand{\II}{\mathcal{I}}

\newcommand{\CC}{\mathbb{C}}

\newcommand{\Aut}{\mathrm{Aut}}
\newcommand{\IA}{\mathrm{IA}}
\newcommand{\Ker}{\mathrm{Ker}}
\renewcommand{\Im}{\mathrm{Im}}
\newcommand{\Symp}{\mathrm{Symp}}
\newcommand{\Ham}{\mathrm{Ham}}
\newcommand{\tHam}{\widetilde{\mathrm{Ham}}}
\newcommand{\Homeo}{\mathrm{Homeo}}
\newcommand{\h}{\Homeo_+(S^1)}
\renewcommand{\th}{\widetilde{\Homeo}_+(S^1)}

\newcommand{\flux}{\mathrm{Flux}}

\newcommand{\IAA}{\mathrm{IA}}
\newcommand{\QQQ}{\mathrm{Q}}

\newcommand{\HHH}{\mathrm{H}}

\newcommand{\GL}{\mathrm{GL}}

\newcommand{\PMod}{\mathrm{PMod}}
\newcommand{\Mod}{\mathrm{Mod}}
\newcommand{\ppi}{\Gamma}
\newcommand{\tae}{\xi}

\newcommand{\qm}{\phi}

\newcommand{\trot}{\widetilde{\mathrm{rot}}}

\newcommand{\rk}{\mathrm{rk}}
\newcommand{\intrk}{\mathrm{int}\textrm{-}\mathrm{rk}}

\newcommand{\genrk}{\mathrm{gen}\textrm{-}\mathrm{rk}}

\newcommand{\VVV}{\mathrm{V}}
\newcommand{\WWW}{\mathrm{W}}

\makeatletter
\newsavebox{\@brx}
\newcommand{\llangle}[1][]{\savebox{\@brx}{\(\m@th{#1\langle}\)}%
  \mathopen{\copy\@brx\kern-0.5\wd\@brx\usebox{\@brx}}}
\newcommand{\rrangle}[1][]{\savebox{\@brx}{\(\m@th{#1\rangle}\)}%
  \mathclose{\copy\@brx\kern-0.5\wd\@brx\usebox{\@brx}}}
\makeatother
\allowdisplaybreaks[1]

\begin{document}

\title[Survey on invariant quasimorphisms and mixed scl]{Survey on invariant quasimorphisms and stable mixed commutator length}

\author[M. Kawasaki]{Morimichi Kawasaki}
\address[Morimichi Kawasaki]{Department of Mathematics, Faculty of Science, Hokkaido University, North 10, West 8, Kita-ku, Sapporo, Hokkaido 060-0810, Japan}
\email{kawasaki@math.sci.hokudai.ac.jp}

\author[M. Kimura]{Mitsuaki Kimura}
\address[Mitsuaki Kimura]{Department of Mathematics, Kyoto University, Kitashirakawa Oiwake-cho, Sakyo-ku, Kyoto 606-8502, Japan}
\email{mkimura@math.kyoto-u.ac.jp}

\author[S. Maruyama]{Shuhei Maruyama}
\address[Shuhei Maruyama]{School of Mathematics and Physics, College of Science and Engineering, Kanazawa University, Kakuma-machi, Kanazawa, Ishikawa, 920-1192, Japan}
\email{smaruyama@se.kanazawa-u.ac.jp}

\author[T. Matsushita]{Takahiro Matsushita}
\address[Takahiro Matsushita]{Department of Mathematical Sciences, Faculty of Science, Shinshu University, Matsumoto, Nagano, 390-8621, Japan}
\email{matsushita@shinshu-u.ac.jp}

\author[M. Mimura]{Masato Mimura}
\address[Masato Mimura]{Mathematical Institute, Tohoku University, 6-3, Aramaki Aza-Aoba, Aoba-ku, Sendai 9808578, Japan}
\email{m.masato.mimura.m@tohoku.ac.jp}

\makeatletter
\@namedef{subjclassname@2020}{%
\textup{2020} Mathematics Subject Classification}
\makeatother


\begin{abstract}
A homogeneous quasimorphism $\phi$ on a normal subgroup $N$ of $G$ is said to be $G$-invariant if $\phi(gxg^{-1}) = \phi(x)$ for every $g \in G$ and for every $x \in N$. Invariant quasimorphisms have naturally appeared in symplectic geometry and the extension problem of quasimorphisms. Moreover, it is known that the existence of non-extendable invariant quasimorphisms is closely related to the behavior of the stable mixed commutator length $\scl_{G,N}$, which is a certain generalization of the stable commutator length $\scl_G$.

In this survey, we review the history and recent developments of invariant quasimorphisms and stable mixed commutator length.
The topics we treat include several examples of invariant quasimorphisms, Bavard's duality theorem for invariant quasimorphisms, Aut-invariant quasimorphisms, and the estimation of the dimension of spaces of non-extendable quasimorphisms. We also mention the extension problem of partial quasimorphisms.
\end{abstract}

\maketitle

\tableofcontents

\section{Introduction}

\subsection{Quasimorphisms and commutator length}

A {\it quasimorphism} on a group $G$ is a real-valued function $\phi \colon G \to \RR$ on $G$ satisfying
\[D(\phi) := \sup \{ |\phi(xy) - \phi(x) - \phi(y)| \; | \; x,y \in G\} < \infty.\]
We call $D(\phi)$ the {\it defect} of the quasimorphism $\phi$. A quasimorphism $\phi$ on $G$ is said to be \textit{homogeneous} if $\phi(x^n) = n \cdot \phi(x)$ for every $x \in G$ and for every integer $n$.
The (homogeneous) quasimorphisms are closely related to the second bounded cohomology group $\HHH^2_b(G)=\HHH^2_b(G;\RR)$ (we always consider real coefficient in this survey), and have been extensively studied in geometric group theory, symplectic geometry and low dimensional topology (see \cite{Ca}, \cite{Mon01}, \cite{Fr}, and \cite{PR}).

Let $\QQQ(G)$ denote the real vector space consisting of homogeneous quasimorphisms on $G$, and $\HHH^1(G) = \HHH^1(G ; \RR)$ the real vector space consisting of homomorphisms from $G$ to $\RR$.
If $G$ is amenable, then $\QQQ(G)$ coincides with the space $\HHH^1(G)$ of homomorphisms from $G$ to $\RR$ (for example, see \cite[Proposition 2.65]{Ca} and \cite[Corollary 3.8]{Fr}).
Contrastingly, when $G$ is a non-elementary hyperbolic group, then the quotient $\QQQ(G) / \HHH^1(G)$ is infinite-dimensional \cite{EpFu}.
Thus, the space $\QQQ(G)/\HHH^1(G)$ (and hence $\QQQ(G)$) is sometimes infinite dimensional and is a difficult object to comprehend.

For $x \in [G,G]$, the \emph{commutator length}
$\cl_{G}(x)$ of $x$ is the smallest integer $n$ such that there exist $n$ commutators whose product is $x$.
The \emph{stable commutator length} $\scl_{G}(x)$ of $x \in[G,G]$ is defined by $\scl_{G}(x)=\lim\limits_{n\to \infty} \frac{\cl_{G}(x^n)}{n}$.
Commutator lengths and stable commutator lengths have been extensively studied in the following fields in geometric topology.

\begin{enumerate}[(1)]
  \item the commutator length on the mapping class group is related to the existence of Lefschetz fibrations (for example, see \cite{EK}, \cite[Subsection 3.6.1]{Ca}).
  \item the commutator length on groups of diffeomorphisms is well studied and related to the existence of transverse foliations (for example, see \cite{MR656217}, \cite{BIP}, \cite{Tsuboi12}, \cite{Tsuboi13} and \cite{Tsuboi17}).
  \item the (stable) commutator length on the fundamental group of a 3-manifold is related to geometric structures on the manifold, such as the Thurston norm, taut foliations, and slithering around the circle (see  \cite{math/9712268}, \cite[Section 4.1]{Ca}).
 % \item the commutator length on the universal covering of the group of homeomorphisms on the circle is related to dynamics on circle (see 2.3.3 of \cite{Ca}).
 \item  the stable commutator length is used to construct prescribed values of the simplicial volume (see \cite{FFL21}, \cite{HL19}, \cite{HL21a}, \cite{HL21b}, \cite{Loh23}).
\end{enumerate}

Many topics related to $\scl$ are discussed in detail in Calegari's book \cite{Ca}.
For other recent extrinsic applications of $\scl$, see \cite{IMT19},  \cite{MP19} and \cite{MP20} for instance.

The following deep relationship between quasimorphisms and $\scl$ is known as \textit{Bavard's duality}.
\begin{thm}[\cite{Bav}]\label{original Bavard}
For every $x \in [G,G]$,
  \[ \scl_{G}(x) = \sup_{\phi \in \QQQ(G)-\HHH^1(G)} \frac{|\phi(x)|}{2D(\phi)}. \]
   Here we set the supremum in the right-hand side of the above equality to be zero if $\QQQ(G) = \HHH^1(G)$.
\end{thm}

\subsection{Invariant quasimorphisms and mixed commutator length}
In this survey, two main objects are considered: the first is the \emph{invariant quasimorphism}.

Let $N$ be a normal subgroup of $G$.
A quasimorphism $\phi \colon N\to \RR$ is said to be \emph{$G$-invariant} if $\phi(gxg^{-1})=\phi(x)$ for every $g\in G$ and every $x\in N$.
Let $\QQQ(N)^G$ denote the real vector space consisting of $G$-invariant homogeneous quasimorphisms on $N$.
Note that $\QQQ(G)^G=\QQQ(G)$:
every homogeneous quasimorphism on $G$ is $G$-invariant (Lemma \ref{basic lemma}).

One of the origins of the concept of invariant quasimorphism is symplectic geometry.
A symplectic manifold has several natural transformation groups and so the concept of invariant quasimorphism naturally appears (see Section \ref{prehistory in symp}).
In some context, $\mathrm{Aut}(G)$-invariant quasimorphism has been studied (see Section \ref{aut section}).

As homogeneous quasimorphisms are closely related to stable commutator lengths,
invariant homogeneous quasimorphisms are related to the stabilization of a certain word length. This word length is the \emph{mixed commutator length} taken note of in \cite{KK}. This is the second main object in this survey.

An element of the form $[g,x] = gxg^{-1}x^{-1}$ with $g\in G$ and $x\in N$ is called a \emph{$(G,N)$-commutator} or a  \emph{mixed commutator}; $[G,N]$ is the subgroup generated by mixed commutators, and it is called the \emph{$(G,N)$-commutator subgroup} or the \emph{mixed commutator subgroup}.
Note that since $N$ is normal, $[G,N]$ is a normal subgroup of $G$ and contained in $N$.
The \emph{$(G,N)$-commutator length} or the \emph{mixed commutator length} $\cl_{G,N}$ is the word length on $[G,N]$ with respect to the set of mixed commutators. That is, for $x\in [G,N]$, $\cl_{G,N}(x)$ is the smallest integer $n$ such that there exist $n$ mixed commutators whose product is $x$.
In the case of $N=G$, the notions of $[G,N]$ and $\cl_{G,N}$ coincide with those of the \emph{commutator subgroup} $[G,G]$ and the \emph{commutator length} $\cl_G$, respectively.

We also define the \emph{stable mixed commutator length} $\scl_{G,N}(x)$ of $x \in[G,N]$ by $\scl_{G,N}(x)=\lim\limits_{n\to \infty} \frac{\cl_{G,N}(x^n)}{n}$.
We have the Bavard duality for $\hG$-invariant quasimorphisms and mixed commutator length.

\begin{thm}[\cite{KK},\cite{KKMM1}] \label{thm:bavard}
 For every $x \in [\hG,\bG]$,
  \[ \scl_{\hG,\bG}(x) = \sup_{\phi \in \QQQ(N)^{\hG}-\HHH^1(N)^{G}} \frac{|\phi(x)|}{2D(\phi)}. \]
   Here we set the supremum in the right-hand side of the above equality to be  zero if $\QQQ(N)^G = \HHH^1(N)^G$.
\end{thm}

Note that Theorem \ref{thm:bavard} in the case $G=N$ is exactly Theorem \ref{original Bavard} (see Figure \ref{fig:bavard}).

We note that symplectic geometry provides strong (and original) motivation to study invariant quasimorphisms. The concept of stable mixed commutator lengths is closely related to that of invariant quasimorphisms via Theorem \ref{thm:bavard}. For this reason, in this survey we regard mixed commutator lengths as an object of geometric and dynamical interests, as well as invariant quasimorphisms.

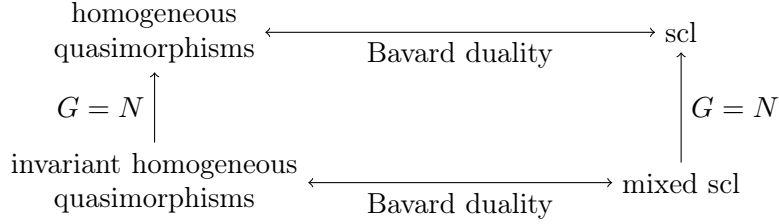
\begin{figure}[t]
  \centering
    \begin{tikzpicture}[auto]
    \node[align=center] (01) at (0, 2) {homogeneous \\ quasimorphisms};
    \node (11) at (7, 2) {scl};
    \node[align=center] (00) at (0, 0) {invariant homogeneous \\ quasimorphisms};
    \node (10) at (7, 0) {mixed scl};
    \draw[<-] (01) to node[left] {$G=N$} (00);
    \draw[<->] (01) to node[below] {Bavard duality} (11);
    \draw[<-] (11) to node {$G=N$} (10);
    \draw[<->] (00) to node[below] {Bavard duality} (10);
    \end{tikzpicture}
    \caption{Bavard duality}  \label{fig:bavard}
\end{figure}

\subsection{Extension problem of quasimorphisms}
The first central problem that we treat in this survey is the extension problem of quasimorphisms. In this context, invariant homogeneous quasimorphisms naturally appear. We start with the following notion:

\begin{definition}
  A homogeneous quasimorphism $\phi\colon N\to\RR$ on $N$ is said to be  \emph{extendable to $G$} if there exists a homogeneous quasimorphism $\hat{\phi}\colon G\to\RR$ on $G$ such that $\hat{\phi}|_N=\phi$.
\end{definition}

Then it is natural to ask whether there exist homogeneous quasimorphisms that are not extendable to $G$. However, this is not an interesting problem in general. Indeed, the following basic fact implies that extendable homogeneous quasimorphisms must be $G$-invariant.

\begin{lem}\label{basic lemma}
Let $G$ be a group.
Every homogeneous quasimorphism $\phi$ on $G$ is $G$-invariant, \emph{i.e.},
\[ \phi(ghg^{-1}) = \phi(h)\]
for all $g,h \in G$.
\end{lem}

\begin{proof}

For every positive integer $m$, we have
\begin{multline*}
  |\phi(g h g^{-1}) - \phi(h) |= \frac1m|\phi(g h^{m} g^{-1})-\phi(h^m) | \\ = \frac1m|\phi(g h^{m} g^{-1})-\phi(g) - \phi(h^{m}) - \phi(g^{-1})| \leq \frac{2D(\phi)}{m} .
\end{multline*}
%\[ |\phi(g h g^{-1}) - \phi(h) |= \frac1m|\phi(g h^{m} g^{-1})-\phi(h^m) | = \frac1m|\phi(g h^{m} g^{-1})-\phi(g) - \phi(h^{m}) - \phi(g^{-1})| \leq \frac{2D(\phi)}{m} .\]
By letting $m\to\infty$, we obtain the assertion.
\end{proof}

It is often easy to construct homogeneous quasimorphisms on $N$ that are not $G$-invariant. Such quasimorphisms are not extendable by Lemma~\ref{basic lemma}. Therefore, the natural problem we should consider is as follows:

\begin{problem}\label{large problem 1}
Let $N$ be a normal subgroup of $G$ and $\phi\colon N\to\RR$ be a $G$-invariant homogeneous quasimorphism on $N$.
Is $\phi$ extendable to $G$ or not?
\end{problem}%Later, we will define the space of non-extendable quasimorphisms in two ways.

Let $i^\ast \colon \QQQ(G) \to \QQQ(N)$ be the restriction map. Then Lemma~\ref{basic lemma} implies that the image $i^* (\QQQ(G))$ is contained in $\QQQ(N)^G$. Consider the quotient space
\[ \mathrm{V}(G,N) := \QQQ(N)^G / i^\ast \QQQ(G).\]
Then Problem~\ref{large problem 1} is equivalent to the following problem: given a $G$-invariant homogeneous quasimorphism $\phi$, determine whether $\phi$ vanishes  in  $\mathrm{V}(G,N)$ or not.

A $G$-invariant homomorphism on $N$ seems `trivial' as a quasimorphism in $\QQQ(N)^G$; also, when we apply the Bavard duality theorem for stable mixed commutator lengths (see Theorem \ref{thm:bavard}), exactly elements in $\HHH^1(N)^G$ behave trivially.
Thus, we also consider the following variant of Problem \ref{large problem 1}.

\begin{problem}\label{large problem 3}
Let $N$ be a normal subgroup of $G$ and $\phi\colon N\to\RR$ a $G$-invariant homogeneous quasimorphism on $N$.
Can $\phi$ be represented as the sum of a homogeneous quasimorphism extendable to $G$
%the restriction of a homogeneous quasimorphism on $G$
and a $G$-invariant homomorphism on $N$?
\end{problem}

As in the case of $\mathrm{V}(G,N)$, consider the quotient space
\[ \WWW(G,N) := \QQQ(N)^G / (\HHH^1(N)^G + i^\ast \QQQ(G)).\]
Then Problem~\ref{large problem 3} is equivalent to the following problem: given a $G$-invariant homogeneous quasimorphism $\phi$, determine whether $\phi$ vanishes  in  $\WWW(G,N)$ or not.
This space $\WWW(G,N)$ is closely related to the comparison problem of $\scl_G$ and $\scl_{G,N}$, which is the second main problem in this survey and is discussed in the next subsection (see also Proposition~\ref{prop equivalence criterion}).

As we will argue in Remark \ref{rmk:acyl_hyp}, the space $\QQQ(N)^G/\HHH^1(N)^G$ is infinite dimensional in several situations related to acylindrically hyperbolic groups.
Nevertheless, the dimensions of the spaces $\VVV(G,N)$ and $\WWW(G,N)$ can be \textit{finite} in natural situations
%(see Section~\ref{sp sec})
(Corollary \ref{cor:fin_dim}).
Furthermore, these spaces can be \emph{non-zero finite-dimensional} in several interesting cases (see Section~\ref{sec:ext_prob}).

We call Problems \ref{large problem 1} and \ref{large problem 3} the \emph{extension problems} of an invariant quasimorphism $\phi$.

%Here we discuss some cases of extension problems. Let $p \colon G \to \Gamma = G/N$ be the projection. We say that $p$ \emph{virtually splits} if there exists a finite index subgroup $\Lambda$ of $\Gamma$ such that $p^{-1}(\Lambda) \to \Lambda$ has a section homomorphism. We will see that when the projection $p$ virtually splits, then the space $\mathrm{V}(G,N)$ vanishes, {\it i.e.}, every element in $\QQQ(N)^G$ is extendable (see Theorem~\ref{virtually split}). In particular, the space $\VVV(G,N)$ \emph{always} vanishes for a pair $(G,N)$ with $G/N \cong \ZZ$. The situation becomes \emph{completely different} if $\Gamma = \ZZ^2$. There exist several examples of $(G,N)$ such that $G/N \cong \ZZ^2$ for which the projection $p$ does \emph{not} virtually split. In fact, we study \emph{Py's Calabi quasimorphism} in this setting, and obtain an extrinsic application to

%of independent interest in symplectic geometry; see Section~\ref{history after}.

We note that the extension problems have some relation to study of the corresponding short exact sequence
\begin{align*}
  1\longrightarrow N \stackrel{i}{\longrightarrow} G\stackrel{p}{\longrightarrow} \Gamma \longrightarrow 1. %\tag{$\star$}
\end{align*}
More precisely, we prove that if the above sequence \emph{virtually splits}, then $\mathrm{V}(G,N)$ vanishes, \emph{i.e.}, every element in $\QQQ(N)^{G}$ is extendable; see Theorem \ref{virtually split}. In particular,  the space $\mathrm{V}(G,N)$ \emph{always} vanishes for a pair $(G,N)$ with $\Gamma=\ZZ$.
The situation is \emph{completely different} if $\Gamma=\ZZ^2$: there exist several interesting examples of $(G,N)$ with $\Gamma=\ZZ^2$ for which the corresponding short exact sequence does \emph{not} virtually split.
In fact, we study \emph{Py's Calabi quasimorphism} in this setting, and obtain an application of independent interest in symplectic geometry; see Section \ref{history after} (Theorem \ref{thm A}).

To conclude this subsection, we remark on the relation between the extension problems and an open problem for Lex groups.
A group $\Gamma$ is said to be \emph{Lex} if, for every group $G$ and for every surjection $p \colon G \to \Gamma$, the pullback $p^* \colon \HHH_b^{\bullet}(\Gamma) \to \HHH_b^{\bullet}(G)$ is injective in all degrees.
The concept of Lex groups was first introduced (in a different language) and studied in \cite{MR1896181}.
There exist many Lex groups, but there are no known examples of non-Lex groups; see \cite[Subsection 3.2]{FLM1}.
By the exact sequence of bounded cohomology groups (\cite{MR1338286}, \cite{Mon01})
\[
  0 \to \HHH_b^2(\Gamma) \to \HHH_b^2(G) \xrightarrow{i^*} \HHH_b^2(N)^G \to \HHH_b^3(\Gamma) \to \HHH_b^3(G),
\]
the non-injectivity of $\HHH_b^3(\Gamma) \to \HHH_b^3(G)$ corresponds to the existence of a non-extendable second bounded cohomology class in $\HHH_b^2(N)^{G}$.
In other words, the non-injectivity of $\HHH_b^3(\Gamma) \to \HHH_b^3(G)$ is equivalent to the non-vanishing of the quotient space $\HHH_b^2(N)^G/i^*\HHH_b^2(G)$.

Note that the coboundary operator $\delta^1 \colon C^1(N) \to C^2(N)$ (see Subsection \ref{subsec:rotation_invq}) induces a well-defined map $\WWW(G,N) \to \HHH_b^2(N)^G/i^*\HHH_b^2(G)$.
Hence the non-zero element of $\WWW(G,N)$ potentially provides a non-zero element of $\HHH_b^2(N)^G/i^*\HHH_b^2(G)$, which implies that the group $\Gamma$ is not Lex (see also around Open Problem \ref{prob:nonlex} for another formulation).
Unfortunately, for all known examples of pairs of groups $(G,N)$ which admits a non-zero element $\WWW(G,N)$, the bounded cohomology group of $\Gamma$ vanishes in all degrees and hence the quotient space $\HHH_b^2(N)^G/i^*\HHH_b^2(G)$ also vanishes.
Therefore the known non-zero elements of $\WWW(G,N)$ vanish in $\HHH_b^2(N)^G/i^*\HHH_b^2(G)$ and do not supply the examples of non-Lex groups.

\begin{comment}
There is the variant of the extension problems which asks whether a given $G$-invariant homogeneous quasimorphism on $N$ extends to a \emph{second bounded cohomology class} on $G$.
These three extension problems relate to an open problem to find examples of groups that is not Lex.
Here a group $\Gamma$ is said to be \emph{Lex} if, for every group $G$ and for every surjection $p \colon G \to \Gamma$, the pullback $p^* \colon \HHH_b^{\bullet}(\Gamma) \to \HHH_b^{\bullet}(G)$ is injective in all degree.
The concept of Lex was introduced (in a different language) and studied in \cite{MR1896181}; see also \cite{FLM1}.
\end{comment}

\subsection{Comparison problem of $\scl_G$ and $\scl_{G,N}$}
We have two functions $\scl_G$ and $\scl_{G,N}$ on $[G,N]$, which are closely related to $\QQQ(G)$ and $\QQQ(N)^G$, respectively.
The second main problem treated in this survey is the comparison problem between $\scl_G$ and $\scl_{G,N}$, \emph{i.e.}, how different are $\scl_G$ and $\scl_{G,N}$? In general, it is difficult to determine the precise values of $\scl_G(x)$ and $\scl_{G,N}(x)$. However, we can sometimes describe a difference in the global behaviors of $\scl_G$ and $\scl_{G,N}$; $\scl_G$ and $\scl_{G,N}$ are bi-Lipschitzly equivalent or not.

Let $\nu_0$ and $\nu_1$ be two real-valued functions on a set $X$. We say that $\nu_0$ and $\nu_1$ are \emph{bi-Lipschitzly equivalent}, (or \emph{equivalent} for short) if there exist positive constants $A$ and $B$ such that $A \nu_0 \le \nu_1 \le B \nu_0$. We say that $\scl_G$ and $\scl_{G,N}$ (or $\cl_G$ and $\cl_{G,N}$) are \emph{equivalent} if they are equivalent as functions on $[G,N]$.

%For two non-negative-valued functions $\nu_0$ and $\nu_1$ on a group $G$,
%we say that $\nu_0$ and $\nu_1$ are \textit{bi-Lipschitzly equivalent} (or \textit{equivalent} in short)
%if there exist positive constants $C_1$ and $C_2$ such that $C_1 \nu_0 \leq \nu_1 \leq C_2 \nu_0$.

\begin{problem}\label{large problem 2}
Let $N$ be a normal subgroup of $G$.
Are $\scl_G$ and $\scl_{G,N}$ \textup{(}or $\cl_\hG$ and $\cl_{\hG, \bG}$\textup{)} equivalent on $[\hG,\bG]$?
More strongly, do they coincide on $[\hG,\bG]$?
\end{problem}

We call Problem \ref{large problem 2} the \emph{comparison problem between \textup{(}stable\textup{)} commutator length and \textup{(}stable\textup{)} mixed commutator length.}
We note that Calegari and Zhuang \cite{CZ} defined other variants of stable commutator length and considered the comparison problem between the usual stable commutator length and their variants.
%$\cl_\hG$ and $\cl_{\hG, \bG}$.

%The following variant of Problem \ref{large problem 1} is related to Problem \ref{large problem 2}.

The extension problem and the comparison problem are closely related (see Proposition~\ref{prop equivalence criterion}). For example, $\WWW(G,N) = 0$ (see Section~1.3) implies that $\scl_G$ and $\scl_{G,N}$ are equivalent (see Theorem~\ref{equiv thm ex}).
Note that there are several examples of $(G,N)$ where $\scl_G$ and $\scl_{G,N}$ are not equivalent (see the end of Section \ref{scl sec}).

%The results on (stable) mixed commutator lengths might be significant not only for their own interest, but also because they reflect developments in the theory of invariant quasimorphisms.

Finally, we also mention the comparison problem between $\scl_N$ and $\scl_{G,N}$, not between $\scl_G$ and $\scl_{G,N}$. In general, $\scl_N$ and $\scl_{G,N}$ are greatly different. One reason of this is that $\scl_{G,N}$ is $G$-invariant while $\scl_N$ is not in general.
There are many examples where $\scl_{G,N}$ and $\scl_N$ are not equivalent on $[N,N]$, even when $N$ is of finite index in $G$ (see \cite[Proposition 3.2]{KK}, \cite[Proposition 7.19]{KKMM1}).
%In general, $\cl_{G,N}$ and $\cl_N$ differ greatly since  $\cl_{G,N}$ is $G$-invariant while $\cl_N$ is not:
%there are many examples where $\cl_{G,N}$ and $\cl_N$ are not equivalent on $[N,N]$, even when $N$ is of finite index in $G$.
%In contrast, finding the difference between $\cl_{G,N}$ and $\cl_N$ (or $\scl_{G,N}$ and $\scl_N$) is not as easy as it seems.

%As we explain in Proposition \ref{prop equivalence criterion}, Problem \ref{large problem 2} is deeply related to Problem \ref{large problem 3}.

\subsection*{Notation and terminology}
Throughout this survey, let $G$ be a group, $N$ a normal subgroup of $G$, and set $\Gamma = G/N$.
Let $i \colon N \to G$ denote the inclusion map and $p \colon G \to \Gamma$ denote the natural projection.
Namely, we have the following short exact sequence
\begin{align*}
  1\longrightarrow N \stackrel{i}{\longrightarrow} G\stackrel{p}{\longrightarrow} \Gamma \longrightarrow 1. \tag{$\star$}
  \end{align*}
  Let $\Sigma_{\genus}$ denote a closed connected oriented surface of genus $\genus$.
We assume surfaces to be connected in this survey.

\subsection*{Organization of the paper}
In Section~2, we describe the motivation of the study of invariant homogeneous quasimorphisms from symplectic geometry.
%In Section~3, we review some basic facts and results on invariant homogeneous quasimorphisms.
 In Section~3, we review several works related to invariant quasimorphisms after 2010. In Section~4, we discuss known explicit examples of invariant quasimorphisms and invariant partial quasimorphisms. 
In Section~5, we review several results on Aut-invariant quasimorphisms and their relation to invariant quasimorphisms.
 In Section~6, we provide an outline of the proof of Bavard's duality theorem of stable mixed commutator lengths (Theorem~\ref{thm:bavard}), following \cite{KKMM1}. 
In Section~7, we review the results in \cite{K2M3} related to the spaces $\VVV(G,N)$ and $\WWW(G,N)$. Section~8 is devoted to the comparison problem between $\scl_G$ and $\scl_{G,N}$. In Section~9, we treat concrete examples of a pair $(G,N)$ and discuss the extension and comparison problems in these cases. In Section~10, we discuss the comparison problem between $\cl_G$ and $\cl_{G,N}$. In Section~11, we discuss invariant partial quasimorphism, which is a generalization of invariant quasimorphisms in some sense.

%%%%%%%%%%%%%%%%%%%%%%%%%%%%%%%%%%%%%%%%%%%%%%%%%%%%%%%%%%%%%%%%%%%%%%%%%%%%%%%%%%%%%%%%%%%%%%%%%%%%%%%%%%%%%%%%%%%%%%%%%%%%%%%%%
%%%%%%%%%%%%%%%%%%%%%%%%%%%%%%%%%%%%%%%%%%%%%%%%%%%%%%%%%%%%%%%%%%%%%%%%%%%%%%%%%%%%%%%%%%%%%%%%%%%%%%%%%%%%%%%%%%%%%%%%%%%%%%%%%
%%%%%%%%%%%%%%%%%%%%%%%%%%%%%%%%%%%%%%%%%%%%%%%%%%%%%%%%%%%%%%%%%%%%%%%%%%%%%%%%%%%%%%%%%%%%%%%%%%%%%%%%%%%%%%%%%%%%%%%%%%%%%%%%%
%%%%%%%%%%%%%%%%%%%%%%%%%%%%%%%%%%%%%%%%%%%%%%%%%%%%%%%%%%%%%%%%%%%%%%%%%%%%%%%%%%%%%%%%%%%%%%%%%%%%%%%%%%%%%%%%%%%%%%%%%%%%%%%%%

\section{Prehistory of invariant quasimorphisms in symplectic geometry}\label{prehistory in symp}

%As is explained in Subsection \ref{subsec:motive}, a symplectic manifold has some natural transformation groups.
Let $(M,\omega)$ be a symplectic manifold.
Then $(M,\omega)$ has several natural transformation groups, and here we consider the following three.
The first one is the group $\Symp(M,\omega)$ of symplectomorphisms with compact support, where a diffeomorphism is called a \textit{symplectomorphism} if it preserves the symplectic form $\omega$.
The second one is the identity component of $\Symp(M,\omega)$, which is denoted by $\Symp_0(M,\omega)$.
The third one is the group $\Ham(M,\omega)$ of Hamiltonian diffeomorphisms, which is defined to be $\Ham(M,\omega)=[\Symp_0(M,\omega),\Symp_0(M,\omega)]$.
Here, we do not explain the geometric definition of $\Ham(M,\omega)$, but we note that the concept of Hamiltonian diffeomorphism comes from the classical analytical dynamics.
We note that $\Ham(M,\omega)$ is a normal subgroup of the both of $\Symp(M,\omega)$ and $\Symp_0(M,\omega)$.
For basics of symplectic geometry, see \cite{Ban97}, \cite{MS}, and \cite{P01} for instance.
In the context of symplectic geometry, $\Symp(M,\omega)$-invariant homogeneous quasimorphisms on $\Ham(M,\omega)$ naturally appear.

Here, we explain the flux homomorphism, which detects the difference between $\Ham(M,\omega)$ and $\Symp_0(M,\omega)$.
Let $\widetilde{\Symp_0}(M, \omega)$ denote the universal covering of $\Symp_0(M ,\omega)$. Then an element of $\widetilde{\Symp_0}(M,\omega)$ can be considered as a homotopy class of a path in $\Symp_0(M,\omega)$ from the identity.
The {\it flux homomorphism} $\widetilde{\flux}_\omega \colon \widetilde{\Symp_0}(M, \omega) \to \HHH_c^1(M; \RR)$ is defined by
\[\widetilde{\flux}_\omega([\{\psi^t\}_{t\in[0,1]}])=\int_0^1[\iota_{X_t}\omega]dt,\]
where $\iota$ is the inner product, $X_t$ is the vector field generating the symplectic isotopy $\{\psi^t\}$ and $[\iota_{X_t}\omega]$ is the cohomology class represented by $\iota_{X_t}\omega$.
Here, $\HHH_c^{\bullet}(M; \RR)$ denotes the de Rham cohomology with compact support.
The image of $\pi_1(\Symp_0(M, \omega))$ $(\subset \widetilde{\Symp_0}(M,\omega)$) with respect to $\widetilde{\flux_\omega}$ is called the {\it flux group of $(M, \omega)$} and denoted by $\Gamma_\omega$
% It is known that the flux group is a discrete subgroup of $\HHH_c^1(M)$ for every closed symplectic manifold $(M,\omega)$. This statement is called the flux conjecture, which had been a long standing conjecture in symplectig geometry, and was finally settled by Ono \cite{O}.
\footnote{The discreteness of the flux group is called the flux conjecture, which had been a long standing famous conjecture in symplectic geometry.
The flux conjecture was proved by Ono \cite{O} for every closed symplectic manifold.}
.
The flux homomorphism descends to a homomorphism
\[\flux_\omega \colon \Symp_0(M, \omega) \to \HHH_c^1(M ; \RR) / \Gamma_\omega, \]
which is also called the \emph{flux homomorphism}.
Then, it is known that the flux homomorphism $\flux_\omega$ is surjective
% (see \cite[Proposition 1.4.2]{Ban97})
and $\Ham(M,\omega)=\mathrm{Ker}(\flux_\omega)$.
In particular, $\Ham(M,\omega)$ and $\Symp_0(M,\omega)$ coincide if  and only if $\HHH_c^1(M;\RR)=0$.
The flux homomorphism is a fundamental object in symplectic geometry and has been extensively studied by many authors
such as \cite{Ban}, \cite{LMP}, and \cite{Ke}.

\begin{example}\label{flux on torus}
Let $T^2$ be the $2$-dimensional torus $\RR^2/\ZZ^2$ with coordinates $(p,q)$ and $\omega_T$ a symplectic (area) form $dq\wedge dp$.
For a pair of real numbers $(a,b)$, define $\varphi_{a,b} \in \Symp_0(T^2,\omega_T)$ to be $\varphi_{a,b}(x,y)=(x+a,y+b)$.
Define loops $\gamma_1, \gamma_2 \colon [0,1] \to \Symp_0(T^2, \omega_T)$ by $\gamma_1(t)=\varphi_{t,0}$, $\gamma_2(t)=\varphi_{0,t}$, respectively.
Then, $\pi_1(\Symp_0(T^2, \omega_T))$ is an abelian group generated by $[\gamma_1]$ and $[\gamma_2]$.
We also have that $\widetilde{\flux}_\omega([\gamma_1])=[dp] \in \HHH^1(T^2;\RR) \cong \RR^2$, $\widetilde{\flux}_\omega([\gamma_2])=-[dq] \in \HHH^1(T^2;\RR) \cong \RR^2$.
Thus, $\HHH^1(T^2 ; \RR) / \Gamma_{\omega_T} \cong \RR^2/\ZZ^2$ and $\varphi_{a,b} \in \Ham(T^2,\omega_T)$ if and only if $(a,b) \in \ZZ^2$.
\end{example}

\begin{example}\label{flux on higher genus}
Let $\Sigma_\genus$ be a closed orientable surface whose genus $\genus$ is at least two and $\omega$ a symplectic form on $\Sigma_\genus$.
Then, $\Gamma_\omega$ is trivial (for example, see \cite[Subsection 7.2.B]{P01}) and hence the flux homomorphism $\flux_\omega$ is a map from $\Symp_0(\Sigma_\genus, \omega)$ to $\HHH^1(\Sigma_\genus ; \RR)$.
\end{example}

In symplectic geometry, a lot of quasimorphisms on (the universal covering of) $\Ham(M,\omega)$ have been constructed.
One of the first works is due to Barge and Ghys \cite{BG}.
They constructed a homogeneous quasimorphism on $\Ham(\RR^n,\omega_0)$, where $\omega_0$ is the standard symplectic form on $\RR^n$.
%Their quasimorphism is $\Symp(\RR^n,\omega_0)$-invariant.

One of the famous works is due to Gambaudo and Ghys \cite{GG}.
They constructed various homogeneous quasimorphisms on $\Symp_0(\Sigma,\omega)$ not vanishing on
%and they do not vanish on
$\Ham(\Sigma,\omega)$, where $\Sigma$ is a compact orientable surface and $\omega$ is a symplectic form on $\Sigma$.

An important example of invariant quasimorphism appeared in Entov--Polterovich's work \cite{EP03}.
They constructed various homogeneous quasimorphisms on $\Ham(M,\omega)$ and proved that their quasimorphisms are $\Symp(M,\omega)$-invariant for some $(M,\omega)$ (for example, $(M,\omega)$ is the $n$-dimensional complex projective space $\CC P^n$ with the Fubini--Study form $\omega_{FS}$).
Their quasimorphisms are constructed as  a constant multiplication of the homogenization of Oh--Schwarz's spectral invariant, which is a real-valued function on the universal covering $\tHam(M,\omega)$ of the group $\Ham(M,\omega)$ of Hamiltonian diffeomorphisms on a closed symplectic manifold $(M,\omega)$ defined through the Hamiltonian Floer theory.
We will explain a sketch of their construction in Subsection \ref{contrcution of EP}.
For more precise description, see \cite{O05}, \cite{O15}, \cite{PR}.

They further generalized this result in \cite{EP06}.
They introduced the concept of partial quasimorphism, which is a generalization of the concept of quasimorphism (see Section \ref{partial section}).
Recall that in \cite{EP03} they proved that the homogenization of Oh--Schwarz's spectral invariant descends to $\Ham(M,\omega)$ and is a quasimorphism for some closed symplectic manifolds.
On the other hand, in \cite{EP06} they proved that the homogenization of Oh--Schwarz's spectral invariant is a partial quasimorphism on every closed symplectic manifold.
For example, the homogenization of Oh--Schwarz's spectral invariant on a closed orientable surface with genus at least one is not a quasimorphism but a partial quasimorphism.

As a surprising extrinsic application of invariant partial quasimorphism, one can prove strong non-displaceability of subsets of a symplectic manifold \cite{EP09}.
Before explaining it, we  introduce some notions. 
A subset $X\subset M$ is said to be \textit{non-displaceable} if  $\phi(X)\cap\overline{X} \neq \emptyset$ for every $\phi\in\Ham(M,\omega)$.
A subset $X\subset M$ is said to be \textit{strongly non-displaceable} if $\phi(X)\cap\overline{X} \neq \emptyset$ for every $\phi\in\Symp(M,\omega)$.
Non-displaceability of Lagrangian submanifold has been a traditional research topic in symplectic geometry after the famous work by Gromov \cite{Gr85} and Floer \cite{Fl88}.
Contrastingly, before Entov--Polterovich's work, proving strong non-displaceability has been a very difficult problem.\footnote{``detecting strong non-displaceabilty has at the moment the status of art rather than science'', Subsection 1.2.3 of \cite{EP09}.}

In Entov--Polterovich's work \cite{EP09},  they proved many closed subset of closed symplectic manifold are strongly non-displaceable.
The homogenization of Oh--Schwarz's spectral invariant is  $\Symp(M,\omega)$-invariant  for some cases;
this is one of the key properties in their argument.
%In their argument, the key property of the homogenization of Oh--Schwarz's spectral invariant is $\Symp(M,\omega)$-invariance for some cases.
The following two results are applications of their theory.

\begin{example}[{\cite[Example 1.21]{EP09}}]\label{EPeg}
Let $n$ be an even natural number, $M$ the Fermat quadratic $\{[z_0;z_1;\ldots;z_{n+1}]\in\CC P^{n+1} \,:\, -z_0^2+\sum_{j=1}^{n+1}z_j^2=0\}\subset\CC P^{n+1}$ and $\omega$ the restriction of the Fubini--Study form to $M$.
Then, every Lagrangian sphere (such as $M\cap\RR P^{n+1}$) is strongly non-displaceable in $M$.
\end{example}

\begin{example}[\cite{Ka19}]\label{union}\label{unioneg}
Let $(M_1,\omega_1)=(\CC P^n,\omega_{FS})$.
Let $X_1$ be either the Clifford torus $\{[z_0;z_1;\ldots;z_n]\in\CC P^n \,:\, |z_1|=\ldots=|z_n|\}$ or the $n$-dimensional real projective space $\RR P^n$.
%Let $(M_2,\omega_2)$ be the $2$-dimensional torus with the standard area (symplectic) form and $X_2,X_3$ a meridean, a longitude in $M_2$, respectively.
Let $(M_2,\omega_2)=(T^2,\omega_T)$ (recall Example \ref{flux on torus}).
Let $X_2$ and $X_3$ be a meridian and a longitude in $M_2$, respectively.
Then, $X_1\times (X_2\cup X_3)$ is strongly non-displaceable in $M_1\times M_2$.
\end{example}

Since $\mathrm{id}_{\CC P^n} \times \varphi_{a,b}$ (see Example \ref{flux on torus}) displaces $X_1\times X_2$ and $X_1\times X_3$ in Example \ref{union} for $(a,b) \in (\RR\setminus \ZZ)^2$, $X_1\times X_2$ and $X_1\times X_3$  are not strongly non-displaceable  in $M_1\times M_2$.

%We note that the extendability of partial quasimorphisms constructed by Entov and Polterovich to $\Symp_0(M,\omega)$ is an open problem for many $(M,\omega)$
%(for a partial result, see Theorem \ref{nonext partial}).
Finally, we have a technical remark.
Partial quasimorphisms constructed by Entov and Polterovich have a useful property called the \textit{Calabi property} (Definition \ref{Calabi prop}).
This property plays an essential role in the proof of the above results on non-displaceability.
Existence of Calabi quasimorphism on a closed symplectic surface with positive genus had been an open problem and Py solved it in \cite{Py06}, \cite{Py06t}.
His quasimorphism in \cite{Py06} plays an important role in some of the author's work \cite{KK}, \cite{KKMM1}.
In the next section, we will explain these works (Theorems \ref{thm:py}, \ref{thm flux}).
We will describe the construction of that quasimorphism in Subsection \ref{contrcution of Py}.

%%%%%%%%%%%%%%%%%%%%%%%%%%%%%%%%%%%%%%%%%%%%%%%%%%%%%%%%%%%%%%%%%%%%%%%%%%%%%%%%%%%%%%%%%%%%%%%%%%%%%%%%%%%%%%%%%%%%%%%%%%%%%%%%%
%%%%%%%%%%%%%%%%%%%%%%%%%%%%%%%%%%%%%%%%%%%%%%%%%%%%%%%%%%%%%%%%%%%%%%%%%%%%%%%%%%%%%%%%%%%%%%%%%%%%%%%%%%%%%%%%%%%%%%%%%%%%%%%%%
%%%%%%%%%%%%%%%%%%%%%%%%%%%%%%%%%%%%%%%%%%%%%%%%%%%%%%%%%%%%%%%%%%%%%%%%%%%%%%%%%%%%%%%%%%%%%%%%%%%%%%%%%%%%%%%%%%%%%%%%%%%%%%%%%
%%%%%%%%%%%%%%%%%%%%%%%%%%%%%%%%%%%%%%%%%%%%%%%%%%%%%%%%%%%%%%%%%%%%%%%%%%%%%%%%%%%%%%%%%%%%%%%%%%%%%%%%%%%%%%%%%%%%%%%%%%%%%%%%%

\section{History of invariant quasimorphisms}\label{history after}

%After the 2010s, more researchers are interested in invariant quasimorphisms.

In this section, we review several works related to invariant quasimorphisms after  2010. 
%In this section, we review several works related to invariant (partial) quasimorphisms after 2010s.
%Ishida \cite{Ish} considered the extension problem of quasimorphisms.
%He proved that any $G$-invariant quasimorphism on $N$ is extendable to $G$ if $G/N$ is a finite group.
The extension problem of quasimorphisms from a finite index subgroup was considered in several contexts
(\cite{BF09}, \cite{Ish}, \cite{Mal}).
As mentioned in \cite[Lemma 4.2]{Mal} for example, if $N$ is a finite index normal subgroup of $G$, then every $G$-invariant homogeneous quasimorphism on $N$ extends to $G$.
Indeed, for every $\phi \in \QQQ(N)^G$, we can construct a (not necessarily homogeneous) quasimorphism $\hat{\phi}$ on $G$ such that $\hat{\phi}|_N = \phi$ without increasing the defect: $D(\hat{\phi}) \leq D(\phi)$ (see the proof of \cite[Lemma 3.1]{Ish}).

In his Ph.D. thesis \cite[Remark 8.3]{IshidaThesis}, Ishida mentioned that if there exists a section homomorphism $\Gamma \to G$, then every $G$-invariant quasimorphism on $N$ extends to $G$ (see also \cite[Proposition 3.1]{KK}).
This fact was also obtained by a result of Shtern \cite[Theorem 3]{Sh}.
Shtern also provided the first example of non-extendable quasimorphism.
% (see Example~\ref{eg:shtern}).
\begin{example}[{\cite[Example 1]{Sh}}] \label{eg:shtern}
Let $G$ be  the (continuous) Heisenberg group.
There exists a non-trivial $G$-invariant homomorphism $\phi$ on the commutator subgroup $[G,G]$.
Then, this quasimorphism $\phi$ is non-extendable to $G$ because $G$ is nilpotent, in particular amenable.
\end{example}
%Shtern \cite{Sh} studied the conditions under which a quasimorphism can be extended.

%He also essentially proved that any $G$-invariant quasimorphism on $N$ is extendable to $G$ if  the projection $G \to \hG/\bG$ admits a section homomorphism $\hG/\bG \to \hG$.

As another work on the extension problem, one of the authors \cite{Ka18} provided some observation on the extendability of partial quasimorphisms (see Conjecture \ref{partial conj}, Theorem \ref{ext partial}).

Brandenbursky and Marcinkowski constructed infinitely many invariant quasimorphisms \cite{BM}.
Their quasimorphisms are $\Aut(F_2)$-invariant quasimorphisms on $F_2 \cong \mathrm{Inn}(F_2)$.
%We note that their quasimorphisms are extendable to $\Aut(F_2)$ (see Theorem \ref{aut_f2}).
After their work, Karlhofer (\cite{Karlh}, \cite{Karlh2}), and Fournier-Facio and Wade \cite{FW} generalized it (see Section \ref{aut section}).
We discuss the extension problem of these quasimorphisms in Section \ref{sec:ext_prob}.

Some of the authors \cite{KK} took note of the concept of the mixed commutator length and proved the Bavard duality theorem between the invariant quasimorphism and the mixed commutator length (Theorem \ref{thm:bavard}) under the assumption $N=[G,N]$.
They also proved that Py's Calabi quasimorphism \cite{Py06}, which is a $\Symp(\Sigma_\genus,\omega)$-invariant homogeneous quasimorphism on $\Ham(\Sigma_\genus,\omega)$, is non-extendable to $\Symp_0(\Sigma_\genus,\omega)$ (in particular, to $\Symp(\Sigma_\genus,\omega)$).
%Here $S$ is a closed surface whose genus is larger than one, and $\omega$ a symplectic (area) form on $S$.
\begin{thm}[{\cite[Theorem 1.11]{KK}}] \label{thm:py}
Let $\Sigma_l$ be a closed orientable surface whose genus $l\geq 2$ and $\omega$ a symplectic form on $\Sigma_l$. Then Py's Calabi quasimorphism $\qm_P \colon \Ham(\Sigma_l, \omega) \to \RR$ is non-extendable to $ \Symp_0(\Sigma_l, \omega)$.
\end{thm}

In \cite{KKMM1} some of the authors proved the invariant version of the Bavard duality theorem (Theorem \ref{thm:bavard}) without any additional assumption on $(G,N)$ and also proved the following theorem (Theorem \ref{virtually split}).
%generalizing the above results by \cite{BF09,Ish,Mal,Sh}.

\begin{definition}
    We say that sequence $(\star)$ in Section 1 \textit{virtually splits} if $p \colon G \to \Gamma$ admits a virtual section, \emph{i.e.}, there exists a subgroup $\Lambda$ of $\Gamma$ of finite index and a homomorphism $s\colon \Lambda \to \hG$ such that $p \circ s=\mathrm{id}_\Lambda$.
\end{definition}

\begin{thm}[{\cite[Proposition 1.6]{KKMM1}}] \label{virtually split}
If $(\star)$ virtually splits,
then every $G$-invariant homogeneous quasimorphism on $N$ is extendable to $G$.
\end{thm}

This result includes the following conditions that appear above:
\begin{itemize}
  \item $N$ is a finite index normal subgroup of $G$
  \item there exists a section homomorphism $\Gamma \to G$
\end{itemize}

The work \cite{KKMM1} also contains results on the comparison problems between $\scl_G$ and $\scl_{G,N}$ (Theorem \ref{survey of braid} (1)) and between $\cl_G$ and $\cl_{G,N}$ (Section \ref{cl comparison sec}).

The authors \cite{K2M3} provided many results on invariant quasimorphisms and stable mixed commutator length.
In this survey, we will exhibit some parts of this work in Sections \ref{sp sec}, \ref{scl sec} and \ref{sec:ext_prob}.

\begin{comment}
Let $\Sigma_l$ be the compact orientable surface of genus $l>2$ without boundary.
Some of the authors constructed a $\pi_1(\Sigma_l)$-invariant homogeneous quasimorphism on the commutator subgroup $[\pi_1(\Sigma_l),\pi_1( \Sigma_l)]$ of $\pi_1( \Sigma_l)$ \cite{MMM}.
They also proved that their quasimorphism is non-extendable to $\pi_1( \Sigma_l)$.
\end{comment}

As a notable extrinsic application of the extension problem of invariant quasimorphisms, we state the following theorem in \cite{KKMM2} by some of the authors.
%As a notable extrinsic application, we state the following theorem in \cite{KKMM2} by the authors.
\begin{thm}[{\cite[Theorem~1.1]{KKMM2}}] \label{thm flux}
Let $\Sigma_{\genus}$ be a closed orientable surface whose genus $l$ is at least two and $\omega$ an area \textup{(}symplectic\textup{)} form on $\Sigma_{\genus}$.
%Let $\diff_0(\Sigma_{\genus},\omega)$ denote the identity component of the group of diffeomorphisms of $\Sigma_{\genus}$ that preserve $\omega$.
Assume that a pair $f,g \in  \Symp_0(\Sigma_{\genus},\omega)$ satisfies $fg = gf$. Then
\[ \flux_\omega(f) \smile \flux_\omega(g) = 0\]
holds true. Here, $\flux_{\omega}\colon \Symp_0(\Sigma_{\genus},\omega)\to \HHH^1(\Sigma_{\genus};\RR)$ is the flux homomorphism \textup{(}recall Example $\ref{flux on higher genus}$\textup{)}, and $\smile \colon \HHH^1(\Sigma_{\genus};\RR) \times \HHH^1(\Sigma_{\genus};\RR) \to \HHH^2(\Sigma_{\genus};\RR) \cong \RR$ denotes the cup product.
\end{thm}

The key to the proof of Theorem \ref{thm flux} is the following result, which improves Theorem \ref{thm:py}.
%the non-extendability theorem of Py's Calabi quasimorphism in \cite{KK}.

\begin{thm}[Non-extendability of Py's Calabi quasimorphism] \label{thm A}
Let $\Sigma_{\genus}$ be a closed orientable surface whose genus $\genus$ is at least two and $\omega$ a symplectic form on $\Sigma_{\genus}$.
Let $\bar{v}$, $\bar{w} \in \HHH^1(\Sigma_{\genus};\RR)$ with $\bar{v} \smile \bar{w}  \neq  0$. For a positive integer $k$, set $\Lambda_k = \langle \bar{v}, \bar{w} / k \rangle$ and $G_{\Lambda_k} = \flux_\omega^{-1}(\Lambda_k)$.
Then, there exists a positive integer $k_0$ such that Py's Calabi quasimorphism $\qm_P$ is {\rm non}-extendable to $G_{\Lambda_k}$ for every $k \ge k_0$.
\end{thm}

We remark that $\Lambda_k \cong \ZZ^2$,
and hence we have the following short exact sequence
\[ 1 \to \Ham(\Sigma_\genus, \omega) \to G_{\Lambda_k} \to \Lambda_k \cong \ZZ^2 \to 1, \]
which does not virtually split.
Theorem \ref{thm flux} follows from Theorems \ref{thm A} and \ref{virtually split} as follows.
Set $\bar{v}=\flux_\omega(f)$ and $\bar{w}=\flux_\omega(g)$ and assume that $\bar{v} \smile \bar{w} \neq 0$.
By Theorem \ref{thm A}, Py's Calabi quasimorphism is non-extendable to $G_{\Lambda_k}$ for sufficiently large $k$. On the other hand, by the assumption $fg = gf$, we can define a homomorphism $s \colon \Lambda_1=\langle \bar{v}, \bar{w} \rangle \to G_{\Lambda_k}$ by setting $s(\bar{v})=f$ and $s(\bar{w})=g$. Since $\Lambda_1$ is a subgroup of finite index in $\Lambda_k$ and by the definition of $s$, the map $s$ is a virtual section of $\flux_\omega \colon G_{\Lambda_k} \to \Lambda_k$. By Theorem \ref{virtually split}, every $G_{\Lambda_k}$-invariant quasimorphism on $\Ham(M,\omega)$ extends, but this is a contradiction. Therefore we can conclude that $\bar{v} \smile \bar{w} = 0$.

At the end of this section, we note that invariant quasimorphisms are abundant in certain cases.

\begin{remark} \label{rmk:acyl_hyp}
Let $G$ be an acylindrically hyperbolic group and $N$ an infinite normal subgroup of $G$.
Then, the space $i^{\ast}\QQQ(G)$ (and thus $\QQQ(N)^G$) is infinite-dimensional (\cite{BF02}, \cite[Corollary 4.3]{FW}).
We note that the following groups are acylindrically hyperbolic: non-elementary Gromov-hyperbolic groups, mapping class groups of $\Sigma_{\genus}$ with $\genus\geq 2$, groups defined by $s$ generators and $r$ relations with $s-r\geq 2$ \cite{Osin15}, and $\Aut(H)$ with $H$ non-elementary Gromov-hyperbolic \cite{GenevoisHorbez}.
Moreover, we can show that the space $\QQQ(N)^G/\HHH^1(N)^G$ is infinite-dimensional as follows:
By the result of Osin \cite{Osin}, $N$ is also acylindrically hyperbolic. In particular, $[N,N]$ is an infinite normal subgroup of $G$.
Thus, even if we restrict homogeneous quasimorphisms on $G$ to $N$ or even to $[N,N]$, the space consisting of them is infinite-dimensional.
Hence, the space $\QQQ(N)^G / \HHH^1(N)^G$ is still infinite-dimensional since every homomorphisms on $N$ vanishes on $[N,N]$.

\end{remark}

%%%%%%%%%%%%%%%%%%%%%%%%%%%%%%%%%%%%%%%%%%%%%%%%%%%%%%%%%%%%%%%%%%%%%%%%%%%%%%%%%%%%%%%%%%%%%%%%%%%%%%%%%%%%%%%%%%%%%%%%%%%%%%%%%
%%%%%%%%%%%%%%%%%%%%%%%%%%%%%%%%%%%%%%%%%%%%%%%%%%%%%%%%%%%%%%%%%%%%%%%%%%%%%%%%%%%%%%%%%%%%%%%%%%%%%%%%%%%%%%%%%%%%%%%%%%%%%%%%%
%%%%%%%%%%%%%%%%%%%%%%%%%%%%%%%%%%%%%%%%%%%%%%%%%%%%%%%%%%%%%%%%%%%%%%%%%%%%%%%%%%%%%%%%%%%%%%%%%%%%%%%%%%%%%%%%%%%%%%%%%%%%%%%%%
%%%%%%%%%%%%%%%%%%%%%%%%%%%%%%%%%%%%%%%%%%%%%%%%%%%%%%%%%%%%%%%%%%%%%%%%%%%%%%%%%%%%%%%%%%%%%%%%%%%%%%%%%%%%%%%%%%%%%%%%%%%%%%%%%

\section{Explicit construction of invariant quasimorphisms}\label{exp_const}

In this section, we describe the constructions of examples of invariant (partial) quasimorphisms.
In Sections \ref{sp sec} and \ref{sec:ext_prob}, we provide a method to prove the existence of a non-extendable quasimorphism, using homological algebra.
However, this method does not provide a concrete construction of a non-extendable quasimorphism and it is still difficult to construct a non-extendable quasimorphism for a wide class of pairs of groups.

\subsection{Py's Calabi quasimorphism}\label{contrcution of Py}
In Theorem~\ref{thm:py}, we state that Py's Calabi quasimorphism $\qm_P$ on $\Ham(\Sigma_l, \omega)$ is non-extendable to $\Symp_0(\Sigma_l, \omega)$. In this subsection, we briefly describe the construction of Py's Calabi quasimorphism \cite{Py06}.

We first recall some terminology. Here we introduce another definition of Hamiltonian vector field. Let $(M,\omega)$ be a symplectic manifold. For a smooth function $H\colon M\to\mathbb{R}$, we define the \textit{Hamiltonian vector field} $X_H$ associated with $H$ by
\[\omega(X_H,V)=-dH(V)\text{ for every }V \in \mathcal{X}(M),\]
where $\mathcal{X}(M)$ is the set of smooth vector fields on $M$.

For a (time-dependent) smooth function $H\colon  [0,1] \times M\to\mathbb{R}$ with compact support and for $t \in  [0,1] $, we define a function $H_t \colon M \to \mathbb{R}$ by $H_t(x)=H(t,x)$.
Let $X_H^t$ denote the Hamiltonian vector field associated with $H_t$ and let $\{ \varphi_H^t \}_{t\in\mathbb{R}}$ denote the isotopy generated by $X_H^t$ such that $\varphi^0=\mathrm{id}$.
We set $\varphi_H=\varphi_H^1$ and $\varphi_H$ is called the \emph{Hamiltonian diffeomorphism generated by $H$}.
For a connected symplectic manifold $(M,\omega)$, we provide another definition of the group $\Ham(M,\omega)$ of Hamiltonian diffeomorphisms by
\[\Ham(M,\omega)=\{\varphi\in\mathrm{Diff}(M)\;|\;\exists H\in C^\infty(   [0,1]\times M )\text{ such that }\varphi=\varphi_H\}.\]
%Then, $\Ham(M,\omega)$ is a normal subgroup of $\Symp_0(M,\omega)$.

%Let $\widetilde{\Symp_0}(M, \omega)$ denote the universal covering of $\Symp_0(M,\omega)$.
%We define the (symplectic) flux homomorphism $\widetilde{\flux}_\omega\colon\widetilde{\Symp_0}(M,\omega)\to H_c^{1}(M;\RR)$ by
%\[\widetilde{\flux}_\omega([\{\psi^t\}_{t\in[0,1]}])=\int_0^1[\iota_{X_t}\omega]dt,\]
%where $\{\psi^t\}_{t\in[0,1]}$ is a path in $\Symp_0(M,\omega)$ with $\psi^0=1$ and $[\{\psi^t\}_{t\in[0,1]}]$ is the element of the universal covering $\widetilde{\Symp_0}(M,\omega)$ represented by the path $\{\psi^t\}_{t\in[0,1]}$.
%It is known that $\widetilde{\flux}_\omega$ is a well-defined homomorphism.

In this section, we recall the definition of the Calabi property with respect to  quasimorphisms and that of Calabi quasimorphisms in symplectic geometry.
A subset $X$ of a connected symplectic manifold $(M,\omega)$ is said to be \textit{displaceable} if there exists $\varphi \in \Ham(M,\omega)$ satisfying $\varphi(X) \cap \overline{X} = \emptyset$. Here, $\overline{X}$ is the topological closure of $X$  in $M$.

Let $(M,\omega)$ be a $2n$-dimensional  \emph{exact} symplectic manifold, meaning that the symplectic form $\omega$ is exact.
Note that the exactness of $\omega$ implies that the symplectic manifold $(M,\omega)$ is an open manifold.
For such $(M,\omega)$, we recall that the \emph{Calabi homomorphism}
is a function $\mathrm{Cal}_{M} \colon \Ham(M,\omega)\to\mathbb{R}$ defined by
\[
	\mathrm{Cal}_{M}(\varphi_H)=\int_0^1\int_M H_t\omega^n\,dt,
\]
 where $H \colon [0,1] \times M \to \RR$ is a smooth function  with compact support. It is known that the Calabi homomorphism is a well-defined group homomorphism (see \cite{Cala,Ban,Ban97,MS,Hum}).
 %The following \emph{Calabi property} plays a key role in this paper.
Py's Calabi quasimorphism has the following property.

\begin{definition}[Calabi property with respect to quasimorphisms] \label{definition of Calabi property}
Let $(M,\omega)$ be a $2n$-dimensional  symplectic manifold.
Let $\mu \colon \Ham(M,\omega) \to \mathbb{R}$ be a homogeneous quasimorphism. A non-empty open subset $U$ of $M$ is said to have the \textit{Calabi property} with respect to $\mu$ if $\omega$ is exact on $U$ and if the restriction of $\mu$ to $\Ham(U,\omega)$ coincides with the Calabi homomorphism $\mathrm{Cal}_U$.
\end{definition}

In terms of subadditive invariants, the Calabi property corresponds to the asymptotically vanishing spectrum condition in \cite[Definition 3.5]{KO19}.

\begin{definition}[\cite{EP03,PR}]\label{Calabi prop}
Let $(M,\omega)$ be a $2n$-dimensional  symplectic manifold.
A \emph{Calabi quasimorphism} is defined as a homogeneous quasimorphism $\mu \colon \Ham(M,\omega) \to \RR$ such that every non-empty displaceable open exact subset of $M$ has the Calabi property with respect to $\mu$.
\end{definition}

%The first example of Calabi quasimorphisms was given by Entov and Polterovich \cite{EP03} by using the Hamiltonian Floer theory.

Let $\Sigma_l$ be a closed surface whose genus $l$ is at least $2$, and $\omega$ a volume form of $\Sigma_l$ such that the volume is $2l - 2$.
Choose a metric with constant negative curvature whose area form is $\omega$. Let $\pi \colon P \to \Sigma_l$ be the unit tangent bundle of $\Sigma_l$, $\mathbb{D}$ the Poincar\'e disc, and $S^1 \mathbb{D}$ the unit tangent bundle of $\mathbb{D}$. Set $S^1_\infty = \partial \mathbb{D}$. Define $p_\infty \colon S^1 \mathbb{D} \to S^1_\infty$ to be the map sending $v \in S^1 \mathbb{D}$  to the value at infinity  of the geodesic starting at $v$.
Let $X$ be the vector field generated by the $S^1$-action on $P$. There exists a  $S^1$-invariant  contact form on $P$ such that $\pi^* \omega = d \alpha$.
Then $X$ is the Reeb vector field of $\alpha$.
Let ${\rm Diff}(P, \alpha)$ denote the group of diffeomorphisms of $P$ preserving $\alpha$, and ${\rm Diff}(P, \alpha)_0$ its identity component.
Now we construct a homomorphism $\Theta \colon \Ham (\Sigma_l, \omega) \to {\rm Diff}_0(P, \alpha)$ as follows. Let $\varphi \in \Ham(\Sigma_l,\omega)$ and let $H\colon  [0,1] \times \Sigma_l  \to\RR $ be a function  such that $\varphi=\varphi_H$  and satisfying $\int_{\Sigma_l} H_t \omega = 0$ for every $t \in [0,1]$. Then set
\[ V_t = \widehat{X}_{H_t} + (H_t \circ \pi)X.\]
Here $\widehat{X}_{H_t}$ denotes the horizontal lift of  the Hamiltonian vector field $X_{H_t}$, \emph{i.e.}, $\alpha(\widehat{X}_{H_t}) = 0$ and $\pi_* (\widehat{X}_{H_t}) = X_{H_t}$.
Let  $\Theta(H)_t$ be the isotopy generated by $\{V_t\}_{t\in [0,1]}$, and set $\Theta(\varphi) = \Theta(H)_1$.
This is  known to be  well-defined.

Let $\widehat{\Theta(H)_t} \colon S^1 \mathbb{D} \to S^1 \mathbb{D}$ be the lift of $\Theta(H)_t$.
Define $\gamma^{(H,v)} \colon [0,1] \to S^1_\infty$ by $\gamma^{(H,v)}(t) = p_\infty (\widehat{\Theta(H)_t} (v))$.
Let $\widetilde{\gamma^{(H,v)}}$ be a lift to $\RR$ of $\gamma^{(H,v)}$ and set ${\rm Rot}(H,v) = \widetilde{\gamma^{(H,v)}}(1) - \widetilde{\gamma^{(H,v)}}(0)$. Let $\widetilde{\pi}$ denote the projection $S^1 \mathbb{D} \to \mathbb{D}$. For $\widetilde{x} \in \mathbb{D}$, set
\[ \widetilde{\rm angle} (H, \widetilde{x}) = - \inf_{\widetilde{\pi}(v) = \widetilde{x}} \lfloor {\rm Rot}(H,v) \rfloor.\]
Since $\widetilde{\rm angle}(H, -)$ is invariant by the action of $\pi_1(\Sigma_l)$, this induces a measurable bounded function ${\rm angle}(H, -)$ on $\Sigma_l$.
Define
\[ \mu_1(\varphi_H) = \int_{\Sigma_l} {\rm angle}(H, -) \omega.\]
Then $\mu_1$ is a quasimorphism and its homogenization is Py's Calabi quasimorphism $\qm_P$.

%%%%%%%%%%%%%%%%%%%%%%%%%%%%%%%%%%%%%%%%%%%%%%%%%%%%%%%%%%%%%%%%%%%%%%%%%%%%%%

\subsection{Entov--Polterovich's partial Calabi quasimorphism}\label{contrcution of EP}

We provide a sketch of the construction of Entov--Polterovich's partial Calabi quasimorphism which appeared in Section \ref{prehistory in symp}.
There are some convention of the Hamiltonian Floer theory and Entov--Polterovich's construction and
our convention is due to \cite{PR}.
For the precise definition of partial quasimorphism, see Definition \ref{qm rt nu}.

Let $(M,\omega)$ be a $2n$-dimensional closed symplectic manifold.
To avoid using the Novikov ring and the quantum homology, we assume that $\pi_2(M)=0$ and consider the coefficient $\ZZ/2\ZZ$ here
(for more general cases, for example see \cite{O15}, \cite{PR}).
First, we start from a rough explanation of the Hamiltonian Floer theory and Oh--Schwarz's spectral invariant.
For an introduction to spectral invariants through a toy model coming from the usual Morse theory, see \cite[Section 10]{PR}.

Let $\mathcal{L}_0M$ denote the space of contractible smooth loops $ S^1 \to M$ in $M$.
For a smooth function $H\colon S^1\times M\to\mathbb{R}$, the action functional $\mathcal{A}_H\colon \mathcal{L}_0M\to\mathbb{R}$ is given by
\[\mathcal{A}_H(z)=\int_0^1H(t,z(t))dt-\int_{D^2}u^\ast\omega,\]
where $u$ is a smoothly embedded disk in $M$ bounding $z$.
Since $d\omega=0$ and $\pi_2(M)=0$, $\mathcal{A}_H(z)$ does not depend on the choice of $u$.
Then we regard $\mathcal{P}(H)$ as the set of critical points of $\mathcal{A}_H$.
For a generic $H$, we can define the Floer chain complex $CF_\ast(H)$, which is generated by $\mathcal{P}(H)$ as a module over $\mathbb{Z}/2\ZZ$.
We formally obtain the boundary map of this complex by counting isolated negative gradient flow lines of $\mathcal{A}_H$ and we define its homology group $HF_\ast(H)$ which is called \textit{the Hamiltonian Floer homology} of $H$.
There exists a natural isomorphism $\Phi\colon \HHH_\ast(M;\ZZ/2\ZZ) \to HF_\ast(H)$.
We call this isomorphism the PSS isomorphism (\cite{PSS}).

Given a generic smooth function $H$ and an element $A=\sum_i a_i\cdot z_i$ of $CF_\ast(H)$, we define the action level $l_H(A)$ of $A$ by
\[l_H(A)=\max\{\mathcal{A}_H([z_i,u_i]);a_i\neq 0\}.\]
Given a generic smooth function $H$ and  a non-zero element $a$ of $\HHH_\ast(M;\ZZ/2\ZZ)$, we define the spectral invariant associated to $H$ and $a$ by
\[c(a,H)=\inf\{l_H(A);[A]=\Phi(a)\}\]
(if we remove the assumption $\pi_2(M)=0$, we should take $a$ as an element of the quantum homology).

The following proposition summarizes the properties of spectral invariants.

\begin{thm}[\cite{O05}, {\cite[(4) of Theorem 10.1]{O09}}]\label{OhSch}
Let $a \in \HHH_\ast(M;\ZZ/2\ZZ)$ be a non-zero element.
The spectral invariant $c(a, \cdot )$ has the following properties.

\begin{description}
\item[(1) \textit{Lipschitz property}]
The map $H{\mapsto}c(a,H)$ is Lipschitz on $C^\infty(S^1\times M)$ with respect to the $C^0$-norm.
\item[(2) \textit{Homotopy invariance}]
Assume that smooth functions $F,G\colon$ $S^1\times M\to\mathbb{R}$ are normalized, \emph{i.e.}, $\int_M F_t(x)\omega^m=0, \int_M G_t(x)\omega^m=0$ for every $t\in S^1$ and satisfy $\varphi^1_F=\varphi^1_G$ and that their Hamiltonian isotopies $\{\varphi^t_F\}$ and $\{\varphi^t_G\}$ are homotopic relative to endpoints. Then $c(a,F)=c(a,G)$.

\item[(3) \textit{Triangle inequality}]$c(a\ast{b},F\sharp G) \leq c(a,F)+c(b,G)$ for every smooth functions $F,G\colon S^1\times M\to\mathbb{R}$, where $\ast$ denotes the quantum product.
Here the smooth function $F{\sharp}G\colon S^1\times M\to\mathbb{R}$ is defined by $(F{\sharp}G)(t,x)=F(t,x)+G(t,(\varphi_F^t)^{-1}(x))$ whose Hamiltonian isotopy is $\{\varphi_F^t\varphi_G^t\}$.

\item[(4) \textit{Symplectic invariance}]
$c(\eta^\ast H; \eta^\ast a) = c(H; a)$ for every Hamiltonian $H$ and every $\eta \in \Symp(M,\omega)$.
Here $\eta^\ast H$ is defined to be $H \circ \eta \colon M\to \RR$.

\end{description}
\end{thm}

We defined the Oh--Schwarz spectral invariant for a generic $H$.
However, we define the Oh--Schwarz spectral invariant for a general $H$ by using (1) of Theorem \ref{OhSch}.

%$c(\eta^\ast H; \eta^\ast a) = c(H; a)$

For a non-zero element $a$ of $\HHH_\ast(M;\ZZ/2\ZZ)$ and $\tilde\varphi \in \tHam(M,\omega)$,
we can define the Oh--Schwarz invariant $c(a,\tilde\varphi)$ for  $\tilde\varphi \in \tHam(M,\omega)$ by
$c(a,\tilde\varphi)=c(a,H)$,
where $H$ is the normalized Hamiltonian whose isotopy $\{\varphi_H^t\}_{t\in [0,1]}$ represent $\tilde\varphi$ in $\tHam(M,\omega)$.
The value $c(a,\tilde\varphi)$ does not depend on the choice of $H$ by (2) of Theorem \ref{OhSch}.
Then, we define a function $\mu_a\colon \tHam(M,\omega) \to \RR$ by
\[\mu_a(\tilde\varphi) = -\int_M \omega^n \cdot \lim_{k\to \infty}\frac{c(a,\tilde\varphi^k)}{k}.\]
Entov and Polterovich \cite{EP06}, \cite{EP09} proved that $\mu_a$ is a partial quasimorphism for the fundamental class $[M]$. 

\begin{remark}
If we remove the assumption $\pi_2(M)=0$, then $\mu_a$ is a partial quasimorphism for other some $a$'s.
When $(M,\omega)=(\CC P^n,\omega_{FS})$, $\mu_{[M]}$ is a genuine homogeneous quasimorphism.
\end{remark}

To state the symplectic invariance of $\mu_a$, we explain how to formulate the symplectic invariance because $\tHam(M,\omega)$ may not be a subgroup of $\Symp(M,\omega)$.
For a Hamiltonian isotopy $\{\varphi_H^t\}_{t\in [0,1]}$ and $\eta \in \Symp(M,\omega)$, the conjugation of $\{\varphi_H^t\}_{t\in [0,1]}$ by $\eta \in \Symp(M,\omega)$ is defined to be $\{\eta^{-1}\varphi_H^t\eta\}_{t\in [0,1]}$.
This action induces the conjugation action of $\eta \in \Symp(M,\omega)$ on $\tHam(M,\omega)$ and $\mu_a$ is called \textit{$\Symp(M,\omega)$-invariant} if $\mu_a$ is invariant under this conjugation action.
We can define the $\Symp_0(M,\omega)$-invariance similarly.
We note that $\eta^\ast a = a$ for every $\eta \in \Symp_0(M,\omega)$ and $\eta^\ast [M] = [M]$ for every $\eta \in \Symp(M,\omega)$, where $[M]\in \HHH_\ast(M;\ZZ/2\ZZ)$ is the fundamental class.
Thus,  (4) of Theorem \ref{OhSch} implies the following proposition.
\begin{prop} The following hold true:
\begin{enumerate}
\item
For every non-zero element $a$ of $\HHH_\ast(M;\ZZ/2\ZZ)$ with $\eta^\ast a = a$ for every $\eta \in \Symp_0(M,\omega)$,
$\mu_a$ is $\Symp(M,\omega)$-invariant.
In particular, $\mu_{[M]}$ is $\Symp(M,\omega)$-invariant.
\item
For every non-zero element $a$ of $\HHH_\ast(M;\ZZ/2\ZZ)$,
$\mu_a$ is $\Symp_0(M,\omega)$-invariant.
\end{enumerate}
\end{prop}

In the proof of Example \ref{EPeg}, we can find a good $a$ such that $\eta^\ast a = a$ for every $\eta \in \Symp_0(M,\omega)$.
In the proof of Example \ref{unioneg}, we use the fundamental class
\footnote{We note that in Examples \ref{EPeg} and \ref{unioneg}, the assumption $\pi_2(M)=0$ is not satisfied and $a$ should be an element of the quantum homology.}
.

Finally, we propose the following problem.
\begin{openproblem}
Is a partial quasimorphism used in the proof of Example $\ref{EPeg}$ extendable to $\widetilde{\Symp}_0(M,\omega)$?
\end{openproblem}

%Before explaining the $\Symp(M,\omega)$-invariance of $\mu_a$, we explain how to define $\Symp(M,\omega)$-invariance.

%If $\eta^\ast=a$,
%then $\mu_a$ is $\Symp(M,\omega)$-invariant.

%Next, we explain a reason why symplecitc invariance of Oh--Schwarz spectral invariant ((4) of Theorem \ref{OhSch}) holds.

%\[A_H (z) = A_{\eta^\ast H} (\eta^\ast z),\].
%where $\eta^\ast z=z\circ \eta \colon S^1 \to M$.

%problem on
%extendability

\subsection{Invariant quasimorphisms for groups acting on the circle} \label{subsec:rotation_invq}

The following construction of invariant quasimorphisms from group actions on the circle was given in \cite{MMM}.

To state the construction, we briefly review the definition of group cohomology (see \cite{brown82} for details).
Let $C^n(G)$ be the space of (inhomogeneous) $n$-cochains $G$, \emph{i.e.}, the space of real-valued functions on the $n$-fold direct product $G^n$ of $G$.
Here the space of $0$-cochains is understood as $C^0(G) = \RR$.
Define the coboundary map $\delta^n \colon C^n(G) \to C^{n+1}(G)$ by
\begin{multline*}
  \delta^n(c)(g_1, \cdots, g_{n+1}) = c(g_2, \cdots, g_{n+1})  \\ 
  + \sum_{i=1}^{n} (-1)^i c(g_1, \cdots, g_i g_{i+1}, g_{n+1})+ (-1)^{n+1} c(g_1, \cdots, g_n).  
\end{multline*}
%\[ \delta^n(c)(g_1, \cdots, g_{n+1}) = c(g_2, \cdots, g_{n+1}) +
%\sum_{i=1}^{n} (-1)^i c(g_1, \cdots, g_i g_{i+1}, g_{n+1}) + (-1)^{n+1} c(g_1, \cdots, g_n)\]
for $n>0$ and $\delta^0 \colon C^0(G) \to C^1(G)$ by the zero-map.
\begin{comment} % for submit
\begin{multline*}
  \delta^n(c)(g_1, \cdots, g_{n+1}) = c(g_2, \cdots, g_{n+1})  \\ + \sum_{i=1}^{n} (-1)^i c(g_1, \cdots, g_i g_{i+1}, g_{n+1})+ (-1)^{n+1} c(g_1, \cdots, g_n).
\end{multline*}
\end{comment}
The cohomology group of the cochain complex $(C^n(G), \delta^n)$ is the ordinary cohomology group $\HHH^n(G)$.

Let $\h$ be the group of orientation preserving homeomorphisms of the circle and $\th$ its universal covering group.
We regard the group $\th$ as
\[
  \{ f \in \Homeo_+(\RR) \colon f(x + 1) = f(x) + 1 \text{ for every } x \in \RR \}.
\]
The \emph{Poincar\'{e} translation number} (\cite{poincare81}) is the map $\trot \colon \th \to \RR$ defined by
\[
  \trot(f) = \lim_{n \to \infty} \frac{f^n(0)}{n}.
\]
This is well-defined and turns out to be a homogeneous quasimorphism on $\th$ (see \cite{MR1876932}).
Let $\chi_b$ be the group two-cocycle on $\h$ defined by
\[
  \chi_b(f,g) = \trot(\widetilde{f} \widetilde{g}) - \trot(\widetilde{f}) - \trot(\widetilde{g})
\]
for every $f, g \in \h$ and their lifts $\widetilde{f}, \widetilde{g} \in \th$.
This cocycle was introduced in \cite{MR848896} and is called the \emph{canonical Euler cocycle}, which represents the \emph{real Euler class $e_{\RR}$ of $\HHH^2(\h)$}.

Let $1 \to N \xrightarrow{i} G \xrightarrow{p} \Gamma \to 1$ be an exact sequence and $\rho \colon G \to \h$ a homomorphism.
Assume that the pullback $\rho^*e_{\RR} \in \HHH^2(G)$ is contained in the image of $p^* \colon \HHH^2(\Gamma) \to \HHH^2(G)$.
By this assumption, there exist a normalized cocycle $A \in C^2(\Gamma)$ and a cochain $u \in C^2(G)$ such that
\[
  \rho^* \chi_b - p^* A = \delta u.
\]
Here a cocycle $A \in C^2(\Gamma)$ is said to be \emph{normalized} if $A(\gamma_1, \gamma_2) = 0$ whenever at least one of $\gamma_1$ and $\gamma_2$ is the group unit of $\Gamma$.
We set $\nu_{\rho, A, u} = u|_{N}$.
Then the following holds.
\begin{thm}[{\cite{MMM}}]\label{thm:rotation_nonext}
  The map $\nu_{\rho, A, u} \colon N \to \RR$ is a $G$-invariant homogeneous quasimorphism on $N$.
  Moreover, if $\HHH_b^2(\Gamma) = 0$ and the pullback $\rho^*e_{\RR}$ is non-zero, then $\nu_{\rho, A, u}$ is a non-zero element in $\WWW(G, N)$.
\end{thm}

If $G$ is the fundamental group $\pi_1(\Sigma_l)$ of a closed orientable surface of genus $l \geq 2$ and $N$ is its commutator subgroup $[G,G]$ for instance, then every Fuchsian action $\rho \colon G \to \h$ satisfies the assumption in Theorem \ref{thm:rotation_nonext}.
Hence $\nu_{\rho, A, u}$ gives rise to a non-zero element of $\WWW(G, N)$.
Note that the space $\WWW(G, N)$ is spanned by this invariant quasimorphism $\nu_{\rho, A, u}$ since the dimension of $\WWW(G, N)$ is equal to one (Theorem \ref{raretu of keisan} (2)).

%%%%%%%%%%%%%%%%%%%%%%%%%%%%%%%%%%%%%%%%%%%%%%%%%%%%%%%%%%%%%%%%%%%%%%%%%%%%%%%%%%%%%%%%%%%%%%%%%%%%%%%%%%%%%%%%%%%%%%%%%%%%%%%%%
%%%%%%%%%%%%%%%%%%%%%%%%%%%%%%%%%%%%%%%%%%%%%%%%%%%%%%%%%%%%%%%%%%%%%%%%%%%%%%%%%%%%%%%%%%%%%%%%%%%%%%%%%%%%%%%%%%%%%%%%%%%%%%%%%
%%%%%%%%%%%%%%%%%%%%%%%%%%%%%%%%%%%%%%%%%%%%%%%%%%%%%%%%%%%%%%%%%%%%%%%%%%%%%%%%%%%%%%%%%%%%%%%%%%%%%%%%%%%%%%%%%%%%%%%%%%%%%%%%%
%%%%%%%%%%%%%%%%%%%%%%%%%%%%%%%%%%%%%%%%%%%%%%%%%%%%%%%%%%%%%%%%%%%%%%%%%%%%%%%%%%%%%%%%%%%%%%%%%%%%%%%%%%%%%%%%%%%%%%%%%%%%%%%%%

\section{Aut-invariant quasimorphisms}\label{aut section}

A quasimorphism $\phi$ on $G$ is said to be \emph{$\Aut$-invariant} if $\phi(\sigma(g))=\phi(g)$ for every $g \in G$ and every $\sigma \in \mathrm{Aut}(G)$.
%The initial motivation to study Aut-invariant quasimorphisms comes from the study of Aut-invariant word norms.
In this section, we review results on Aut-invariant quasimorphisms and discuss the relation to invariant quasimorphisms.

 In \cite[Question 47]{Abert}, Ab\'{e}rt asked the question whether free groups $F_n$ ($n\geq 2$) admit a non-trivial Aut-invariant quasimorphism.
According to Hase's result \cite{Hase}, the space of quasimorphisms on $F_n$ that are not Aut-invariant is dense in $\QQQ(F_n)$ (with respect to the topology of pointwise convergence). Nevertheless,
Brandenbursky and Marcinkowski proved the following,
 in particular, giving an answer to Ab\'{e}rt's question for $n=2$ in the affirmative.

  \begin{thm}[\cite{BM}] \label{thm:BM_aut}
The space of $\Aut$-invariant homogeneous quasimorphisms on $F_2$
 is infinite-dimensional.
  \end{thm}

 Karlhofer proved that the space of Aut-invariant homogeneous quasimorphisms on $G$ is infinite-dimensional if
\begin{itemize}
  \item $G = A \ast B$ is the free product of two non-trivial freely indecomposable
  groups $A$ and $B$, which are not the infinite dihedral group \cite{Karlh},
  \item $G = G_1 \ast \cdots \ast G_k$ is a free product of freely indecomposable groups $G_i$, under the assumption that at most two factors are infinite cyclic and there exists $j \in \{1, \dots, k\}$ such that $G_j \not\cong \ZZ / 2\ZZ$   \cite{Karlh2}.
\end{itemize}

  Recently, Fournier-Facio and Wade proved the following theorem that recovers the above results
   and settles Ab\'{e}rt's question above.
  \begin{thm}[\cite{FW}]\label{thm:autqm}
    Let $G$ be a group, and assume that one of the following holds:
\begin{enumerate}[$(1)$]
    \item $G$ is a non-elementary Gromov-hyperbolic group;
    \item $G$ is hyperbolic relative to a collection of subgroups $\mathcal{P}$, and no group in $\mathcal{P}$ is relatively hyperbolic;
    \item $G$ has infinitely many ends;
    \item $G$ is a graph product of finitely generated abelian groups on a finite graph, and $G$ is not virtually abelian.
\end{enumerate}
Then the space of $\Aut$-invariant homogeneous quasimorphisms on $G$ is infinite-dimensional.
  \end{thm}

One motivation to study Aut-invariant quasimorphisms comes from the study of \emph{Aut-invariant word norms} (\emph{i.e.}, word norms on Aut-invariant generating sets).
Aut-invariant quasimorphisms provide lower bounds for Aut-invariant word norms: if a group $G$ admits a non-trivial Aut-invariant homogeneous quasimorphism, and $S \subset G$ is the closure of a finite set under the action of $\Aut(G)$ such that $S$ generates $G$, then the Aut-invariant word norm with respect to $S$ on $G$ is unbounded. Aut-invariant word norms have been studied for several natural Aut-invariant generating sets, such as the set of primitive words in free groups and the set of simple loops in surface groups. See \cite{BM,FW} for more information.

  Aut-invariant quasimorphisms are related to invariant quasimorphisms in two ways as follows.
  We have the following two exact sequences for $\mathrm{Aut}(G)$:
  \begin{align}
     1 \to \mathrm{Inn}(G) \to \mathrm{Aut}(G) \to \mathrm{Out}(G) \to 1, \label{eq:out} \\
     1 \to G \to G \rtimes \mathrm{Aut}(G) \to \mathrm{Aut}(G) \to 1. \label{eq:semi}
  \end{align}
  The first way is according to \eqref{eq:out}.
  If $G$ has trivial center (\emph{i.e.}, $Z(G)=1$), we can regard an Aut-invariant quasimorphism as an $\mathrm{Aut}(G)$-invariant quasimorphism since $\mathrm{Inn}(G)\cong G/Z(G)$.
  The second way is according to \eqref{eq:semi}.
  We can regard an Aut-invariant homogeneous quasimorphism as an $(G \rtimes \mathrm{Aut}(G))$-invariant quasimorphism.

Since sequence \eqref{eq:semi} splits, every Aut-invariant quasimorphism on $G$ is extendable to $G \rtimes \mathrm{Aut}(G)$.
In case of $Z(G)=1$, we can consider the extension problem for Aut-invariant quasimorphisms in terms of \eqref{eq:out} and we will discuss it in Section \ref{sec:ext_prob}.

As invariant quasimorphisms are related to mixed commutator length, Aut-invariant quasimorphisms are related to \textit{autocommutator length}, which is the word length with respect to the set of \textit{autocommutators}
$\{ \sigma(g)g^{-1}  \mid g \in G, \sigma \in \Aut(G) \}$.
Autocommutator length is an interesting example of mixed commutator length associated with the specific pair $(G,N)=(G\rtimes \Aut(G), G)$.
See \cite[Section 3]{FW} for more details on the relationship between autocommutator length and mixed commutator lengths.

%Let $(M,\omega)$ be a closed symplectic manifold and
%define a homomorphism $T(\varphi) \colon \Symp(M,\omega) \to \mathrm{Aut}(\Ham(M,\omega))$ by $\varphi \mapsto (\psi \mapsto \varphi\psi\varphi^{-1})$.
%Then, we see that this homomorphism is an isomorphism  by combining Banyaga's several results (Theorem 2, Lemma 5 of \cite{Ban88} and Theorem 1 of \cite{Ban86}).
%Thus, we can regard the non-extendabily of Py's Calabi quasimorphism in \cite{KK} as the non-extendabilty of $\mathrm{Aut}$-invariant quasimorphism.

%%%%%%%%%%%%%%%%%%%%%%%%%%%%%%%%%%%%%%%%%%%%%%%%%%%%%%%%%%%%%%%%%%%%%%%%%%%%%%%%%%%%%%%%%%%%%%%%%%%%%%%%%%%%%%%%%%%%%%%%%%%%%%%%%
%%%%%%%%%%%%%%%%%%%%%%%%%%%%%%%%%%%%%%%%%%%%%%%%%%%%%%%%%%%%%%%%%%%%%%%%%%%%%%%%%%%%%%%%%%%%%%%%%%%%%%%%%%%%%%%%%%%%%%%%%%%%%%%%%
%%%%%%%%%%%%%%%%%%%%%%%%%%%%%%%%%%%%%%%%%%%%%%%%%%%%%%%%%%%%%%%%%%%%%%%%%%%%%%%%%%%%%%%%%%%%%%%%%%%%%%%%%%%%%%%%%%%%%%%%%%%%%%%%%
%%%%%%%%%%%%%%%%%%%%%%%%%%%%%%%%%%%%%%%%%%%%%%%%%%%%%%%%%%%%%%%%%%%%%%%%%%%%%%%%%%%%%%%%%%%%%%%%%%%%%%%%%%%%%%%%%%%%%%%%%%%%%%%%%

\section{On the proof of invariant Bavard's duality} \label{prf of bavard}

In this section we provide an outline of the proof of Bavard's duality theorem of invariant quasimorphisms (Theorem~\ref{thm:bavard}) in \cite{KKMM1}.
%Although the proof is a generalization to the original proof by Bavard \cite{Bav}, we introduce the geometric interpretation of mixed commutator lengths.
One of the novel points of our proof from the proof of the original Bavard theorem (Theorem~\ref{original Bavard}) by Bavard is to study the geometric interpretation of mixed commutator length.

\subsection{$\nu_N$-quasimorphism}\label{Bavard N qm}

Here we recall some terminology related to the group homology. Let $C_n(G) = C_n(G; \RR)$ be the space of $n$-chains of $G$ with real coefficients, i.e., the $\RR$-module freely generated by the $n$-fold direct product of $G$. Here we regard $C_0(G)$ as $\RR$. Define the boundary map $\partial_n \colon C_n(G) \to C_{n-1}(G)$ by
\[ \partial_n(g_1, \cdots, g_n) = (g_2, \cdots, g_n) + \sum_{i = 1}^{n-1} (-1)^i (g_1, \cdots, g_i g_{i+1}, \cdots, g_n) + (-1)^n(g_1, \cdots, g_{n-1})\]
for $n \ge 2$ and set $\partial_1 = 0$. Let $Z_n(G)$ denote the kernel of $\partial_n$ and $B_n(G)$ the image of $\partial_{n+1}$.

One of the key observations of the proof of the original Bavard duality theorem \cite{Bav} is to regard the quotient space $\widehat{\QQQ}(G) / \HHH^1(G)$ as the continuous dual of $B_1(G) \cong C_2(G) / Z_2(G)$. Here $\widehat{\QQQ}(G)$ denotes the vector space consisting of (not necessarily homogeneous) quasimorphisms on $G$. The notion which plays a similar role to $\widehat{\QQQ}(G)$ in our setting is the \emph{$\nu_N$-quasimorphism} (see Example \ref{Nqm}).

%Before providing the definition of $\nu_N$-quasimorphisms,
Before explaining an important property of $\nu_N$-quasimorphism, we introduce the following notion which is a relaxation of invariant quasimorphisms. A quasimorphism $\phi \colon N \to \RR$ is \emph{$G$-quasi-invariant} if there exists $D' \ge 0$ such that for every $g \in G$ and for every $x \in N$,
\[ |\phi(gxg^{-1}) - \phi(x)| \le D'.\]

%Then, we now introduce the notion of $\nu_N$-quasimorphisms.
%An \emph{$\nu_N$-quasimorphism} is a function $\phi \colon G \to \RR$ such that there exists $D'' \ge 0$ such that for every $g \in G$ and for every $x \in N$, the following inequalities hold:
%\[ |\phi(gx) - \phi(g) - \phi(x)| \le D'' \quad \textrm{and} \quad |\phi(xg) - \phi(x) - \phi(g)| \le D''.\]
%The smallest non-negative real number $D''$ satisfying the above two inequalities is called the \emph{defect} of $\phi$, and is denoted by $D''(\phi)$. We write $\widehat{\QQQ}_N(G)$ to indicate the vector space consisting of $\nu_N$-quasimorphisms on $G$. Then the following hold:

\begin{lem}[{\cite[Lemma~2.3 and Proposition~2.4]{KKMM1}}]
The following hold:
\begin{enumerate}[$(1)$]
\item If $f \colon G \to \RR$ is an $\nu_N$-quasimorphism, then $f|_N$ is a $G$-quasi-invariant quasimorphism.

\item For every $G$-quasi-invariant quasimorphism $f \colon N \to \RR$, there exists an $\nu_N$-quasimorphism $F \colon G \to \RR$ such that $F|_N = f$.
\end{enumerate}
\end{lem}

We call an $\nu_N$-quasimorphism $f$ an \emph{$\nu_N$-homomorphism} if $D''(f) = 0$. We write $\HHH^1_N(G)$ to mean the vector space consisting of $\nu_N$-homomorphisms on $G$.  Then the defect $D''$ induces a (genuine) norm on $\widehat{\QQQ}_N(G) / \HHH^1_N(G)$.

\subsection{Filling norm}

Let $C_2'(G)$ be the $\RR$-submodule of $C_2(G)$ generated by the set
\[ \{ (g_1, g_2) \in G \times G \; | \; \textrm{$N$ contains either $g_1$ or $g_2$}\}.\]
Set $B'_1 = \partial C'_2$ and $Z'_2 = Z_2(G ; \RR) \cap C_2'$.
Endow $C_2'$ with the $\ell^1$-norm. Then $Z_2'$ is a closed subspace of $C_2'$, and we consider $B'_1 \cong C_2' / Z_2'$ as a quotient normed space. Let $\| \cdot \|'$ denote the norm of $B'_1$.

\begin{prop}[{\cite[Proposition~3.5]{KKMM1}}]
The space $(\widehat{\QQQ}_N(G) / \HHH^1_N(G),D'')$ is isometrically isomorphic to the continuous dual of $(B_1',\| \cdot \|')$.
%$\widehat{\QQQ}_N(G) / \HHH^1_N(G)$ and the dual of $(C'_2 / Z'_2)^*$ are isometric.
In particular, $(\widehat{\QQQ}_N(G) / \HHH^1_N(G), D'')$ is a Banach space.
\end{prop}

%Then the following hold:
We can show the following properties on $\|\cdot\|'$.
\begin{enumerate}[$(1)$]
\item (\cite[Lemma~3.1]{KKMM1}) If $g \in G$ and $x \in N$, then $[g,x] \in B'_1$ and $\| [g,h]\|' \le 3$. Indeed,
\[ \partial ([g,x], xg) - \partial (g,x) + \partial (x,g) = [g,x].\]

\item (\cite[Lemma~3.2]{KKMM1}) If $x,y \in [G,N]$, then $x,y,xy \in B'_1$ and $\| xy\|' \le \| x\|' + \| y\|' + 1$.
\end{enumerate}
Property (2) implies that for $x \in [G,N]$ the sequence $\{ \| x^n\|' + 1\}_n$ is subadditive, and hence the limit
\[ \mathrm{fill}_{G,N}(x) = \lim_{n \to \infty} \frac{\| x^n\|'}{n}\]
exists. We call $\mathrm{fill}_{G,N}(x)$ the \emph{$(G,N)$-filling norm of $x$}. As is the case of the proof of the original Bavard duality \cite{Ca}, Theorem~\ref{thm:bavard} is deduced from the equality
\[ \mathrm{fill}_{G,N}(x) = 4 \cdot \scl_{G,N}(x).\]
By the same argument as the usual case, it is straightforward to show
\[ \mathrm{fill}_{G,N}(x) \le 4 \cdot \scl_{G,N}(x).\]
The main part of our proof of Theorem~\ref{thm:bavard} is to prove the other inequality. This is obtained by the geometric interpretation of mixed commutator lengths, which will be described in the next subsection.

\subsection{Geometric interpretation of $\cl_{G,N}$}

In the original proof of Bavard's duality theorem, Bavard used the geometric interpretation of the commutator lengths described as follows. Note that for an element $x \in [G,G]$, a map $f \colon S^1 \to BG$ representing the homotopy class $x \in G \cong \pi_1(BG)$ can be extended to a compact orientable surface with boundary $S^1$. Here $BG$ denotes the classifying space of $G$.
The commutator length of $x \in [G,G]$ in $G$ equals the minimum genus of compact orientable surface $\Sigma$ with boundary $S^1$ such that there exists a continuous map $F \colon \Sigma \to BG$ satisfying that the homotopy class of $F|_{\partial \Sigma}$ coincides with $x \in G \cong \pi_1(BG)$.

In the rest of this section we assume surfaces to be compact and orientable. A \emph{simplicial surface} is a $\Delta$-complex (see \cite{Hatcher} or \cite{KKMM1} for the precise definition of the $\Delta$-complex) which is homeomorphic to a surface. For a simplicial surface $\Sigma$, let $E(\Sigma)$ be the set of edges of $\Sigma$, and $T(\Sigma)$ the set of triangles of $\Sigma$. Then every triangle $\sigma$ in $\Sigma$ is surrounded by the edges $\partial_0 \sigma$, $\partial_1 \sigma$ and $\partial_2 \sigma$ as is depicted in Figure~\ref{fig 1}.

A \emph{$G$-labelling} on a simplicial surface $\Sigma$ is a function $f \colon E(\Sigma) \to G$ such that $f(\partial_1 \sigma) = f(\partial_2 \sigma) \cdot f(\partial_0 \sigma)$ for every $\sigma \in T(\Sigma)$. A \emph{$(G,N)$-labelling} on $\Sigma$ is a $G$-labelling on $\Sigma$ such that for every $\sigma \in T(\Sigma)$ either $f(\partial_0 \sigma)$ or $f(\partial_2 \sigma)$ belongs to $N$. A $(G,N)$-simplicial surface is a simplicial surface together with its $(G,N)$-labelling.

\begin{definition}
Let $x \in G$. A \emph{$(G,N)$-simplicial surface with boundary $x$} is a simplicial surface $\Sigma$ with a $(G,N)$-labelling $f$ satisfying that the boundary $\partial \Sigma$ consists of one vertex and one edge $e$, and that $f(e) = x$.
\end{definition}

The following is the geometric interpretation of the mixed commutator length:

\begin{prop}[{\cite[Proposition~4.8]{KKMM1}}] \label{prop interpretation}
Let $N$ be a normal subgroup of a group $G$, $n$ a non-negative integer, and let $x \in G$. Then the following are equivalent:
\begin{enumerate}[$(1)$]
\item There exist $n$ $(G,N)$-commutators $c_1$, $\cdots$, $c_n$ whose product $c_1 \cdots c_n$ is $x$.

\item There exists a $(G,N)$-simplicial surface with boundary $x$ whose genus is $n$.
\end{enumerate}
In particular, if there exists a $(G,N)$-simplicial surface with boundary $x$, then $x$ belongs to $[G,N]$.
\end{prop}

Note that Proposition~\ref{prop interpretation} is a generalization of the geometric interpretation of the ordinary commutator length. Indeed, if there exists a $G$-labelling $f$ on a simplicial surface $\Sigma$ with boundary $x$, then we can construct a continuous map $F \colon \Sigma \to BG$ such that the homotopy class of $f|_{\partial \Sigma}$ coincides with $x$ as follows: first, $F$ sends every vertex of $\Sigma$ to the basepoint of $BG$. For every edge $e$ of $\Sigma$, we define $F|_e$ so that $F|_e$ represents the homotopy class of $f(e) \in G \cong \pi_1(BG)$. Now we have constructed the continuous map from the $1$-skeleton of $\Sigma$. By the equation $f(\partial_1 \sigma) = f(\partial_2 \sigma) \cdot f(\partial_0 \sigma)$ for every triangle $\sigma$ in $\Sigma$, this continuous map from the $1$-skeleton of $\Sigma$ extends to the whole surface $\Sigma$.  Conversely, if $F \colon \Sigma \to BG$ is a continuous map such that the homotopy class of $F|_{\partial \Sigma}$ coincides with $x$, then it is straightforward to construct a $G$-labelling on $\Sigma$ with boundary $x$.

The number of triangles of $(G,N)$-simplicial surface whose boundary is $x$ is related to the filling norm of $x$, and we obtain
\[ \textrm{fill}_{G,N} = 4 \cdot \scl_{G,N}(x) \]
for every $x \in [G,N]$. By this, we have the Bavard duality theorem for invariant quasimorphisms (Theorem~\ref{thm:bavard}).

\begin{figure}[t]
\begin{picture}(100,100)(0,0)
\put(0,40){\vector(3,-2){60}}
\put(0,40){\vector(3,1){96}}
\put(60,0){\vector(1,2){36}}

\put(49,32){$\sigma$}

\put(83,25){$\partial_0 \sigma$}
\put(16,6){$\partial_2 \sigma$}
\put(33,64){$\partial_1 \sigma$}

%\put(25,-15){\bf Figure 1.}
\end{picture}
\caption{} \label{fig 1}
\end{figure}

%%%%%%%%%%%%%%%%%%%%%%%%%%%%%%%%%%%%%%%%%%%%%%%%%%%%%%%%%%%%%%%%%%%%%%%%%%%%%%%%%%%%%%%%%%%%%%%%%%%%%%%%%%%%%%%%%%%%%%%%%%%%%%%%%
%%%%%%%%%%%%%%%%%%%%%%%%%%%%%%%%%%%%%%%%%%%%%%%%%%%%%%%%%%%%%%%%%%%%%%%%%%%%%%%%%%%%%%%%%%%%%%%%%%%%%%%%%%%%%%%%%%%%%%%%%%%%%%%%%
%%%%%%%%%%%%%%%%%%%%%%%%%%%%%%%%%%%%%%%%%%%%%%%%%%%%%%%%%%%%%%%%%%%%%%%%%%%%%%%%%%%%%%%%%%%%%%%%%%%%%%%%%%%%%%%%%%%%%%%%%%%%%%%%%
%%%%%%%%%%%%%%%%%%%%%%%%%%%%%%%%%%%%%%%%%%%%%%%%%%%%%%%%%%%%%%%%%%%%%%%%%%%%%%%%%%%%%%%%%%%%%%%%%%%%%%%%%%%%%%%%%%%%%%%%%%%%%%%%%

\section{The space of non-extendable quasimorphisms} \label{sp sec}

In this section, we explain the results in \cite{K2M3} related to the vector spaces $\VVV(G, N)$ and $\WWW(G,N)$.
%related to non-extendable quasimorphisms.
These spaces are central objects of the extension problem of quasimorphisms. We first recall the notation and terminology introduced in Section 1. Let $N$ be a normal subgroup of a group $G$, and $\QQQ(N)^G$ the vector space of $G$-invariant homogeneous quasimorphisms on $N$. Let $i^\ast$ denote the restriction map $\QQQ(G) \to \QQQ(N)^G, \phi \mapsto \phi |_N$. Then set
\[ \VVV(G, N) := \QQQ(N)^G / i^\ast \QQQ(G), \quad \WWW(G,N) := \QQQ(N)^G / (\HHH^1(N)^G + i^\ast \QQQ(G)).\]
Note that $\VVV(G, N) = 0$ implies that every $G$-invariant homogeneous quasimorphism on $N$ is extendable to $G$. Hence Theorem~\ref{virtually split} implies that if the projection $p \colon G \to \Gamma$ has a virtual section, then $\VVV(G,N) = 0$. Contrastingly, Shtern's work (Example~\ref{eg:shtern}) and the work (Theorem~\ref{thm:py}) by some of the authors
%, which are explained in Section \ref{history after}
are reinterpretated as the first works of proving non-vanishing of $\VVV(G,N)$ and $\WWW(G,N)$, respectively.

To state the results in \cite{K2M3}, we review the definition of bounded cohomology. For a comprehensive introduction to this subject, see \cite{Fr}.
%Here we review bounded cohomology. See \cite{Ca,Fr} for more information.
\begin{comment}
Let $C^n(G)$ denote the space of (inhomogeneous) $n$-cochains $G$, \emph{i.e.}, the space of real-valued functions on the $n$-fold direct product $G^n$ of $G$. Define the coboundary map $\delta^n \colon C^n(G) \to C^{n+1}(G)$ by
\[ \delta^n(c)(g_1, \cdots, g_{n+1}) = c(g_2, \cdots, g_{n+1}) +
\sum_{i=1}^{n} (-1)^i c(g_1, \cdots, g_i g_{i+1}, g_{n+1}) + (-1)^{n+1} c(g_1, \cdots, g_n).\]
\begin{comment} % for submit
\begin{multline*}
  \delta^n(c)(g_1, \cdots, g_{n+1}) = c(g_2, \cdots, g_{n+1})  \\ + \sum_{i=1}^{n} (-1)^i c(g_1, \cdots, g_i g_{i+1}, g_{n+1})+ (-1)^{n+1} c(g_1, \cdots, g_n).
\end{multline*}
end{comment}
The cohomology group of the cochain complex $(C^n(G), \delta^n)$ is the ordinary cohomology group $\HHH^n(G)$. Let $C^n_b(G)$ be the subspace of $C^n(G)$ consisting of bounded functions on $G^n$.
\end{comment}

Let $C^n_b(G)$ be the subspace of $C^n(G)$ consisting of bounded functions on $G^n$  (the definition of $C^n(G)$ is given in Subsection \ref{subsec:rotation_invq}).
Then the subspaces $C^n_b(G)$ form a subcomplex of $C^n(G)$, and its cohomology group is the \emph{bounded cohomology group $\HHH_b^n(G)$ of $G$}.
%$c \colon G^n  \to \mathbb{R}$ and $\delta \colon C^n(G) \to C^{n+1}(G)$ the coboundary map.
%If we consider the subcomplex
%\[C^n_b (G)=\{c \in C^n(G) \mid c \; \text{is bounded} \}\]
%of $C^n(G)$, the homology of the complex $(C^n_b (G), \delta)$ is called the \emph{bounded cohomology} of $G$ and denoted by $\HHH^n_b(G)$.
The natural inclusion $C^n_b(G) \to C^n(G)$ induces the homomorphism
$c^n_G \colon \HHH^n_b(G) \to \HHH^n(G)$, which is called the \emph{comparison map} of degree $n$.

For a positive integer $k$, a group $G$ is said to be \emph{boundedly $k$-acyclic} if $\HHH^i_b(G) = 0$ for all $1 \le i \le k$.
A group $G$ is \emph{boundedly acyclic} if $G$ is boundedly $k$-acyclic for every $k$.
%Amenable groups are typical examples of boundedly acyclic groups.
We mainly care about the case $k \geq 3$ for the following two reasons.
Firstly, computing higher-order bounded cohomology is more challenging compared to second bounded cohomology, which is considerably better understood due to its connection with quasimorphisms.
Secondly, the assumption of boundedly 3-acyclicity is used for the results on the spaces of non-extendable quasimorphisms that we will discuss later (Theorems \ref{calculation of dim}, \ref{easy cor} and Corollary \ref{cor:fin_dim}).
We collect known results on bounded $k$-acyclicity for $k\geq 3$.

\begin{thm}[Known results for boundedly acyclic groups]\label{thm=bdd_acyc}
The following hold.
\begin{enumerate}[\textup{(}$1$\textup{)}]
\item \textup{(}\cite{Gr}\textup{)} Every amenable group is boundedly acyclic.
\item \textup{(}\cite{MatsumotoMorita}\textup{)}  Let $n\in \NN$. Then, the group $\Homeo_{c}(\RR^n)$ of homeomorphisms on $\RR^n$ with compact support is boundedly acyclic.

\item \textup{(}combination of \cite{Mon04} and \cite{MonodShalom}\textup{)}
For $n \geq 3$, every lattice in ${\rm SL}(n,\RR)$ is $3$-boundedly acyclic.

\item \textup{(}\cite{BucherMonod}\textup{)} Burger--Mozes groups \cite{BurgerMozes} are $3$-boundedly acyclic.

\item \textup{(}\cite{MR21}\textup{)} Let $k\in \NN$. Let $1\to N \to G \to \Gamma \to 1$ be a short exact sequence of groups. Assume that $N$ is boundedly $k$-acyclic. Then $G$ is boundedly $k$-acyclic if and only if $\Gamma$ is.
\item \textup{(}\cite{FLM1}\textup{)} Every binate group \textup{(}see \cite[Definition~3.1]{FLM1}\textup{)} is boundedly acyclic.
\item \textup{(}\cite{FLM2}\textup{)} There exist continuum many non-isomorphic $5$-generated non-amenable groups that are boundedly acyclic. There exists a finitely presented non-amenable group that is boundedly acyclic.
\item \textup{(}\cite{Monod2021}\textup{)} Thompson's group $F$ is boundedly acyclic.
\item \textup{(}\cite{Monod2021}\textup{)} Let $L$ be an arbitrary group. Let $\Gamma$ be an infinite amenable group. Then the wreath product $L\wr \Gamma =\left( \bigoplus_{\Gamma}L\right)\rtimes \Gamma$ is boundedly acyclic.

\item \textup{(}\cite{MN21}\textup{)} For every integer $n$ at least two, the identity component $\Homeo_0(S^n)$ of the group of orientation-preserving homeomorphisms of $S^n$ is boundedly $3$-acyclic. The group $\Homeo_0(S^3)$ is boundedly $4$-acyclic.
\end{enumerate}
\end{thm}

Note that the class of amenable groups includes  all finite groups and all abelian groups; this class  is closed under taking subgroups, group quotients, inductive limits and extensions. In particular, all (virtually) solvable groups are amenable, and hence boundedly acyclic by (1) of Theorem \ref{thm=bdd_acyc}.

The following theorems show that the spaces $\VVV(G, N)$ and $\WWW(G,N)$ are not so large when the quotient group $\Gamma = G/N$ is boundedly $3$-acyclic, although $\QQQ(N)^G$ and $i^* \QQQ(G)$ can be infinite-dimensional (Remark \ref{rmk:acyl_hyp}).

\begin{thm}[{\cite[Theorem~1.9]{K2M3}}] \label{calculation of dim}
If the quotient group $\Gamma=\hG/\bG$ is boundedly $3$-acyclic, then
\[\dim  \VVV(G,N)  \leq \dim \HHH^2(\Gamma).\]
Moreover, if $G$ is Gromov-hyperbolic, then
\[\dim \VVV(G,N) = \dim \HHH^2(\Gamma).\]
\end{thm}

\begin{thm}[{\cite[Theorem~1.10]{K2M3}}] \label{easy cor}
%We assume that $\HHH^i_b(G/N)=0$ for $i=2,3$.
If $\Gamma=G/N$ is boundedly $3$-acyclic, then
% the map $p^\ast \circ (\tae_4)^{-1} \circ \tau_{/b}$ induces
there exists
an isomorphism
  \[
    \WWW(G,N) \cong \Im (p^\ast) \cap \Im (c^2_G),
  \]
  where $c^2_G\colon \HHH_b^2(G) \to \HHH^2(G)$ is the comparison map and  $p^\ast \colon \HHH^2(\Gamma) \to \HHH^2(G)$ is the map induced by the projection $p \colon G \to \Gamma$.
In particular, if $\Gamma$ is boundedly $3$-acyclic, then
\[\dim \WWW(G,N) \le \dim \HHH^2(G).\]
\end{thm}

%From Theorems \ref{calculation of dim} and \ref{easy cor}, we obtain the following corollary in contrast to Remark \ref{rmk:acyl_hyp}.
We obtain the following corollary to Theorems \ref{calculation of dim} and \ref{easy cor}, in contrast to Remark \ref{rmk:acyl_hyp}.

\begin{cor} \label{cor:fin_dim}
Assume that $\Gamma = G/N$ is boundedly $3$-acyclic.

\begin{enumerate}[$(1)$]
    \item If $\Gamma$ is finitely presented, then $\VVV(G,N)$ is finite-dimensional. In particular, $\WWW(G,N)$ is finite-dimensional.
    \item If $G$ is finitely presented, then $\WWW(G,N)$ is finite-dimensional.
\end{enumerate}
\end{cor}

We also consider the cohomology groups $\HHH_{/b}^{\bullet}(G)$ of
%with respect to
the relative complex $C^{\bullet}(G)/C^{\bullet}_b(G)$.
An exact sequence of complexes
\[ 0 \to C^{\bullet}_b(G) \to C^{\bullet}(G) \to C^{\bullet}(G)/C^{\bullet}_b(G) \to 0 \]
induces the long exact sequence
\[ \cdots \to \HHH_b^1(G) \to \HHH^1(G) \to \HHH_{/b}^1(G) \to \HHH^2_b(G) \to \HHH^2(G) \to \HHH_{/b}^2(G) \to \cdots .\]
Since $\HHH_b^1(G)=0$ and $\HHH^1_{/b}(G) \cong \QQQ(G)$, we obtain the following well-known sequence (see \cite[Theorem 2.50]{Ca})
\[ 0 \to \HHH^1(G) \to \QQQ(G) \to \HHH^2_b(G) \to \HHH^2(G). \]
Hence, the kernel of the comparison map $\Ker( c^2_G )$ is isomorphic to $\QQQ(G)/\HHH^1(G)$.

The authors proved the following five-term exact sequence of the cohomology $\HHH^\bullet_{/ b}$ which extends
the well known exactness sequence (see \cite[Remark 2.90]{Ca})
\[0 \to \QQQ(\Gamma) \xrightarrow{p^*} \QQQ(G) \xrightarrow{i^*} \QQQ(N)^G.\]

\begin{thm}[{\cite[Theorem 1.5]{K2M3}}] \label{thm:5-term}

For exact sequence $(\star)$, there exists an exact sequence

\begin{align}\label{seq:5-term}
  0 \to \QQQ(\ppi) \xrightarrow{p^*} \QQQ(\hG) \xrightarrow{i^*} \QQQ(\bG)^{\hG} \xrightarrow{\tau_{/b}} \HHH_{/b}^2(\ppi) \xrightarrow{p^*} \HHH_{/b}^2(\hG).
\end{align}

Moreover, exact sequence \eqref{seq:5-term} is compatible with the five-term exact sequence of group cohomology, that is, the following diagram commutes:

\begin{align}\label{diagram_coh_qm_rel}
  \xymatrix{
  0 \ar[r] & \HHH^1(\ppi) \ar[r]^-{p^*} \ar[d]^-{\tae_1} & \HHH^1(\hG) \ar[r]^-{i^*} \ar[d]^-{\tae_2} & \HHH^1(\bG)^{\hG} \ar[r]^-{\tau} \ar[d]^-{\tae_3} & \HHH^2(\ppi) \ar[r]^-{p^*} \ar[d]^-{\tae_4} & \HHH^2(\hG) \ar[d]^-{\tae_5} \\
  0 \ar[r] & \QQQ(\ppi) \ar[r]^-{p^*} & \QQQ(\hG) \ar[r]^-{i^*} & \QQQ(\bG)^{\hG} \ar[r]^-{\tau_{/b}} & \HHH_{/b}^2(\ppi) \ar[r]^-{p^*} & \HHH_{/b}^2(\hG).
  }
\end{align}

\end{thm}

By exact sequence \eqref{seq:5-term}, we obtain the inequality
$\dim \VVV(G,N) \leq \dim \HHH^2_{/b}(\Gamma)$ as follows.
%and the equality holds if $\HHH^2_{/b}(G)=0$.
If $\Gamma$ is boundedly 3-acyclic, then $ \xi_4 \colon \HHH^2_{/b}(\Gamma) \to \HHH^2(\Gamma)$ is an isomorphism.
If $G$ is Gromov-hyperbolic, then the comparison map $c^2_G \colon \HHH_b^2(G)\to \HHH^2(G)$ is surjective (\cite{Gromov}, \cite{Mineyev}) and thus
$\mathrm{Im} (\xi_5) =0$. Hence  $p^\ast \colon \HHH_{/b}^2(\Gamma) \to \HHH_{/b}^2(G)$ is zero and therefore Theorem \ref{calculation of dim} follows.
The isomorphism in Theorem \ref{easy cor} is induced from the map $p^\ast \circ (\tae_4)^{-1} \circ \tau_{/b}$ in diagram \eqref{diagram_coh_qm_rel}.

%%%%%%%%%%%%%%%%%%%%%%%%%%%%%%%%%%%%%%%%%%%%%%%%%%%%%%%%%%%%%%%%%%%%%%%%%%%%%%%%%%%%%%%%%%%%%%%%%%%%%%%%%%%%%%%%%%%%%%%%%%%%%%%%%
%%%%%%%%%%%%%%%%%%%%%%%%%%%%%%%%%%%%%%%%%%%%%%%%%%%%%%%%%%%%%%%%%%%%%%%%%%%%%%%%%%%%%%%%%%%%%%%%%%%%%%%%%%%%%%%%%%%%%%%%%%%%%%%%%
%%%%%%%%%%%%%%%%%%%%%%%%%%%%%%%%%%%%%%%%%%%%%%%%%%%%%%%%%%%%%%%%%%%%%%%%%%%%%%%%%%%%%%%%%%%%%%%%%%%%%%%%%%%%%%%%%%%%%%%%%%%%%%%%%
%%%%%%%%%%%%%%%%%%%%%%%%%%%%%%%%%%%%%%%%%%%%%%%%%%%%%%%%%%%%%%%%%%%%%%%%%%%%%%%%%%%%%%%%%%%%%%%%%%%%%%%%%%%%%%%%%%%%%%%%%%%%%%%%%

\section{Comparison problem between stable commutator length and stable mixed commutator length}\label{scl sec}

In this section, we collect results on the comparison problem of scl and mixed scl.
The following theorem shows that the space $\WWW(G,N)$ (see Section~1.3 and Section~5) plays an important role.

\begin{thm}[\cite{KKMM1}, \cite{K2M3}] \label{equiv thm ex}
Assume that $\WWW(G,N) = 0$.
%$\QQQ(\bG)^{\hG} = \HHH^1(\bG)^{\hG} + i^* \QQQ(\hG)$.
Then
\begin{itemize}
  \item[$(1)$] $\scl_{\hG}$ and $\scl_{\hG,\bG}$ are bi-Lipschitzly equivalent on $[\hG,\bG]$.
  \item[$(2)$] If $\Gamma = G/N$ is  amenable, then $\scl_\hG(x) \le \scl_{\hG, \bG}(x) \le 2 \cdot \scl_{\hG}(x)$ for all $x \in [\hG, \bG]$.
  \item[$(3)$] If $\Gamma = G/N$ is solvable, then $\scl_{\hG}(x) = \scl_{\hG,\bG}(x)$ for all $x \in [\hG,\bG]$.
  \item[$(4)$] %If the projection $q \colon G \to \hG/\bG$ admits a virtual section,
  If $(\star)$ virtually splits,
  then $\scl_\hG(x) \le \scl_{\hG, \bG}(x) \le 2 \cdot \scl_{\hG}(x)$ for all $x \in [\hG, \bG]$.
\end{itemize}
\end{thm}
We note that the existence of a virtual section implies that $\QQQ(\bG)^{\hG} = i^* \QQQ(\hG)$ (in particular, $\WWW(\hG,\bG)=0$) due to Theorem \ref{virtually split}.

The following criterion is the key to show Theorem \ref{equiv thm ex}.

\begin{prop}[{\cite[Proposition 7.2]{K2M3}}] \label{prop equivalence criterion}
 Let $C$ be a real number such that for every $\phi \in \QQQ(N)^G$ there exists $\phi' \in \QQQ(G)$ satisfying $\phi - \phi'|_N \in \HHH^1(N)^G$ and $D(\phi') \le C \cdot D(\phi)$.
Then for every $x \in [G,N]$,
\[ \scl_G(x) \le \scl_{G,N}(x) \le C \cdot \scl_G(x).\]
\end{prop}
For the convenience of the readers, we write the proof.
In this proof, the Bavard duality is the key.
\begin{proof}
Let $x \in [G,N]$.
The inequality $\scl_G(x) \le \scl_{G,N}(x)$ immediately follows from definition.
Let $\varepsilon > 0$.
Then Theorem \ref{thm:bavard} implies that there exists $\phi \in \QQQ(N)^G$ such that
\[ \scl_{G,N}(x) - \varepsilon \le \frac{\phi(x)}{2D(\phi)}.\]
 By assumption,
% that $\QQQ(N)^G = \HHH^1(N)^G + i^* \QQQ(G)$,
there exists $\phi' \in \QQQ(G)$ such that $\phi'' = \phi - \phi'|_N \in \HHH^1(N)^G$ and $D(\phi') \le C \cdot D(\phi)$.
Since $\phi''$ is a $G$-invariant homomorphism and $x \in [G,N]$, we obtain $\phi''(x) = 0$, and hence $\phi'(x) = \phi(x)$.
Hence
\[ \scl_{G,N}(x) - \varepsilon \le \frac{\phi(x)}{2 D(\phi)} \le C \cdot \frac{\phi'(x)}{2D(\phi')} \le C \cdot \scl_G(x).\]
Here we use Theorem \ref{original Bavard} to prove the last inequality.
 Since $\varepsilon$ is an arbitrary positive number, we complete the proof.
\end{proof}

The outline of the proof of Theorem \ref{equiv thm ex} is the following.
Under the conditions of (1), (2) and (3) in Theorem \ref{equiv thm ex}, for each $\phi \in \QQQ(N)^G$, one can take $\phi' \in \QQQ(G)$ with controlled defect.

%In Problem 9.15 of \cite{K2M3}, the authors asked whether  $\cl_\hG(x)$ and $\cl_{\hG, \bG}$ are bi-Lipshitz equivalent when $N=[P_n,P_n]$.
%As far as the author know, it is still an open problem.

%The authors do not even know whether there exists a normal subgroup $N$ of $B_n$ and $x \in [\hG, \bG]$ such that $\cl_\hG(x) \neq \cl_{\hG, \bG}(x)$.

Finally, we explain some results on the non-equivalence of the stable commutator length and the stable mixed commutator length.
%To obtain such results, we use the following criterion essentially due to \cite{KK}.
To obtain such results, we use the following criterion. Variants of this criterion have been employed in \cite{KK}, \cite{KKMM2} and \cite{MMM}.

\begin{prop}[\cite{KK}, \cite{KKMM2}, \cite{MMM}]\label{noneq_criterion}
Let $\qm$ be a homogeneous $G$-invariant quasimorphism on $N$.
Assume that $x_i \in [G,N]$ \textup{(}$i=1,2,\ldots$\textup{)} satisfies the following conditions.
\begin{itemize}
\item there exists some positive integer $C$ such that $\cl_G(x_i) \leq C$ for every $i$,
\item $\qm(x_i) \to +\infty$ as $i \to +\infty$.
\end{itemize}
Then,
$[\qm] \neq 0$ in $\WWW(G,N)$, where $[\cdot]$ means the equivalence class in $\WWW(G,N)$. Furthermore,
%$\qm$ is non-extendable to $G$ and
$\scl_G$ and $\scl_{G,N}$ are not bi-Lipschitzly equivalent.
%\begin{itemize}
%\item $\mu$ is non-extendable to $G$, and
%\item $\scl_G$ and $\scl_{G,N}$ are not equivalent.
%\end{itemize}
\end{prop}

\begin{proof}
By the first condition, we have that $\scl_G(x_i) \leq C$ for every $i$.
Suppose that
%$\qm$ is extendable to $G$,
$[\qm] = 0$ in $\WWW(G,N)$,
\emph{i.e.}, there exist $\qm' \in \QQQ(G)$ and $\qm'' \in \HHH^1(N)^G $ such that $\qm'|_N+\qm''=\qm$.
Since $x_i \in [G,N]$, we have $\qm''(x_i)=0$ for every $i$.
Moreover, Theorem \ref{original Bavard} implies that $|\qm'(x_i)| \leq 2D(\qm') \cdot C$ for every $i$.
Hence we obtain that $|\qm(x_i)| \leq 2D(\qm') \cdot C$ for every $i$.
This contradicts the second condition, and hence we see that %$\qm$ is non-extendable to $G$.
$[\qm] \neq 0$ in $\WWW(G,N)$.

By the second condition and Theorem \ref{thm:bavard}, we have that $\scl_{G,N}(x_i) \to +\infty$ as $x_i \to +\infty$ and hence $\scl_G$ and $\scl_{G,N}$ are not bi-Lipschitzly equivalent.
\end{proof}

%The papers \cite{KK} and \cite{MMM} provide example of pairs $(G,N)$ such that $\cl_G$ and $\cl_{G,N}$ are not equivalent by Proposition \ref{noneq_criterion}.
%In \cite{KK},

We explain two applications of Proposition \ref{noneq_criterion}.
One of them contains Theorem \ref{thm:py}.

Let $\Sigma_l$ be a closed orientable surface whose genus $l\geq 2$ and $\omega$ a symplectic form on $\Sigma_l$.
We set $G=\Symp_0(\Sigma_l,\omega)$ and $N=\Ham(\Sigma_l,\omega)(=[G,G])$.
%Then Py's Calabi quasimorphism $\qm_P \colon \Ham(\Sigma_l, \omega) \to \RR$ is non-extendable to $ \Symp_0(\Sigma_l, \omega)$.
In \cite{KK}, some of the authors proved Py's Calabi quasimorphism $\qm_P \colon N \to \RR$ and some $x_i\in [G,N](=\Ham(\Sigma_l,\omega))$ ($i=1,2,\ldots$) satisfy the conditions in Proposition \ref{noneq_criterion} and hence  %$\qm_P$ is non-extendable to $G$ (Theorem \ref{thm:py})
$[\qm_P] \neq 0$ in $\WWW(G,N)$ (implying Theorem \ref{thm:py}) and $\scl_G$ and $\scl_{G,N}$ are not equivalent.

%Maruyama, Matsushita and Mimura gave a construction of invariant quasimorphisms for groups $G$ acting on the circle whose Euler class on $G$ descends to a class on the quotient $\G$.
In \cite{MMM}, some of the authors provided
%Maruyama, Matsushita and Mimura gave
a construction of invariant quasimorphisms for groups acting on the circle whose Euler class descends to a class on the quotient (for its precise description, see Subsection \ref{subsec:rotation_invq}).
%These invariant quasimorphisms are non-extendable if the Euler class is non-zero.
For example, if we take a Fuchsian action $\rho$ of the fundamental group $G=\pi_1(\Sigma_l)$ of a closed orientable surface of genus $l \geq 2$, the construction supplies a $G$-invariant homogeneous quasimorphism $\mu_{\rho}$ on $N=[G,G]$.

As noted in Subsection \ref{subsec:rotation_invq}, the fact $[\mu_{\rho}] \neq 0$ in $\WWW(G,N)$,  the first conclusion  of Proposition \ref{noneq_criterion}, is deduced from Theorem \ref{thm:rotation_nonext}.
%Note that the proof of Theorem \ref{thm:rotation_nonext} is done without Proposition \ref{noneq_criterion}.
%The quasimorphism $\mu_{\rho}$ is a non-zero element of $\WWW(G,N)$ by Theorem \ref{thm:rotation_nonext}, the proof of which can be done without Proposition \ref{noneq_criterion}.
 In order to derive the second conclusion of  Proposition~\ref{noneq_criterion}, they also constructed such an sequence $(x_i)_{i\in \NN}$ in $[G,N]$.
%They nevertheless found $x_i\in [G,N]$ satisfying the conditions in Proposition \ref{noneq_criterion} and %hence
%their quasimorphism is non-extendable to $G$
%$[\mu_\rho] \neq 0$ in $\WWW(G,N)$ and
%concluded that $\scl_G$ and $\scl_{G,N}$ are not equivalent.

(Note that $\WWW(G,N) \neq 0$
%the existence of a non-extendable quasimorphism
had been known in \cite{K2M3}, see (2) of Theorem \ref{raretu of keisan}.)
% which is non-extendable to $\pi_1(\Sigma_l)$.

%%%%%%%%%%%%%%%%%%%%%%%%%%%%%%%%%%%%%%%%%%%%%%%%%%%%%%%%%%%%%%%%%%%%%%%%%%%%%%%%%%%%%%%%%%%%%%%%%%%%%%%%%%%%%%%%%%%%%%%%%%%%%%%%%
%%%%%%%%%%%%%%%%%%%%%%%%%%%%%%%%%%%%%%%%%%%%%%%%%%%%%%%%%%%%%%%%%%%%%%%%%%%%%%%%%%%%%%%%%%%%%%%%%%%%%%%%%%%%%%%%%%%%%%%%%%%%%%%%%
%%%%%%%%%%%%%%%%%%%%%%%%%%%%%%%%%%%%%%%%%%%%%%%%%%%%%%%%%%%%%%%%%%%%%%%%%%%%%%%%%%%%%%%%%%%%%%%%%%%%%%%%%%%%%%%%%%%%%%%%%%%%%%%%%
%%%%%%%%%%%%%%%%%%%%%%%%%%%%%%%%%%%%%%%%%%%%%%%%%%%%%%%%%%%%%%%%%%%%%%%%%%%%%%%%%%%%%%%%%%%%%%%%%%%%%%%%%%%%%%%%%%%%%%%%%%%%%%%%%

\section{Examples} \label{sec:ext_prob}

In this section, we collect results on the extension problem and comparison problem.
Most of the results in this section are previously known (see \cite{KKMM1} and \cite{K2M3}), but some of the results (Example \ref{aut_f2} and Theorem \ref{cantor}) are newly mentioned in this paper.
%Most are summaries of previous results in \cite{KKMM1} and \cite{K2M3}, but some are newly mentioned (Example \ref{aut_f2} and Theorem \ref{cantor} for example).
%(for example, are newly mensioned.)

\begin{comment}
First we consider Problem \ref{large problem 1}, which asks for the difference between $\QQQ(N)^G$ and $i^\ast \QQQ(G)$ (\emph{i.e.}, the non-triviality of $\VVV(G,N)$).
%As mentioned in Section \ref{history after},
In \cite{Sh}, Shtern gave the first example of non-extendable quasimorphism as follows.

\begin{example}[{\cite[Example 1]{Sh}}] \label{eg:shtern}
Let $G$ be a (continuous) Heisenberg group.
There is a non-trivial $G$-invariant homomorphism $\phi$ on the commutator subgroup $[G,G]$.
Then, this quasimorphism $\phi$ is non-extendable to $G$ because $G$ is nilpotent, in particular amenable.
\end{example}

That is, Example \ref{eg:shtern} provides an example of a pair $(G,N)$ such that $\VVV(G,N) \neq 0$.
\end{comment}

By definition, $\VVV(G,N)=0$ means that every $G$-invariant quasimorphisms on $N$ is extendable. Note that $\VVV(G,N)=0$ implies $\WWW(G,N)=0$.
By (1) of Theorem \ref{equiv thm ex}, if $G$ and $N$ satisfy $\WWW(G,N)=0$,
then $\scl_G$ and $\scl_{G,N}$ are equivalent on $[G,N]$.

By Theorem \ref{virtually split}, %$\QQQ(N)^G=i^\ast\QQQ(G)$
$\VVV(G,N) = 0$
if $p \colon G \to \Gamma$ virtually splits (in particular $\Gamma$ is virtually free). Thus we obtain the following example.

  \begin{example} \label{aut_f2}
    Let $G=\Aut(F_2)$ and $N=F_2$. Here, $F_2$ is considered as a normal subgroup of $\Aut(F_2)$ under the identification $F_2={\rm Inn}(F_2)$.
    Then,
    %every $G$-invariant homogeneous quasimorphism on $N$ is extendable to $G$
    $\VVV(G,N)=0$
    since $\Gamma = {\rm Out}(F_2) \cong \GL(2,\ZZ)$ is virtually free.
    %\[\QQQ(\bG)^{\hG} / i^* \QQQ(\hG) = 0.\]
    In particular, Brandenbursky--Marcinkowski's quasimorphisms in Theorem \ref{thm:BM_aut} are extendable.
  \end{example}

The authors do not know Example \ref{aut_f2} is generalized to $G={\rm Aut}(F_n)$ and $N=F_n$ $(n\geq3)$.
However, a generalization of Example \ref{aut_f2} in the other direction is obtained.
In \cite{K2M3}, the authors proved the following.

\begin{thm}[{\cite[Theorem 1.9]{K2M3}}]\label{ia_extendable}
Let $\IAA_n$ be the \emph{IA-automorphism group} of the free group $F_n$, \emph{i.e.}, the kernel of the natural homomorphism $\Aut(F_n) \to \GL(n,\ZZ)$.
%For every $n\geq 2$, every $\Aut(F_n)$-invariant quasimorphism on $\IA_n$ is extendable.
For every $n\geq 2$, $\VVV(\Aut(F_n),\IA_n)=0$ and thus $\scl_G$ and $\scl_{G,N}$ are equivalent. % by Theorem $\ref{equiv thm ex}$.
\end{thm}

We can regard the above result as a generalization of Example \ref{aut_f2} as follows.
Since $\IAA_2 = \mathrm{Inn}(F_2)$ (\cite{Nielsen}), the exact sequence
\[1 \to \mathrm{Inn}(F_2) \cong F_2 \to \Aut(F_2) \to {\rm Out}(F_2) \to 1,\]
which we considered in Example \ref{aut_f2}, coincides with
\[1 \to \IA_2 \to \Aut(F_2) \to \GL(2,\ZZ) \to 1.\]
Thus Example \ref{aut_f2} implies that
%every $\Aut(F_2)$-invariant quasimorphism on $\IA_2$ is extendable.
$\VVV(\Aut(F_2),\IA_2)=0$.

Let $G$ be a group with trivial center.
If $G$ satisfies one of the assumption of (1), (2) or (3) of Theorem \ref{thm:autqm}, then Fournier-Facio--Wade's quasimorphisms (in Theorem \ref{thm:autqm}) on $G \cong {\rm Inn}(G)$ are extendable to $\Aut(G)$ since they are constructed by restricting quasimorphisms on $\Aut(G)$.
However, the author do not know the answer of the following problem.

\begin{openproblem}
Let $G$ be a group with trivial center and satisfying the assumption of \textup{(4)} of Theorem \textup{\ref{thm:autqm}}. Are Fournier-Facio--Wade's quasimorphisms extendable?
\end{openproblem}

\begin{comment}
We provide several examples of pairs $(G,N)$ such that $\VVV(G,N)=0$, which have not been mentioned in previous papers.
Building on some known results and recent work, we obtain the following results.

\begin{thm}\label{cantor}
If $G$ and $N$ are one of the following, then $\VVV(G,N)=0$.
\begin{enumerate}[$(1)$]
    \item $G$ is the \textit{big mapping class group} $\textrm{MCG}(\RR^2 \setminus C )$ of the plane minus a Cantor set $\RR^2 \setminus C$ and $N$ is the pure mapping class group of $G$, \emph{i.e.}, the kernel of the homomorphism $p \colon G \to \mathrm{Homeo}(C)$ induced by the action of $G$ on $C$
     \item $G$ is the \textit{braided Thompson group} $bV$ and $N = bP$ is the kernel of the projection $p \colon G \to V$ to the Thompson group $V$
\end{enumerate}
\end{thm}
\begin{proof}
   It is known that $\textrm{Homeo}(C)$ is acyclic \cite{SerTsu} and boundedly acyclic \cite{Andritsch}.
   It is also known that $V$ is $\QQ$-acyclic \cite{Brown92} (moreover $\ZZ$-acyclic \cite{SzWa}) and boundedly acyclic \cite{FFLZ}.
   Thus we obtain $\VVV(G,N)=0$ in both cases by Theorem \ref{calculation of dim}.
\end{proof}
\end{comment}

Next, we consider Problem \ref{large problem 3}, which asks for the difference between $\QQQ(N)^G$ and $\HHH^1(N)^G + i^\ast \QQQ(G)$ (\emph{i.e.}, the non-vanishing of $\WWW(G,N)$).

\begin{comment}
In \cite{KK}, some of the authors provided the first example of non-extendable quasimorphism which is not a homomorphism.

\begin{thm}[{\cite[Theorem 1.11]{KK}}] \label{thm:py}
Let $\Sigma_l$ be a closed orientable surface whose genus $l\geq 2$ and $\omega$ a symplectic form on $\Sigma_l$. Then Py's Calabi quasimorphism $\qm_P \colon \Ham(\Sigma_l, \omega) \to \RR$ is non-extendable to $ \Symp_0(\Sigma_l, \omega)$.
\end{thm}

Note that Theorem \ref{thm:py} provides an example of $(G,N)$ such that $\WWW(G,N) \neq 0$.

\begin{remark} \label{rmk:Z2}
Moreover, there exist $f,g \in \Symp_0(\Sigma_l, \omega)$ such that $\qm_P$ is non-extendable to $\flux_\omega^{-1}( \langle \flux_\omega(f), \flux_\omega(g) \rangle )$.
Therefore, if we set $N = \Ham(\Sigma_l, \omega)$, $\Gamma = \langle \flux_\omega(f), \flux_\omega(g) \rangle$, and $G = \flux_\omega^{-1}(\Gamma)$, then they provide an example of the sequence $(\star)$ such that $\Gamma \cong \ZZ^2$ $\WWW(G,N) \neq 0$.
\end{remark}
\end{comment}

The following is a list of examples of pairs $(G,N)$ %in \cite{K2M3}
 satisfying $\WWW(G,N)=0$.

\begin{thm}[\cite{K2M3}]\label{various_extendable}
    Let $G$ be a group and set $N=[G,G]$.
    If $G$ is one of the following, then
    %$\VVV(G,N)=0$. In particular
    $\WWW(G,N)=0$.
    \begin{enumerate}[$(1)$]
\item free groups $F_n$.
%\item the braid group $B_n$
\item the fundamental group $\pi_1(N_\genus)$ of
a non-orientable closed surface $N_\genus$ with genus $\genus \geq 1$.
\item the knot group $\pi_1(S^3 \setminus K)$ of a knot in $S^3$.
\item the link group $\pi_1(S^3 \setminus L)$ of a hyperbolic link $L$ with two components.
\item the group of the form $F_n / \llangle r_1, \cdots, r_m \rrangle$, where $r_1, \cdots, r_m \in [F_n, [F_n, F_n]]$.
    \end{enumerate}
Here, $\llangle \cdot \rrangle$ means the normal closure \textup{(}in $F_n$\textup{)}.
\end{thm}

Items (1)-(3) follow from $\HHH^2(G)=0$ and Theorem \ref{easy cor}.
For (4) and (5), see \cite[Corollary 4.12]{K2M3} and \cite[Corollary 4.16]{K2M3}, respectively.

%Thus, scl and mixed scl are equivalent for the examples given above.
%Moreover, in the setting in Theorem \ref{various_extendable}, $\scl_G$ and $\scl_{G,N}$ coincide on $[G,N]$ by
In the setting of Theorem \ref{various_extendable}, we conclude that $\scl_{G}$ and $\scl_{G,N}$ are equivalent on $[G,N]$. In fact, they coincide on $[G,N]$ by (3) of Theorem \ref{equiv thm ex}.

In the case where $G$ is the braid group $B_n$ ($n\geq 3$), we study the comparison problem between $\scl_{G}(x)$ and $\scl_{G,N} (x)$ for several interesting normal subgroups $N$
(for example, $[B_n,B_n]$, $P_n$, $[B_n,P_n]=P_n\cap[B_n,B_n]$ and $[P_n,P_n]$, where $P_n$ is the pure braid group).
%If the whole group $G$ is the braid group $B_n$ ($n \geq 3$), we consider a lot of natural normal subgroups (for examples, $[B_n,B_n]$, $P_n$, $[B_n,P_n]=P_n\cap[B_n,B_n]$ and $[P_n,P_n]$).
%It is interesting to consider the comparison problem between $\scl_\hG(x)$ and $\scl_{\hG, \bG}(x)$.

\begin{thm}[\cite{KKMM1}, \cite{K2M3}]\label{survey of braid}
Let $G=B_n$ \textup{(}$n \geq 3$\textup{)}.
\begin{enumerate}[$(1)$]
%\item If $N=[B_n,B_n]$, then $\scl_\hG(x) = \scl_{\hG, \bG}(x)$ for all $x \in [\hG, \bG]$.
\item If $N$ is $P_n$ or $[B_n,P_n]=P_n\cap[B_n,B_n]$ or $[P_n,P_n]$, then $\scl_\hG(x) \le \scl_{\hG, \bG}(x) \le 2 \cdot \scl_{\hG}(x)$ for all $x \in [\hG, \bG]$.
Moreover, if $n\leq4$, $\scl_\hG(x) = \scl_{\hG, \bG}(x)$ for all $x \in [\hG, \bG]$.
\item If $N=[B_n,B_n]$, then
%$\cl_\hG(x) = \cl_{\hG, \bG}(x)$ %(in particular, $\scl_\hG(x) = \scl_{\hG, \bG}(x)$)
$\scl_\hG(x) = \scl_{\hG, \bG}(x)$
for all $x \in [\hG, \bG]$.
%\item If $N=P_n$, then there exists a positive constant $C_1$ such that $\cl_\hG(x) \le \cl_{\hG, \bG}(x) \le C_1 \cdot \cl_{\hG}(x)$ for all $x \in [\hG, \bG]$
%\item If $N=[B_n,P_n]$, then there exists a positive constant $C_2$ such that $\cl_\hG(x) \le \cl_{\hG, \bG}(x) \le C_2 \cdot \cl_{\hG}(x)$ for all $x \in [\hG, \bG]$.
\end{enumerate}
\end{thm}
\begin{proof}%[Outline of the proof]
It is known that $\HHH^2(G)=0$ (see \cite{AL} for example) and $\Gamma=G/N$ is amenable.
Hence Theorem \ref{easy cor} implies that $\WWW(G,N)=0$ and thus the former of (1) follows from (2) of Theorem \ref{equiv thm ex}.
The latter of (1) and (2) follow from (3) of Theorem \ref{equiv thm ex} (note that the symmetric group $\mathfrak{S}_n$ is solvable for $n \leq 4$).
\end{proof}

Note that ``Moreover'' part of (1) is not written in the earlier papers.
In the setting of (2) of Theorem \ref{survey of braid}, furthermore $\cl_\hG(x) = \cl_{\hG, \bG}(x)$ holds for all $x \in [\hG, \bG]$ (see Example \ref{eg:braid}).

We also consider other groups related to braid groups.

Let $C$ be a Cantor set in $\RR^2$ and $\Mod(\RR^2 \setminus C )$ the mapping class group of the plane minus a Cantor set.
Such a mapping class group of an infinite type surface is called a \emph{big mapping class group}. Big mapping class groups have been actively studied in recent years (see \cite{bigMCG}).
Let $\PMod(\RR^2 \setminus C )$ be the pure mapping class group, \textit{i.e.}, the kernel of the homomorphism $p \colon \Mod(\RR^2 \setminus C ) \to \mathrm{Homeo}(C)$ induced by the action of $\Mod(\RR^2 \setminus C )$ on $C$.

Let $bV$ be the \textit{braided Thompson group} introduced by Brin \cite{Brin} and Dehornoy \cite{Dehornoy}.
Let $bP$ be the kernel of the projection $p \colon bV \to V$ to the Thompson group $V$.
Note that the groups $bV$, $bP$, and $V$ are obtained by a certain limit of families $(B_n)_{n \in \NN}$, $(P_n)_{n \in \NN}$ and $(\mathfrak{S}_n)_{n \in \NN}$, respectively (see \cite{Zaremsky}).

%We provide several examples of pairs $(G,N)$ such that $\VVV(G,N)=0$,
Building on the work \cite{Brown92}, \cite{SerTsu} and \cite{Andritsch}, we obtain the following result, which have not been mentioned in previous papers.

\begin{thm}\label{cantor}
If $(G,N)=(\Mod(\RR^2 \setminus C ),\PMod(\RR^2 \setminus C ))$ or $(bV,bP)$, then $\VVV(G,N)=0$.
\end{thm}
\begin{proof}
   It is known that $\textrm{Homeo}(C)$ is acyclic \cite{SerTsu} and boundedly acyclic \cite{Andritsch}.
   It is also known that $V$ is $\QQ$-acyclic \cite{Brown92} (moreover $\ZZ$-acyclic \cite{SzWa}) and boundedly acyclic \cite{Andritsch}.
   Thus we obtain $\VVV(G,N)=0$ in both cases by Theorem \ref{calculation of dim}.
\end{proof}

Note that the space $\QQQ(G)$ (and hence $\QQQ(N)^G$) is infinite-dimensional in both cases \cite{JBavard}, \cite{FFLZ}.

In \cite{{K2M3}}, the authors provided several computations
of $\dim \VVV(G,N)$ and $\dim \WWW(G,N)$.
%on $\VVV(G,N)$ and $\WWW(G,N)$.
We list them at the end of this section.
\begin{thm}\label{raretu of keisan}
\begin{enumerate}[$(1)$]
    \item \textup{(\cite[Theorem 4.5]{K2M3})}
    For $n \ge 1$, set $\hG = F_n$ and $\bG = [F_n, F_n]$. Then
  \[ \dim\VVV(G,N) = \frac{n(n-1)}{2} \quad \textrm{and} \quad  \dim \WWW(G,N) = 0.\]
  \item \textup{(\cite[Theorem 1.1]{K2M3})}
  Let $\genus$ be an integer at least two, $\hG = \pi_1(\Sigma_\genus)$ the surface group with genus $\genus$, and $\bG$ the commutator subgroup $[\pi_1(\Sigma_\genus),\pi_1(\Sigma_\genus)]$ of $\pi_1(\Sigma_\genus)$. Then
  \[ \dim \VVV(G,N) = \genus (2\genus-1) \quad \textrm{and} \quad \dim \WWW (G,N) = 1.\]
  \item \textup{(\cite[a special case of Theorem 1.2]{K2M3})}
  Let $\genus$ be an integer at least two and $f \colon \Sigma_{\genus} \to \Sigma_{\genus}$ an orientation preserving diffeomorphism whose isotopy class $[f]$ is contained in the Torelli group $\II(\Sigma_{\genus})$ and pseudo-Anosov. If $\hG$ is the fundamental group of the mapping torus $T_f$ and $\bG$ is the commutator subgroup $[\hG,\hG]$ of $\hG$, then
\[ \dim \VVV(G,N) = 2\genus + \binom{2\genus}{2} \quad \textrm{and} \quad \dim \WWW(G,N) = 2\genus + 1.\]
\item \textup{(\cite[Corollary 4.14]{K2M3})}
Let $E \to \Sigma_{\genus}$ be a non-trivial circle bundle over a closed oriented surface of genus $\genus \geq 2$.
  For the fundamental group $\hG = \pi_1(E)$ and its normal subgroup $\bG = [G, G]$, we have
  \[
    \dim \VVV(G,N) = \genus(2\genus - 1) \quad \textrm{and} \quad \dim \WWW(G,N) = 0.
  \]
  \item \textup{(\cite[Theorem 4.18]{K2M3})}
Let $n$ and $k$ be integers at least two, and $r$ an element of $[F_n, F_n]$ such that there exists $f_0 \in \HHH^1([F_n,F_n])^{F_n}$ with $f_0(r) \ne 0$, equivalently, $r\in [F_n,F_n]\setminus [F_n, [F_n,F_n]]$.
Set $\hG = F_n / \llangle r^k \rrangle$ and $\bG = [\hG,\hG]$.
Then
\[\dim \VVV(G,N) = \frac{n(n - 1)}{2} \quad \textrm{and} \quad \dim \WWW(G,N) = 1.\]
\end{enumerate}
\end{thm}

Finding examples of $(G,N)$ whose quotient $\Gamma$ is not boundedly $3$-acyclic with non-zero $\WWW(G,N)$ is widely open.
Such an example may lead to finding an interesting (especially in the context of Lex groups) $3$-dimensional bounded cohomology class of $\Gamma$.
For example, let us consider the mapping class group $\Mod(\Sigma_l)$ of $\Sigma_l$ of genus $l \geq 2$.
Let $\Mod(\Sigma_l, \ast)$ be the mapping class group of $\Sigma_l$ with once-marked point $\ast$.
These groups give rise to the well-known Birman exact sequence
\[
  1 \to \pi_1(\Sigma_l) \to \Mod(\Sigma_l, \ast) \to \Mod(\Sigma_l) \to 1.
\]
Applying (\ref{diagram_coh_qm_rel}) to this exact sequence, we obtain the following commutative diagram:

\begin{align*}
  \xymatrix{
  & & \HHH_b^2(\mathcal{M}_l) \ar[r] \ar[d] & \HHH_b^2(\mathcal{M}_{l,\ast}) \ar[d] \\
  \HHH^1(\mathcal{M}_{l,\ast}) \ar[r] \ar[d] & \HHH^1(\pi_1(\Sigma_l))^{\mathcal{M}_{l,\ast}} \ar[r] \ar[d] & \HHH^2(\mathcal{M}_l) \ar[r] \ar[d]^-{\tae_4} & \HHH^2(\mathcal{M}_{l,\ast}) \ar[d] \\
  \QQQ(\mathcal{M}_{l,\ast}) \ar[r] & \QQQ(\pi_1(\Sigma_l))^{\mathcal{M}_{l,\ast}} \ar[r]^-{\tau_{/b}} & \HHH_{/b}^2(\mathcal{M}_l) \ar[r] \ar[d] & \HHH_{/b}^2(\mathcal{M}_{l,\ast}) \ar[d] \\
  & & \HHH_b^3(\mathcal{M}_l) \ar[r]^-{p^*} & \HHH_b^3(\mathcal{M}_{l,\ast}).
  }
\end{align*}
Here, in this diagram we write $\mathcal{M}_l$ and $ \mathcal{M}_{l,\ast}$ for $\Mod(\Sigma_l)$ and $\Mod(\Sigma_l,\ast)$, respectively.

\begin{comment}
\begin{align*}
  \xymatrix{
  & & \HHH_b^2(\Mod(\Sigma_l)) \ar[r] \ar[d] & \HHH_b^2(\Mod(\Sigma_l, x)) \ar[d] \\
  \HHH^1(\Mod(\Sigma_l, x)) \ar[r] \ar[d] & \HHH^1(\pi_1(\Sigma_l))^{\Mod(\Sigma_l, x)} \ar[r] \ar[d] & \HHH^2(\Mod(\Sigma_l)) \ar[r] \ar[d]^-{\tae_4} & \HHH^2(\Mod(\Sigma_l, x)) \ar[d] \\
  \QQQ(\Mod(\Sigma_l, x)) \ar[r] & \QQQ(\pi_1(\Sigma_l))^{\Mod(\Sigma_l, x)} \ar[r]^-{\tau_{/b}} & \HHH_{/b}^2(\Mod(\Sigma_l)) \ar[r] \ar[d] & \HHH_{/b}^2(\Mod(\Sigma_l, x)) \ar[d] \\
  & & \HHH_b^3(\Mod(\Sigma_l)) \ar[r]^-{p^*} & \HHH_b^3(\Mod(\Sigma_l, x)).
  }
\end{align*}
\end{comment}
The second cohomology group $\HHH^2(\Mod(\Sigma_l))$ is trivial for $l = 2$ and isomorphic to $\ZZ$ for $l \geq 3$ (see \cite{MR1892804}).
Moreover, for $l \geq 3$, the generator is given by the signature cocycle, which is a bounded cocycle.
Hence, the map $\tae_4 \colon \HHH^2(\Mod(\Sigma_l)) \to \HHH_{/b}^2(\Mod(\Sigma_l))$ is the zero map for $l \geq 2$.
This implies that the space $\WWW(\Mod(\Sigma_l,\ast), \pi_1(\Sigma_l))$ injects to $\Ker \big( p^* \colon \HHH_b^3(\Mod(\Sigma_l)) \to \HHH_b^3(\Mod(\Sigma_l,\ast)) \big)$.
(Note that the space $\HHH^1(\pi_1(\Sigma_l))^{\Mod(\Sigma_l,\ast)}$ vanishes by the change of coordinates principle (see \cite{farb_margalit12}), and thus we have $\WWW(\Mod(\Sigma_l,\ast), \pi_1(\Sigma_l)) = \VVV(\Mod(\Sigma_l,\ast), \pi_1(\Sigma_l))$.)
Hence if there exists a non-zero element of $\WWW(\Mod(\Sigma_l,\ast), \pi_1(\Sigma_l))$, the group $\Mod(\Sigma_l)$ turns out to be a non-Lex group.

\begin{openproblem}\label{prob:nonlex}
 For each fixed $l \geq 2$, is  $\WWW(\Mod(\Sigma_l,\ast), \pi_1(\Sigma_l))$ the zero space?
%  Is the space $\WWW(\Mod(\Sigma_l, x), \pi_1(\Sigma_l))$ equal to zero?
\end{openproblem}

%%%%%%%%%%%%%%%%%%%%%%%%%%%%%%%%%%%%%%%%%%%%%%%%%%%%%%%%%%%%%%%%%%%%%%%%%%%%%%%%%%%%%%%%%%%%%%%%%%%%%%%%%%%%%%%%%%%%%%%%%%%%%%%%%
%%%%%%%%%%%%%%%%%%%%%%%%%%%%%%%%%%%%%%%%%%%%%%%%%%%%%%%%%%%%%%%%%%%%%%%%%%%%%%%%%%%%%%%%%%%%%%%%%%%%%%%%%%%%%%%%%%%%%%%%%%%%%%%%%
%%%%%%%%%%%%%%%%%%%%%%%%%%%%%%%%%%%%%%%%%%%%%%%%%%%%%%%%%%%%%%%%%%%%%%%%%%%%%%%%%%%%%%%%%%%%%%%%%%%%%%%%%%%%%%%%%%%%%%%%%%%%%%%%%
%%%%%%%%%%%%%%%%%%%%%%%%%%%%%%%%%%%%%%%%%%%%%%%%%%%%%%%%%%%%%%%%%%%%%%%%%%%%%%%%%%%%%%%%%%%%%%%%%%%%%%%%%%%%%%%%%%%%%%%%%%%%%%%%%

\section{Comparison problem between commutator length and mixed commutator length}\label{cl comparison sec}

In this section, we treat the equivalence and coincidence problems between $\cl_\hG$ and $\cl_{\hG,\bG}$. We first mention the equivalence problem, following \cite[Section~7]{KKMM1}.

\begin{thm}[{\cite[Theorem~7.1]{KKMM1}, \cite[Theorem~1.5]{K2M32}}]
%Assume that the projection $p \colon \hG \to \Gamma$ has a section homomorphism,
Assume that short exact sequence $(\star)$ splits.
Then for each $x \in [\hG,\bG]$, the following inequality holds:
\[ \cl_{\hG,\bG}(x) \le 3 \cdot \cl_\hG(x).\]
Moreover, if $N$ is abelian, then
\[ \cl_{\hG, \bG}(x) \le 2 \cdot \cl_\hG(x).\]
\end{thm}

\begin{thm}[{\cite[Theorem~7.2]{KKMM1}}] \label{thm:finite_index_cl}
Assume that the quotient group $\Gamma = \hG/\bG$ is a finite group. Then there exists $C \ge 1$ such that for each $x \in [\hG,\bG]$, the following inequality holds:
\[ \cl_{\hG,\bG}(x) \le C \cdot \cl_\hG(x).\]
\end{thm}

These two results imply that if $\Gamma$ is finite or the projection $p \colon \hG \to \Gamma$ has a section homomorphism, then $\cl_\hG$ and $\cl_{\hG,\bG}$ are equivalent. It is not known that  whether the projection $p \colon \hG \to \Gamma$ has a virtual section, then $\cl_\hG$ and $\cl_{\hG,\bG}$ are equivalent. However, under some additional assumption,
%if we assume additional assumption, then
we have an equivalence results of $\cl_{G,N}$ and $\cl_G$.
%the following equivalence result:

\begin{thm}[{\cite[Theorem~7.16]{KKMM1}}] \label{thm cl equivalence pp}
Let $\bG$ be a normal subgroup of a group $\hG$ and $p \colon \hG \to \Gamma$ be the projection.
If there exists a virtual section $(s,\Lambda)$ of $\ppi$ such that the image of $s$ is contained in the center $Z(\hG)$ for $\hG$,
then $\cl_{\hG,\bG}$ and $\cl_\hG$ are equivalent on $[\hG,\bG]$.
%In particular, if the projection $G \to \hG/\bG$ admits a section homomorphism $\hG/\bG \to \hG$ or $\hG/\bG$ is finite, then $\cl_{\hG,\bG}$ and $\cl_\hG$ are equivalent on $[\hG,\bG]$.
\end{thm}

Note that Theorem \ref{thm cl equivalence pp} is a generalization of Theorem \ref{thm:finite_index_cl} since the trivial subgroup is in the center.
This theorem has the following application.

\begin{prop}[{\cite[Proposition~7.15]{KKMM1}}]\label{pure braid and commutator}
For every positive integer $n$, $\cl_{B_n}$ and $\cl_{B_n, P_n \cap [B_n,B_n]}$ are equivalent on $[B_n,P_n \cap[B_n,B_n]]$.
\end{prop}

Here we provide the following open problem related to the equivalence problem of $\cl_\hG$ and $\cl_{\hG,\bG}$.
Note that in these cases, $\scl_G$ and $\scl_{G,N}$ are known to be equivalent (Theorems \ref{ia_extendable} and \ref{survey of braid}).

\begin{openproblem}[{\cite[Problem 9.15]{K2M3}}]
For $(\hG,\bG)=\left(B_n, [P_n,P_n]\right)$ $(n \ge 3)$, $\left(\Aut(F_n),\IA_n\right)$ $(n \ge 2)$, are $\cl_{\hG}$ and $\cl_{\hG,\bG}$ equivalent  on $[\hG,\bG]$?
\end{openproblem}

For the rest of this section, we consider the coincidence problem of $\cl_\hG$ and $\cl_{\hG,\bG}$, following \cite{K2M32}. As was mentioned in Theorem~\ref{equiv thm ex}, there are several cases that $\scl_\hG$ and $\scl_{\hG,\bG}$ coincide. There is no guarantee that $\cl_\hG$ and $\cl_{\hG,\bG}$ coincide although $\scl_\hG$ and $\scl_{\hG,\bG}$ coincide. However, rather surprisingly, $\cl_\hG$ and $\cl_{\hG,\bG}$ coincide in the case the quotient group $\Gamma$ has a certain algebraic property, called \emph{local cyclicity}.

Recall that a group $\Gamma$ is \emph{locally cyclic} if every finitely generated subgroup of $\Gamma$ is cyclic. Cyclic groups, $\QQ$ and $\ZZ[1/2]$ are typical examples of locally cyclic groups, and $(\ZZ / 2\ZZ)^2$ and $\RR$ are not locally cyclic.

A locally cyclic group is abelian, and its second cohomology group with real coefficients vanishes.
Thus, Theorems \ref{easy cor} and \ref{equiv thm ex} imply that $\scl_\hG$ and $\scl_{\hG,\bG}$ coincide when $\Gamma = G/N$ is locally cyclic. In this case, however, the following stronger assertion holds.

\begin{thm}[{\cite[Theorem~1.7]{K2M32}}] \label{thm locally cyclic}
If $\Gamma = \hG/\bG$ is locally cyclic, then $\cl_\hG$ and $\cl_{\hG,\bG}$ coincide on $[\hG,\bG]$.
\end{thm}

\begin{example} \label{eg:braid}
Let $G$ be the braid group $B_n$ and $N$ the commutator subgroup $[G,G]$ of $G$.
Then $\cl_{G,N}$ and $\cl_G$ coincide on $[G,N]$ since $G/N \cong \ZZ$ is (locally) cyclic.
\end{example}

From this result, it is natural to consider the following problem:

\begin{openproblem}
    If $\Gamma$ is not locally cyclic, then does there exist a triple $(\hG,\bG,\Gamma)$ fitting in $(\star)$ such that $\cl_\hG(x)$ and $\cl_{\hG,\bG}(x)$ does not coincide for some element $x \in [\hG,\bG]$?
\end{openproblem}
In the case that $\Gamma$ is abelian, we can construct such a triple $(\hG,\bG, \Gamma)$ using the wreath product.

Here we recall the definition of the wreath product. Let $H$ and $\Gamma$ be groups. Then $\Gamma$ acts on $\bigoplus_\Gamma H$ by right, and the wreath product $H \wr \Gamma$ is the semidirect product of this action. Note that in this case there exists a split short exact sequence
\[ 1 \to \bigoplus_{\Gamma} H \to H \wr \Gamma \to \Gamma \to 1\]
of groups. When $G=\ZZ \wr \Gamma$, mixed commutator length is closely related to general rank \cite{K2M32}.

Let $\Gamma$ be a group and $\Lambda$ its subgroup. Recall that the \emph{rank} $\rk(\Gamma)$ of $\Gamma$ is the smallest cardinality of generators of $\Gamma$. Define the \emph{intermediate rank} of the pair $(\Gamma, \Lambda)$ to be the number
\[ \intrk^\Gamma(\Lambda) = \inf \{ \rk (\Theta) \; | \; \Lambda \leqslant \Theta \leqslant \Gamma \} \in \ZZ_{\geq 0} \cup \{ \infty \}.\]
Define the \emph{general rank} of a group $\Gamma$ by
\[ \genrk(\Gamma) = \sup \{ \intrk^\Gamma (\Lambda) \; | \; \textrm{$\Lambda$ is a finitely generated subgroup of $\Gamma$}\}.\]
General rank was originally introduced by Malcev \cite{Malcev}, which is a generalization of the rank of finitely generated group. In fact, if $\Gamma$ is finitely generated, then $\rk (\Gamma) = \genrk (\Gamma)$. If $\Gamma$ is abelian, then the general rank of $\Gamma$ coincides with the maximal rank of a finitely generated subgroup of $\Gamma$ (see \cite[Lemma~3.2]{K2M32}). Note that $\genrk(\Gamma) \le 1$ if and only if $\Gamma$ is locally cyclic.

Using general rank, the mixed commutator lengths of $\hG = \ZZ \wr \Gamma$ and $N = \bigoplus_\Gamma \ZZ$ can be described as follows:

\begin{thm}[{\cite[Theorems~1.4 and 1.8]{K2M32}}] \label{thm not coincidence}
Set $G = \ZZ \wr \Gamma$ and $N = \bigoplus_\Gamma \ZZ$. Then the following hold:
\begin{enumerate}[$(1)$]
\item For $r \in \ZZ_{\ge 0}$, $r \le \genrk(\Gamma)$ if and only if there exists $x \in [\hG, \bG]$ such that $\cl_{\hG, \bG}(x) = r$.

\item If $\Gamma$ is abelian, $\cl_G(x) = \lceil \cl_{G,N}(x) / 2\rceil$ for $x \in [\hG, \bG]$. Here $\lceil \cdot \rceil$ denotes the ceiling function.
\end{enumerate}
\end{thm}

Combining Theorems~\ref{thm locally cyclic} and \ref{thm not coincidence}, we have the following dichotomy.

\begin{cor} \label{cor:loc_cyc}
Assume that $\Gamma$ is \emph{abelian}.
Then, exactly one of the following two options holds true.
\begin{enumerate}[$(1)$]
\item The group $\Gamma$ is locally cyclic.

\item %There exists a short exact sequence
There exists a pair $(G,N)$ of groups fitting in the short exact sequence
\[ 1 \to N \to G \to \Gamma \to 1\]
%of groups such that $\cl_{G,N}(x) \ne \cl_G(x)$ for some element $x$ of $[G,N]$.
such that $\cl_G$ and $\cl_{G,N}$ do \emph{not} coincide on $[G,N]$.
\end{enumerate}
\end{cor}

%%%%%%%%%%%%%%%%%%%%%%%%%%%%%%%%%%%%%%%%%%%%%%%%%%%%%%%%%%%%%%%%%%%%%%%%%%%%%%%%%%%%%%%%%%%%%%%%%%%%%%%%%%%%%%%%%%%%%%%%%%%%%%%%%
%%%%%%%%%%%%%%%%%%%%%%%%%%%%%%%%%%%%%%%%%%%%%%%%%%%%%%%%%%%%%%%%%%%%%%%%%%%%%%%%%%%%%%%%%%%%%%%%%%%%%%%%%%%%%%%%%%%%%%%%%%%%%%%%%
%%%%%%%%%%%%%%%%%%%%%%%%%%%%%%%%%%%%%%%%%%%%%%%%%%%%%%%%%%%%%%%%%%%%%%%%%%%%%%%%%%%%%%%%%%%%%%%%%%%%%%%%%%%%%%%%%%%%%%%%%%%%%%%%%
%%%%%%%%%%%%%%%%%%%%%%%%%%%%%%%%%%%%%%%%%%%%%%%%%%%%%%%%%%%%%%%%%%%%%%%%%%%%%%%%%%%%%%%%%%%%%%%%%%%%%%%%%%%%%%%%%%%%%%%%%%%%%%%%%

\section{On invariant partial quasimorphisms}\label{partial section}
As is explained in Section \ref{prehistory in symp}, Entov and Polterovich \cite{EP06} implicitly introduced the concept of partial quasimorphism.
However, they did not provide a precise definition of partial quasimorphism and some works provided some definitions of partial quasimorphisms (\cite{FOOO}, \cite{MVZ}, \cite{E}, \cite{Ka16}, \cite{Ka17}, \cite{BK}, \cite{KKMM1} and \cite{K22}).
Here, we provide a general definition containing these definitions.
\begin{definition}\label{qm rt nu}
Let $G$ be a group and $\nu \colon G \to \RR_{\geq0}\cup \{+\infty\}$ a function on $G$.
A function $\phi\colon G\to\mathbb{R}$ is called a \textit{$\nu$-quasimorphism} (\textit{quasimorphism relative to $\nu$} or \textit{quasimorphism controlled by $\nu$}) if there exists a non-negative number $D$ such that for all $f,g \in G$,
%every pair $f$ and $g$ of elements of $G$,
\[|\phi(fg)-\phi(f)-\phi(g)| \leq D\,\min\{\nu(f),\nu(g)\}.\]
%The infimum of such $C$ is called \textit{the defect of} $\phi$ and let $D(\phi)$ denote the defect of $\phi$.
The function $\phi$ is said to be \textit{semi-homogeneous} if $\phi(f^n)=n\phi(f)$ for every $f \in G$ and
%element $f$ of $G$ and
every non-negative integer $n$.
\end{definition}

The word ``$\nu$-quasimorphism'' was introduced in \cite{Ka16} when $\nu$ is a conjugate-invariant norm in the sense of \cite{BIP}.
K{\k{e}}dra \cite{K22} also considers $\nu$-quasimorphism in the same case.\footnote{He used the word ``partial quasimorphism relative to $\nu$.''}
We refer the reader to \cite{CZ} for a related work.

\begin{example}\label{trivial norm}
We define a function $\nu_0\colon G\to \mathbb{R}$ by
\begin{equation*}
\nu_0(g)=
\begin{cases}
0, & \text{if } g=1, \\
1, & \text{otherwise}.
\end{cases}
\end{equation*}
Then, a semi-homogeneous function $\phi$ is a $\nu_0$-quasimorphism if and only if $\phi$ is a quasimorphism.
\end{example}

\begin{example}\label{fragmentation norm}
Let $G$ be a group and $H$ a subgroup of $G$.
We define the fragmentation norm $\nu_H$ with respect to $H$ by for an element $f$ of $G$,
%\[\nu_H(f)=\min\{k;\exists g_1\ldots,g_k\in G, \exists h_1,\ldots h_k\in H\text{ such that }f=g_1h_1g_1^{-1}\cdots g_kh_kg_k^{-1}\}.\]
\[ \nu_H(f)=\min
\left\{ k \; ; \;
\begin{gathered}
  \begin{gathered}\end{gathered}
  \exists g_1\ldots,g_k\in G, \; \exists h_1,\ldots h_k\in H\text{ such that } \\
  f=g_1h_1g_1^{-1}\cdots g_kh_kg_k^{-1}.
\end{gathered}
\right\} \]
If there does not exist such a decomposition of $f$, we set $\nu_H(f)=+\infty$.
If $\nu(f)<+\infty$ for every $f\in G$, $\nu$ is a conjugation-invariant norm in the sense of \cite{BIP}.
\end{example}

When Entov and Polterovich \cite{EP06} introduced
%first gave
the concept of partial quasimorphism, they considered the case when $G$ is (the universal covering of) the group $\Ham(M,\omega)$ of Hamiltonian diffeomorphisms and $H$ a subgroup of $G$ consisting of Hamiltonian diffeomorphisms with some displaceable support.
They proved that the homogenization of Oh--Schwarz's spectral invariant is a $\nu_H$-quasimorphism.
The papers \cite{FOOO}, \cite{MVZ} considered a similar situation and proved that  the homogenization of other spectral invariants are $\nu_H$-quasimorphisms.

In \cite{Ka16} and \cite{Ki} some of the authors constructed a non-trivial $\nu_H$-quasimorphism on $G$ when $(G,H)=(\mathrm{Ham}(\mathbb{R}^{2n},\omega_0),\mathrm{Ham}(\mathbb{B}^{2n},\omega_0))$, and $(G,H)=(B_\infty,B_n)$, respectively.
Here $\omega_0$ is the standard symplectic structure on $\RR^{2n}$.
Brandenbursky and K\k{e}dra \cite{BK} constructed a non-trivial $H$-quasimorphism on $G$ when $G$ is the identity component of the group of volume-preserving diffeomorphisms.

\begin{example}\label{Nqm}
Let $N$ be a normal subgroup.
Then, the norm $\nu_N$ satisfies the following.
\begin{equation*}
\nu_N(g)=
\begin{cases}
0, & \text{if } g = 1, \\
1, & \text{if } g \in N \setminus \{ 1 \}, \\
+\infty, & \text{otherwise}.
\end{cases}
\end{equation*}
Thus, a function $\phi$ is a $\nu_N$-quasimorphism if and only if there exists $D'' \ge 0$ such that for every $g \in G$ and for every $x \in N$, the following inequalities hold:
\[ |\phi(gx) - \phi(g) - \phi(x)| \le D'' \quad \textrm{and} \quad |\phi(xg) - \phi(x) - \phi(g)| \le D''.\]
The smallest non-negative real number $D''$ satisfying the above two inequalities is called the \emph{defect} of $\phi$, and is denoted by $D''(\phi)$.
% We write $\widehat{\QQQ}_N(G)$ to indicate the vector space consisting of $\nu_N$-quasimorphisms on $G$. Then the following hold:

%  if and only if  there exists a non-negative number $D'$ such that
%\[|\phi(\hg x) - \phi(\hg ) - \phi(x)| \le D'\]
%and
%\[|\phi(xg) - \phi(x) - \phi(\hg )| \le D'\]
%for every $g \in \hG$ and every $x \in \bG$.
\end{example}
The authors consider $\nu_N$-quasimorphisms in \cite{KKMM1}, \cite{K2M3}.
The concept of $\nu_N$-quasimorphism is technically useful in the proof of the main theorems in these papers (for examples, see Subsection \ref{Bavard N qm}).

Here, we explain some results on the extension problem of partial quasimorphisms.

\begin{conjecture}[{\cite[Conjecture 1.3]{Ka18}}]\label{partial conj}
Let $(G,H)$ be $(\Ham(\mathbb{R}^{2n},\omega_0),\mathrm{Ham}(\mathbb{B}^{2n},\omega_0))$ or $(B_\infty,B_n)$.
For every semi-homogeneous $H$-quasimorphism $\phi$ on $[G,G]$, there exists a homogeneous $H$-quasimorphism $\hat{\phi}$ on $G$ such that $\hat{\phi}|_{[G,G]}=\phi$.
\end{conjecture}
If Conjecture \ref{partial conj} holds, then every semi-homogeneous $H$-quasimorphism on $[G,G]$ is a $G$-invariant homogeneous $H$-quasimorphism (use \cite[Lemma 4.6]{KK}).

The following theorem supports the above conjecture.
\begin{thm}[{\cite[Theorem 1.4]{Ka18}}]\label{ext partial}
Let $(G,H)$ be $(\Ham(\mathbb{R}^{2n},\omega_0),\Ham(\mathbb{B}^{2n},\omega_0))$ or $(B_\infty,B_n)$.
For every semi-homogeneous $H$-quasimorphism $\phi$ on $[G,G]$ and for every $g \in [G,G]$ such that $\phi(g)\neq0$, there exists a homogeneous $H$-quasimorphism $\hat{\phi}_g$ on $G$ such that $\hat{\phi}_g(g)\neq0$.
\end{thm}

Some of the authors formulated the non-extendability of partial quasimorphisms (Definition 4.5 of \cite{KK}) and proved the following theorem.
\begin{thm}[{\cite[Theorem 1.12]{KK}}]\label{nonext partial}
Let $\Sigma_\genus$ be a closed orientable surface whose genus $\genus$ is at least one and $\omega$ a symplectic form on $\Sigma_\genus$.
Then, the homogenization of Oh--Schwarz's spectral invariant is non-extendable to $\Symp_0(\Sigma_\genus,\omega)$.
\end{thm}
When $\genus=1$, under the notation in Example \ref{flux on torus}, the map $([s],[t]) \mapsto \varphi_{[s],[t]}$ is the section homomorphism of the flux homomorphism $\flux_{\omega_T} \colon \Symp_0(T^2, \omega_T) \to \HHH^1(T^2 ; \RR)/\Gamma_\omega \cong (\RR/\ZZ)^2$.
Thus, Theorem \ref{nonext partial} means that Theorem \ref{virtually split} does not hold for partial quasimorphism.

It is an interesting problem to formulate and prove a Bavard-type duality for invariant partial quasimorphisms.
However, this is a difficult problem for even usual partial quasimorphisms.
One of the authors proved a Bavard-type duality for usual partial quasimorphisms in a very special case \cite{Ka17}.
It might be useful to the problem to use the norm-controlled cohomology introduced by one of the authors \cite{kim_norm_coh}, which is a generalization of bounded cohomology and provides a framework for partial quasimorphisms.

%%%%%%%%%%%%%%%%%%%%%%%%%%%%%%%%%%%%%%%%%%%%%%%%%%%%%%%%%%%%%%%%%%%%%%%%%%%%%%%%%%%%%%%%%%%%%%%%%%%%%%%%%%%%%%%%%%%%%%%%%%%%%%%%%
%%%%%%%%%%%%%%%%%%%%%%%%%%%%%%%%%%%%%%%%%%%%%%%%%%%%%%%%%%%%%%%%%%%%%%%%%%%%%%%%%%%%%%%%%%%%%%%%%%%%%%%%%%%%%%%%%%%%%%%%%%%%%%%%%
%%%%%%%%%%%%%%%%%%%%%%%%%%%%%%%%%%%%%%%%%%%%%%%%%%%%%%%%%%%%%%%%%%%%%%%%%%%%%%%%%%%%%%%%%%%%%%%%%%%%%%%%%%%%%%%%%%%%%%%%%%%%%%%%%
%%%%%%%%%%%%%%%%%%%%%%%%%%%%%%%%%%%%%%%%%%%%%%%%%%%%%%%%%%%%%%%%%%%%%%%%%%%%%%%%%%%%%%%%%%%%%%%%%%%%%%%%%%%%%%%%%%%%%%%%%%%%%%%%%

%%%%%%%%%%%%%%%%%%%%%%%%%%%%%%%%%%%%%%%%%%%%%%%%%%%%%%%%%%%%%%%%%%%%%%%%%%%%%%%%%%%%%%%%%%%%%%%%%%%%%%%%%%%%%%%%%%%%%%%%%%%%%%%%%
%%%%%%%%%%%%%%%%%%%%%%%%%%%%%%%%%%%%%%%%%%%%%%%%%%%%%%%%%%%%%%%%%%%%%%%%%%%%%%%%%%%%%%%%%%%%%%%%%%%%%%%%%%%%%%%%%%%%%%%%%%%%%%%%%
%%%%%%%%%%%%%%%%%%%%%%%%%%%%%%%%%%%%%%%%%%%%%%%%%%%%%%%%%%%%%%%%%%%%%%%%%%%%%%%%%%%%%%%%%%%%%%%%%%%%%%%%%%%%%%%%%%%%%%%%%%%%%%%%%
%%%%%%%%%%%%%%%%%%%%%%%%%%%%%%%%%%%%%%%%%%%%%%%%%%%%%%%%%%%%%%%%%%%%%%%%%%%%%%%%%%%%%%%%%%%%%%%%%%%%%%%%%%%%%%%%%%%%%%%%%%%%%%%%%

\section*{Acknowledgment}
We truly wish to thank the referee for a thorough reading of the submitted version and for very useful comments and suggestions, which have considerably improved this survey.
The second author and the third author are supported by JST-Mirai Program Grant Number JPMJMI22G1 and JSPS KAKENHI Grant Number JP21J11199, respectively.
The first author, the fourth author and the fifth author are partially supported by JSPS KAKENHI Grant Number JP21K13790, JP19K14536 and JP21K03241, respectively.

\bibliography{reference}

\newcommand{\etalchar}[1]{$^{#1}$}
\providecommand{\bysame}{\leavevmode\hbox to3em{\hrulefill}\thinspace}
\providecommand{\MR}{\relax\ifhmode\unskip\space\fi MR }
\providecommand{\MRhref}[2]{%
  \href{http://www.ams.org/mathscinet-getitem?mr=#1}{#2}
}
\providecommand{\href}[2]{#2}
\begin{thebibliography}{KKMM23}

\bibitem[Ab{\'{e}}10]{Abert}
M.~Ab{\'{e}}rt, \emph{Some questions}, 2010,
  http://www.renyi.hu/\textasciitilde abert/questions.pdf.

\bibitem[AL18]{AL}
T.~Akita and Y.~Liu, \emph{Second mod 2 homology of {A}rtin groups}, Algebr.
  Geom. Topol. \textbf{18} (2018), no.~1, 547--568.

\bibitem[And22]{Andritsch}
K.~Andritsch, \emph{Bounded cohomology of groups acting on {C}antor sets},
  preprint, arXiv:2210.00459 (2022).

\bibitem[AV20]{bigMCG}
J.~Aramayona and N.~G. Vlamis, \emph{Big mapping class groups: an overview}, In
  the tradition of {T}hurston---geometry and topology, Springer, Cham, 2020,
  pp.~459--496.

\bibitem[Ban78]{Ban}
A.~Banyaga, \emph{Sur la structure du groupe des diff\'{e}omorphismes qui
  pr\'{e}servent une forme symplectique}, Comment. Math. Helv. \textbf{53}
  (1978), no.~2, 174--227.

\bibitem[Ban97]{Ban97}
\bysame, \emph{The structure of classical diffeomorphism groups}, Mathematics
  and its Applications, vol. 400, Kluwer Academic Publishers Group, Dordrecht,
  1997.

\bibitem[Bav91]{Bav}
C.~Bavard, \emph{Longueur stable des commutateurs}, Enseign. Math. (2)
  \textbf{37} (1991), no.~1-2, 109--150.

\bibitem[Bav16]{JBavard}
J.~Bavard, \emph{Hyperbolicit\'{e} du graphe des rayons et quasi-morphismes sur
  un gros groupe modulaire}, Geom. Topol. \textbf{20} (2016), no.~1, 491--535.

\bibitem[BF02]{BF02}
M.~Bestvina and K.~Fujiwara, \emph{Bounded cohomology of subgroups of mapping
  class groups}, Geom. Topol. \textbf{6} (2002), 69--89.

\bibitem[BF09]{BF09}
\bysame, \emph{A characterization of higher rank symmetric spaces via bounded
  cohomology}, Geom. Funct. Anal. \textbf{19} (2009), no.~1, 11--40.

\bibitem[BG92]{BG}
J.~Barge and \'{E}. Ghys, \emph{Cocycles d'{E}uler et de {M}aslov}, Math. Ann.
  \textbf{294} (1992), no.~2, 235--265.

\bibitem[BIP08]{BIP}
D.~Burago, S.~Ivanov, and L.~Polterovich, \emph{Conjugation-invariant norms on
  groups of geometric origin}, Groups of diffeomorphisms, Adv. Stud. Pure
  Math., vol.~52, Math. Soc. Japan, Tokyo, 2008, pp.~221--250.

\bibitem[BK22]{BK}
M.~Brandenbursky and J.~K{\k{e}}dra, \emph{Fragmentation norm and relative
  quasimorphisms}, Proc. Amer. Math. Soc. \textbf{150} (2022), no.~10,
  4519--4531.

\bibitem[BM97]{BurgerMozes}
M.~Burger and S.~Mozes, \emph{Finitely presented simple groups and products of
  trees}, C. R. Acad. Sci. Paris S\'{e}r. I Math. \textbf{324} (1997), no.~7,
  747--752.

\bibitem[BM19a]{BM}
M.~Brandenbursky and M.~Marcinkowski, \emph{Aut-invariant norms and
  {A}ut-invariant quasimorphisms on free and surface groups}, Comment. Math.
  Helv. \textbf{94} (2019), no.~4, 661--687.

\bibitem[BM19b]{BucherMonod}
M.~Bucher and N.~Monod, \emph{The bounded cohomology of {$\rm SL_2$} over local
  fields and {$S$}-integers}, Int. Math. Res. Not. IMRN (2019), no.~6,
  1601--1611.

\bibitem[Bou95]{MR1338286}
A.~Bouarich, \emph{Suites exactes en cohomologie born\'{e}e r\'{e}elle des
  groupes discrets}, C. R. Acad. Sci. Paris S\'{e}r. I Math. \textbf{320}
  (1995), no.~11, 1355--1359.

\bibitem[Bou01]{MR1896181}
\bysame, \emph{Exactitude \`a gauche du foncteur {$H^n_b(-,\Bbb R)$} de
  cohomologie born\'{e}e r\'{e}elle}, Ann. Fac. Sci. Toulouse Math. (6)
  \textbf{10} (2001), no.~2, 255--270.

\bibitem[Bri07]{Brin}
M.~G. Brin, \emph{The algebra of strand splitting. {I}. {A} braided version of
  {T}hompson's group {$V$}}, J. Group Theory \textbf{10} (2007), no.~6,
  757--788.

\bibitem[Bro82]{brown82}
K.~S. Brown, \emph{Cohomology of groups}, Graduate Texts in Mathematics,
  vol.~87, Springer-Verlag, New York-Berlin, 1982.

\bibitem[Bro92]{Brown92}
\bysame, \emph{The geometry of finitely presented infinite simple groups},
  Algorithms and classification in combinatorial group theory ({B}erkeley,
  {CA}, 1989), Math. Sci. Res. Inst. Publ., vol.~23, Springer, New York, 1992,
  pp.~121--136.

\bibitem[Cal70]{Cala}
E.~Calabi, \emph{On the group of automorphisms of a symplectic manifold},
  Problems in analysis ({L}ectures at the {S}ympos. in honor of {S}alomon
  {B}ochner, {P}rinceton {U}niv., {P}rinceton, {N}.{J}., 1969), 1970,
  pp.~1--26.

\bibitem[Cal09]{Ca}
D.~Calegari, \emph{scl}, MSJ Memoirs, vol.~20, Mathematical Society of Japan,
  Tokyo, 2009.

\bibitem[CZ11]{CZ}
D.~Calegari and D.~Zhuang, \emph{Stable {$W$}-length}, Topology and geometry in
  dimension three, Contemp. Math., vol. 560, Amer. Math. Soc., Providence, RI,
  2011, pp.~145--169.

\bibitem[Deh06]{Dehornoy}
P.~Dehornoy, \emph{The group of parenthesized braids}, Adv. Math. \textbf{205}
  (2006), no.~2, 354--409.

\bibitem[EF97]{EpFu}
D.~B.~A. Epstein and K.~Fujiwara, \emph{The second bounded cohomology of
  word-hyperbolic groups}, Topology \textbf{36} (1997), no.~6, 1275--1289.

\bibitem[EHN81]{MR656217}
D.~Eisenbud, U.~Hirsch, and W.~Neumann, \emph{Transverse foliations of
  {S}eifert bundles and self-homeomorphism of the circle}, Comment. Math. Helv.
  \textbf{56} (1981), no.~4, 638--660.

\bibitem[EK01]{EK}
H.~Endo and D.~Kotschick, \emph{Bounded cohomology and non-uniform perfection
  of mapping class groups}, Invent. Math. \textbf{144} (2001), no.~1, 169--175.

\bibitem[Ent14]{E}
M.~Entov, \emph{Quasi-morphisms and quasi-states in symplectic topology},
  Proceedings of the {I}nternational {C}ongress of {M}athematicians---{S}eoul
  2014. {V}ol. {II}, Kyung Moon Sa, Seoul, 2014, pp.~1147--1171.

\bibitem[EP03]{EP03}
M.~Entov and L.~Polterovich, \emph{Calabi quasimorphism and quantum homology},
  Int. Math. Res. Not. (2003), no.~30, 1635--1676.

\bibitem[EP06]{EP06}
\bysame, \emph{Quasi-states and symplectic intersections}, Comment. Math. Helv.
  \textbf{81} (2006), no.~1, 75--99.

\bibitem[EP09]{EP09}
\bysame, \emph{Rigid subsets of symplectic manifolds}, Compos. Math.
  \textbf{145} (2009), no.~3, 773--826.

\bibitem[FFL21]{FFL21}
F.~Fournier-Facio and Y.~Lodha, \emph{Second bounded cohomology of groups
  acting on 1-manifolds and applications to spectrum problems}, preprint,
  arXiv:2204.05272, (2021).

\bibitem[FFLM21]{FLM1}
F.~Fournier-Facio, C.~L\"{o}h, and M.~Moraschini, \emph{Bounded cohomology of
  finitely presented groups: vanishing, non-vanishing, and computability},
  preprint, arXiv:2106.13567, to appear in Ann. Sc. Norm. Super. Pisa Cl. Sci
  (2021).

\bibitem[FFLM22]{FLM2}
\bysame, \emph{Bounded cohomology and binate groups}, J. Aust. Math. Soc,
  online published, https://doi.org/10.1017/S1446788722000106 (2022).

\bibitem[FFLZ22]{FFLZ}
F.~Fournier-Facio, Y.~Lodha, and M.~C.~B. Zaremsky, \emph{Braided thompson
  groups with and without quasimorphisms}, preprint, arXiv:2204.05272, to
  appear in Algebr. Geom. Topol. (2022).

\bibitem[FFW22]{FW}
F.~Fournier-Facio and R.~D. Wade, \emph{Aut-invariant quasimorphisms on
  groups}, \emph{to appear in} Trans. Amer. Math. Soc.\emph{,
  arXiv:2211.00800v1} (2022).

\bibitem[Flo88]{Fl88}
A.~Floer, \emph{Morse theory for {L}agrangian intersections}, J. Differential
  Geom. \textbf{28} (1988), no.~3, 513--547.

\bibitem[FM12]{farb_margalit12}
B.~Farb and D.~Margalit, \emph{A primer on mapping class groups}, Princeton
  Mathematical Series, vol.~49, Princeton University Press, Princeton, NJ,
  2012.

\bibitem[FOOO19]{FOOO}
K.~Fukaya, Y.-G. Oh, H.~Ohta, and K.~Ono, \emph{Spectral invariants with bulk,
  quasi-morphisms and {L}agrangian {F}loer theory}, Mem. Amer. Math. Soc.
  \textbf{260} (2019), no.~1254, x+266.

\bibitem[Fri17]{Fr}
R.~Frigerio, \emph{Bounded cohomology of discrete groups}, Mathematical Surveys
  and Monographs, vol. 227, American Mathematical Society, Providence, RI,
  2017.

\bibitem[GG04]{GG}
J.-M. Gambaudo and \'{E}. Ghys, \emph{Commutators and diffeomorphisms of
  surfaces}, Ergodic Theory Dynam. Systems \textbf{24} (2004), no.~5,
  1591--1617.

\bibitem[GH21]{GenevoisHorbez}
A.~Genevois and C.~Horbez, \emph{Acylindrical hyperbolicity of automorphism
  groups of infinitely ended groups}, J. Topol. \textbf{14} (2021), no.~3,
  963--991.

\bibitem[Ghy01]{MR1876932}
\'{E}. Ghys, \emph{Groups acting on the circle}, Enseign. Math. (2) \textbf{47}
  (2001), no.~3-4, 329--407.

\bibitem[Gro82]{Gr}
M.~Gromov, \emph{Volume and bounded cohomology}, Inst. Hautes \'{E}tudes Sci.
  Publ. Math. (1982), no.~56, 5--99 (1983).

\bibitem[Gro85]{Gr85}
\bysame, \emph{Pseudo holomorphic curves in symplectic manifolds}, Invent.
  Math. \textbf{82} (1985), no.~2, 307--347.

\bibitem[Gro87]{Gromov}
\bysame, \emph{Hyperbolic groups}, Essays in group theory, Math. Sci. Res.
  Inst. Publ., vol.~8, Springer, New York, 1987, pp.~75--263.

\bibitem[Has18]{Hase}
A.~Hase, \emph{Dynamics of $\mathrm{Out}({F}_n)$ on the second bounded
  cohomology of ${F}_n$}, preprint, arXiv:1805.00366 (2018).

\bibitem[Hat02]{Hatcher}
A.~Hatcher, \emph{Algebraic topology}, Cambridge university press, 2002.

\bibitem[HL19]{HL19}
N.~Heuer and C.~L\"oh, \emph{Transcendental simplicial volumes},
  arXiv:1911.06386, to appear in Annales de l'Institut Fourier (2019).

\bibitem[HL21a]{HL21a}
\bysame, \emph{The spectrum of simplicial volume of non-compact manifolds},
  Inventiones mathematicae \textbf{223} (2021), no.~1, 103--148.

\bibitem[HL21b]{HL21b}
\bysame, \emph{The spectrum of simplicial volume of non-compact manifolds},
  Geometriae Dedicata \textbf{215} (2021), 243--253.

\bibitem[Hum11]{Hum}
V.~Humili{\`e}re, \emph{The {C}alabi invariant for some groups of
  homeomorphisms}, J. Symplectic Geom. \textbf{9} (2011), no.~1, 107--117.

\bibitem[IMT19]{IMT19}
T.~Ito, K.~Motegi, and M.~Teragaito, \emph{Generalized torsion and
  decomposition of 3-manifolds}, Proc. Amer. Math. Soc. \textbf{147} (2019),
  no.~11, 4999--5008.

\bibitem[Ish13]{IshidaThesis}
T.~Ishida, \emph{Quasi-morphisms on the group of area-preserving
  diffeomorphisms of the 2-disk}, PhD Thesis at the University of Tokyo (2013).

\bibitem[Ish14]{Ish}
\bysame, \emph{Quasi-morphisms on the group of area-preserving diffeomorphisms
  of the 2-disk via braid groups}, Proc. Amer. Math. Soc. Ser. B \textbf{1}
  (2014), 43--51.

\bibitem[Kar21a]{Karlh}
B.~Karlhofer, \emph{Aut-invariant quasimorphisms on free products}, Ann. Math.
  Qu\'{e}bec, online published, https://doi.org/10.1007/s40316-021-00184-4
  (2021).

\bibitem[Kar21b]{Karlh2}
\bysame, \emph{Aut-invariant quasimorphisms on graph products of abelian
  groups}, preprint, arXiv:2107.12171 (2021).

\bibitem[Kaw16]{Ka16}
M.~Kawasaki, \emph{Relative quasimorphisms and stably unbounded norms on the
  group of symplectomorphisms of the {E}uclidean spaces}, J. Symplectic Geom.
  \textbf{14} (2016), no.~1, 297--304.

\bibitem[Kaw17]{Ka17}
\bysame, \emph{Bavard's duality theorem on conjugation-invariant norms},
  Pacific J. Math. \textbf{288} (2017), no.~1, 157--170.

\bibitem[Kaw18]{Ka18}
\bysame, \emph{Extension problem of subset-controlled quasimorphisms}, Proc.
  Amer. Math. Soc. Ser. B \textbf{5} (2018), 1--5.

\bibitem[Kaw19]{Ka19}
\bysame, \emph{Superheavy {L}agrangian immersions in surfaces}, J. Symplectic
  Geom. \textbf{17} (2019), no.~1, 239--249.

\bibitem[K{\k{e}}d00]{Ke}
J.~K{\k{e}}dra, \emph{Remarks on the flux groups}, Math. Res. Lett. \textbf{7}
  (2000), no.~2-3, 279--285.

\bibitem[K{\k{e}}d22]{K22}
\bysame, \emph{On {L}ipschitz functions on groups equipped with
  conjugation-invariant norms}, preprint, arXiv:2204.09373v1 (2022).

\bibitem[Kim18]{Ki}
M.~Kimura, \emph{Conjugation-invariant norms on the commutator subgroup of the
  infinite braid group}, J. Topol. Anal. \textbf{10} (2018), no.~2, 471--476.

\bibitem[Kim23]{kim_norm_coh}
\bysame, \emph{Norm-controlled cohomology of transformation groups}, Int. J.
  Math., online published, https://doi.org/10.1142/S0129167X23500222 (2023).

\bibitem[KK22]{KK}
M.~Kawasaki and M.~Kimura, \emph{{$\hat G$}-invariant quasimorphisms and
  symplectic geometry of surfaces}, Israel J. Math. \textbf{247} (2022), no.~2,
  845--871.

\bibitem[KKM{\etalchar{+}}21]{K2M3}
M.~Kawasaki, M.~Kimura, S.~Maruyama, T.~Matsushita, and M.~Mimura, \emph{The
  space of non-extendable quasimorphisms}, preprint, arXiv:2107.08571v4 (2021).

\bibitem[KKM{\etalchar{+}}22]{K2M32}
\bysame, \emph{Mixed commutator lengths, wreath products and general ranks},
  preprint, arXiv:2203.04048v2, to appear in Kodai Math. J. (2022).

\bibitem[KKMM22]{KKMM1}
M.~Kawasaki, M.~Kimura, T.~Matsushita, and M.~Mimura, \emph{Bavard's duality
  theorem for mixed commutator length}, Enseign. Math. \textbf{68} (2022),
  no.~3-4, 441--481.

\bibitem[KKMM23]{KKMM2}
Morimichi Kawasaki, Mitsuaki Kimura, Takahiro Matsushita, and Masato Mimura,
  \emph{Commuting symplectomorphisms on a surface and the flux homomorphism},
  Geom. Funct. Anal.\emph{, online published} (2023).

\bibitem[KO21]{KO19}
M.~Kawasaki and R.~Orita, \emph{Disjoint superheavy subsets and fragmentation
  norms}, J. Topol. Anal. \textbf{13} (2021), no.~2, 443--468.

\bibitem[Kor02]{MR1892804}
M.~Korkmaz, \emph{Low-dimensional homology groups of mapping class groups: a
  survey}, Turkish J. Math. \textbf{26} (2002), no.~1, 101--114.

\bibitem[LMP98]{LMP}
F.~Lalonde, D.~McDuff, and L.~Polterovich, \emph{On the flux conjectures},
  Geometry, topology, and dynamics ({M}ontreal, {PQ}, 1995), CRM Proc. Lecture
  Notes, vol.~15, Amer. Math. Soc., Providence, RI, 1998, pp.~69--85.

\bibitem[L{\"o}h23]{Loh23}
C.~L{\"o}h, \emph{The spectrum of simplicial volume with fixed fundamental
  group}, Geometriae Dedicata \textbf{217} (2023), no.~2.

\bibitem[Mal48]{Malcev}
A.~I. Mal'tsev, \emph{On groups of finite rank}, Mat. Sbornik N.S.
  \textbf{22(64)} (1948), 351--352.

\bibitem[Mal09]{Mal}
A.~V. Malyutin, \emph{Operators in the spaces of pseudocharacters of braid
  groups}, Algebra i Analiz \textbf{21} (2009), no.~2, 136--165.

\bibitem[Mat86]{MR848896}
S.~Matsumoto, \emph{Numerical invariants for semiconjugacy of homeomorphisms of
  the circle}, Proc. Amer. Math. Soc. \textbf{98} (1986), no.~1, 163--168.

\bibitem[Min02]{Mineyev}
I.~Mineyev, \emph{Bounded cohomology characterizes hyperbolic groups}, Q. J.
  Math. \textbf{53} (2002), no.~1, 59--73.

\bibitem[MM85]{MatsumotoMorita}
S.~Matsumoto and S.~Morita, \emph{Bounded cohomology of certain groups of
  homeomorphisms}, Proc. Amer. Math. Soc. \textbf{94} (1985), no.~3, 539--544.

\bibitem[MMM22]{MMM}
Shuhei Maruyama, Takahiro Matsushita, and Masato Mimura, \emph{Invariant
  quasimorphisms for groups acting on the circle and non-equivalence of {SCL}},
  \emph{to appear in} Israel J. Math.\emph{, arXiv:2203.09221v3} (2022).

\bibitem[MN23]{MN21}
N.~Monod and S.~Nariman, \emph{Bounded and unbounded cohomology of
  homeomorphism and diffeomorphism groups}, Invent. Math., online published,
  https://doi.org/10.1007/s00222-023-01181-w (2023).

\bibitem[Mon01]{Mon01}
N.~Monod, \emph{Continuous bounded cohomology of locally compact groups},
  Springer, 2001.

\bibitem[Mon04]{Mon04}
\bysame, \emph{Stabilization for {${\rm SL}_n$} in bounded cohomology},
  Discrete geometric analysis, Contemp. Math., vol. 347, Amer. Math. Soc.,
  Providence, RI, 2004, pp.~191--202.

\bibitem[Mon22]{Monod2021}
\bysame, \emph{Lamplighters and the bounded cohomology of {T}hompson's group},
  Geometry and Functional Analysis \textbf{32} (2022), no.~3, 662--675.

\bibitem[MP19]{MP19}
M.~Magee and D.~Puder, \emph{Matrix group integrals, surfaces, and mapping
  class groups i: U(n)}, Inventiones mathematicae \textbf{218} (2019),
  341--411.

\bibitem[MP20]{MP20}
D.~Margalit and A.~Putman, \emph{Surface groups, infinite generating sets, and
  stable commutator length}, Proc. Roy. Soc. Edinburgh Sect. A \textbf{150}
  (2020), no.~5, 2379--2386.

\bibitem[MR21]{MR21}
M.~Moraschini and G.~Raptis, \emph{Amenability and acyclicity in bounded
  cohomology theory}, preprint, arXiv:2105.02821, to appear in Rev. Mat.
  Iberoam (2021).

\bibitem[MS04]{MonodShalom}
N.~Monod and Y.~Shalom, \emph{Cocycle superrigidity and bounded cohomology for
  negatively curved spaces}, J. Differential Geom. \textbf{67} (2004), no.~3,
  395--455.

\bibitem[MS17]{MS}
D.~McDuff and D.~Salamon, \emph{Introduction to symplectic topology}, third
  ed., Oxford Graduate Texts in Mathematics, Oxford University Press, Oxford,
  2017.

\bibitem[MVZ12]{MVZ}
A.~Monzner, N.~Vichery, and F.~Zapolsky, \emph{Partial quasimorphisms and
  quasistates on cotangent bundles, and symplectic homogenization}, J. Mod.
  Dyn. \textbf{6} (2012), no.~2, 205--249.

\bibitem[Nie17]{Nielsen}
J.~Nielsen, \emph{Die {I}somorphismen der allgemeinen, unendlichen {G}ruppe mit
  zwei {E}rzeugenden}, Math. Ann. \textbf{78} (1917), no.~1, 385--397.

\bibitem[Oh05]{O05}
Y.-G. Oh, \emph{Construction of spectral invariants of {H}amiltonian paths on
  closed symplectic manifolds}, The breadth of symplectic and {P}oisson
  geometry, Progr. Math., vol. 232, Birkh\"{a}user Boston, Boston, MA, 2005,
  pp.~525--570.

\bibitem[Oh09]{O09}
\bysame, \emph{Floer mini-max theory, the {C}erf diagram, and the spectral
  invariants}, J. Korean Math. Soc. \textbf{46} (2009), no.~2, 363--447.

\bibitem[Oh15]{O15}
\bysame, \emph{Symplectic topology and {F}loer homology. {V}ol. 2}, New
  Mathematical Monographs, vol.~29, Cambridge University Press, Cambridge,
  2015, Floer homology and its applications.

\bibitem[Ono06]{O}
K.~Ono, \emph{Floer-{N}ovikov cohomology and the flux conjecture}, Geom. Funct.
  Anal. \textbf{16} (2006), no.~5, 981--1020.

\bibitem[Osi15]{Osin15}
D.~Osin, \emph{On acylindrical hyperbolicity of groups with positive first
  {$\ell^2$}-{B}etti number}, Bull. Lond. Math. Soc. \textbf{47} (2015), no.~5,
  725--730.

\bibitem[Osi16]{Osin}
\bysame, \emph{Acylindrically hyperbolic groups}, Trans. Amer. Math. Soc.
  \textbf{368} (2016), no.~2, 851--888.

\bibitem[Poi82]{poincare81}
H.~Poincar\'{e}, \emph{M\'{e}moire sur les courbes d\'{e}finies par une
  \'{e}quation diff\'{e}rentielle}, Journal de Math\'{e}matiques {\bf 7} (1881)
  375-422 et {\bf 8} (1882), 251--296.

\bibitem[Pol01]{P01}
L.~Polterovich, \emph{The geometry of the group of symplectic diffeomorphisms},
  Lectures in Mathematics ETH Z\"{u}rich, Birkh\"{a}user Verlag, Basel, 2001.

\bibitem[PR14]{PR}
L.~Polterovich and D.~Rosen, \emph{Function theory on symplectic manifolds},
  CRM Monograph Series, vol.~34, American Mathematical Society, Providence, RI,
  2014.

\bibitem[PSS96]{PSS}
S.~Piunikhin, D.~Salamon, and M.~Schwarz, \emph{Symplectic {F}loer-{D}onaldson
  theory and quantum cohomology}, Contact and symplectic geometry ({C}ambridge,
  1994), Publ. Newton Inst., vol.~8, Cambridge Univ. Press, Cambridge, 1996,
  pp.~171--200.

\bibitem[Py06a]{Py06t}
P.~Py, \emph{Quasi-morphismes de {C}alabi et graphe de {R}eeb sur le tore}, C.
  R. Math. Acad. Sci. Paris \textbf{343} (2006), no.~5, 323--328.

\bibitem[Py06b]{Py06}
\bysame, \emph{Quasi-morphismes et invariant de {C}alabi}, Ann. Sci. \'{E}cole
  Norm. Sup. (4) \textbf{39} (2006), no.~1, 177--195.

\bibitem[Sht16]{Sh}
A.~I. Shtern, \emph{Extension of pseudocharacters from normal subgroups,
  {III}}, Proc. Jangjeon Math. Soc. \textbf{19} (2016), no.~4, 609--614.

\bibitem[ST94]{SerTsu}
V.~Sergiescu and T.~Tsuboi, \emph{A remark on homeomorphisms of the {C}antor
  set}, Geometric study of foliations ({T}okyo, 1993), World Sci. Publ., River
  Edge, NJ, 1994, pp.~431--436.

\bibitem[SW19]{SzWa}
M.~Szymik and N.~Wahl, \emph{The homology of the {H}igman-{T}hompson groups},
  Invent. Math. \textbf{216} (2019), no.~2, 445--518.

\bibitem[Thu97]{math/9712268}
W.~P. Thurston, \emph{Three-manifolds, foliations and circles, {I}},
  arXiv:math/9712268 (1997).

\bibitem[Tsu12]{Tsuboi12}
T.~Tsuboi, \emph{On the uniform perfectness of the groups of diffeomorphisms of
  even-dimensional manifolds}, Comment. Math. Helv. \textbf{87} (2012), no.~1,
  141--185.

\bibitem[Tsu13]{Tsuboi13}
\bysame, \emph{Homeomorphism groups of commutator width one}, Proc. of Amer.
  Math. Soc. \textbf{141} (2013), 1839--1847.

\bibitem[Tsu17]{Tsuboi17}
\bysame, \emph{Several problems on groups of diffeomorphisms}, Geometry,
  dynamics, and foliations 2013, Adv. Stud. Pure Math., vol.~72, Math. Soc.
  Japan, Tokyo, 2017, pp.~239--248.

\bibitem[Zar18]{Zaremsky}
M.~C.~B. Zaremsky, \emph{A user's guide to cloning systems}, Topology Proc.
  \textbf{52} (2018), 13--33.

\end{thebibliography}
\bibliographystyle{amsalpha}

\end{document}